\documentclass[reqno,11pt]{amsart}
\usepackage{relsize}
\usepackage{scalerel}
\usepackage[margin=1in]{geometry}
\usepackage{stackengine,wasysym}
\usepackage{todonotes}
\usepackage{xcolor}
\usepackage{mathrsfs}
\usepackage{dsfont}
\usepackage{mathtools}
\usepackage[hyperfootnotes=false,colorlinks=true,linkcolor = blue, urlcolor  = blue, citecolor = blue]{hyperref}
\usepackage[sort,nocompress]{cite}
\usepackage{float}
\usepackage{amsmath, amsthm, amssymb}
\usepackage{times}
\usepackage{color}
\usepackage{comment}
\usepackage[toc,page]{appendix}
\usepackage{verbatim}
\usepackage{enumerate}
\usepackage{cite}

\allowdisplaybreaks

\newcommand{\pa}{\partial}
\newcommand{\la}{\label}
\newcommand{\fr}{\frac}
\newcommand{\na}{\nabla}
\newcommand{\be}{\begin{equation}}
\newcommand{\ee}{\end{equation}}
\newcommand{\ba}{\begin{array}{l}}
\newcommand{\ea}{\end{array}}

\newcommand{\beg}{\begin}

\renewcommand{\l}{\Lambda}
\renewcommand{\j}{\mathcal J_\eta}
\newcommand{\jk}{\mathcal J_{\eta_k}}

\usepackage{MnSymbol,wasysym}
\newcommand{\N}{\mathbb N}

\newcommand{\R}{\mathbb R}

\def\TT{{\mathbb T}}
\def\NN{\mathbb N}

\theoremstyle{plain}

\newtheorem{Thm}{Theorem}[section]

\newtheorem{lem}[Thm]{Lemma}
\newtheorem{prop}[Thm]{Proposition}

\theoremstyle{definition}

\theoremstyle{remark}
\newtheorem{rem}{Remark}

\numberwithin{equation}{section}
\title{On the Non-isothermal Nernst-Planck-Navier-Stokes System}

\author[E. Abdo]{Elie Abdo}
\address[E. Abdo]
{	Department of Mathematics \\
     American University of Beirut \\
	Beirut 1107-2020\\Lebanon.} \email{ea94@aub.edu.lb}

\author[Q. Lin]{Quyuan Lin}
\address[Q. Lin]
{	School of Mathematical and Statistical Sciences \\
Clemson University\\
Clemson, SC 29634, USA.} \email{quyuanl@clemson.edu}
\date{\today}

\begin{document}

\begin{abstract}

Electrodiffusion has been extensively studied in the isothermal setting, whereas the mathematical theory of thermally coupled electrodiffusion remains comparatively underdeveloped. We investigate a non-isothermal electrodiffusion model describing the evolution of multiple ionic species with different diffusivities and valences in a two-dimensional incompressible viscous fluid. The coupling to a spatially and temporally varying temperature gives rise to a nonlinear and nonlocal system with logarithmic nonlinearities in the ionic fluxes. We establish local well-posedness for strictly positive initial concentrations and prove global well-posedness when the initial temperature is close to a homogeneous state  by developing a new entropy structure tailored to thermodiffusive
effects. No smallness assumption is imposed on the initial ionic concentrations or fluid velocity. To overcome the singularity of the logarithmic terms, we develop a novel cutoff–mollification regularization, derive uniform logarithmic estimates, and prove persistence of strict positivity of the ionic concentrations. This positivity removes the singular behavior of the logarithmic nonlinearities and is essential for the uniqueness argument. These results
provide a rigorous mathematical framework for the analysis of non-isothermal electrohydrodynamics systems.

\end{abstract}

\maketitle

\vspace{0.5cm}


\vspace{0.5cm}

\section{Introduction}
Electrodiffusion describes the transport of charged particles under the combined effects of concentration gradients and electric fields. 
It arises in a wide range of applications, including neuroscience \cite{jasielec2021electrodiffusion,mori2009numerical,lopreore2008computational,koch2004biophysics,savtchenko2017electrodiffusion}, semiconductor theory \cite{biler1994debye,gajewski1986basic,mock1983analysis}, water purification, desalination, ion separation \cite{alkhad2022electrochemical,zhu2020ion,lee2018diffusiophoretic,yang2019review,gao2014high,lee2016membrane}, and battery technology \cite{tan2016computational}. 

Temperature variations play an important role in many applications of electrodiffusion models. One particularly important example is the control of dendrite growth near nucleation sites, which is closely related to the performance, safety, and lifetime of high-energy-density batteries \cite{tan2016computational}. During charging and discharging cycles, the temperature distribution inside a battery evolves and may influence ionic transport, electrochemical efficiency, and hydrodynamic stability. Classical electrodiffusion models are often studied in the isothermal setting, which removes thermal coupling and thermodiffusive effects. In order to obtain a more accurate description of transport phenomena in nonuniform thermal environments, it is natural to develop and analyze non-isothermal models that incorporate temperature-dependent effects. %In addition, buoyancy forces arising from nonhomogeneous temperature and salinity distributions can significantly affect the fluid motion and stability of the system, especially in contexts like groundwater and seawater \cite{yang2010numerically,gray1976validity,mayeli2021buoyancy,petcu2009some}. These forces are crucial to the Electrodiffusion process but are neglected in existing Electrodiffusion models. 
To push the boundaries of precision and realism, it is important to develop and analyze more comprehensive models that account for these key factors from thermodynamics.

\smallskip

\noindent
{\bf Non-isothermal Poisson--Nernst--Planck system.} 
Mathematically, for a system of $N$ ionic species with concentrations $c_i=c_i(x,t)$, the evolution of each species is governed by the conservation law
\begin{equation}\label{eqn:NP}
    \partial_t c_i + \nabla\cdot J_i = 0, \quad i=1,\dots, N,
\end{equation}
where $J_i$ denotes the ionic flux. Assuming linear Onsager kinetics, the ionic flux is postulated as 
\[
    J_i = -b_i c_i \nabla\mu_i,
\]
where $b_i=\frac{D_i}{k_BT}$ is the mobility of the $i$-th species by Einstein's relation, $D_i>0$ is its diffusivity, $k_B$ is Boltzmann's constant, $T=T(x,t)>0$ is the absolute temperature, and $\mu_i$ is the corresponding chemical potential. 
Denoting by $\phi=\phi(x,t)$ the electrostatic potential, by $e$ the elementary charge, and by $\varepsilon>0$ the dielectric permittivity of the solvent, the Helmholtz free energy functional is given by
\begin{equation*}
\mathcal F[c_i,\phi,T]
=
\int
\left(
\sum_{i=1}^N k_B T\, (c_i\log c_i - c_i+1)
+
\sum_{i=1}^N z_i e\, c_i \phi
+
\frac{\varepsilon}{2}|\nabla\phi|^2
\right)\,dx,
\end{equation*}
where the first term represents the entropic contribution, the second term accounts
for electrostatic energy, and the last term is the energy of the electric field. The chemical potential $\mu_i$ is then obtained by variational differentiation
with respect to $c_i$ \cite{kilic2007steric1,kilic2007steric2}, yielding
\begin{equation*}
\mu_i
=
\frac{\delta \mathcal F}{\delta c_i}
=
k_B T \log c_i + z_i e \phi.
\end{equation*}
Here $z_i$ is the valence of the $i$-th species.
Substituting the expression for
$\nabla\mu_i$ in $J_i$ yields
\begin{equation}\label{eqn:flux}
J_i
=
- D_i
\left(
\nabla c_i
+
\frac{z_i e}{k_B T} c_i \nabla \phi
+
c_i \log c_i \nabla \log T
\right),
\end{equation}
which consists of diffusion, electrostatic drift (electromigration), and a thermodiffusive contribution induced by temperature gradients. 
The electrostatic potential $\phi$ is determined self-consistently by the charge density $\rho$ through Gauss's Law:
\begin{equation}\label{eqn:poisson}
    -\varepsilon \Delta \phi = \rho \text{ with } \rho =  e\sum\limits_{i=1}^N z_i c_i.
\end{equation}
The temperature evolution is described by a heat equation
\begin{equation}\label{eqn:temperature}
   \partial_t T -\kappa  \Delta T=0, 
\end{equation}
where $\kappa>0$ is the thermal conductivity. The resulting coupled system \eqref{eqn:NP}--\eqref{eqn:temperature} will be referred to as the non-isothermal Poisson--Nernst--Planck (PNP) system. When the temperature is constant in both space and time, the thermodiffusion term vanishes, and the system reduces to the classical isothermal PNP system.

\smallskip

\noindent
{\bf Non-isothermal Nernst--Planck--Navier--Stokes system.} 
When the ionic solution is transported by an incompressible viscous fluid, the non-isothermal PNP system is coupled to the Navier--Stokes equations. This leads to the non-isothermal Nernst--Planck--Navier--Stokes (NPNS) system
\begin{subequations}\label{N-NPNS-system-original}
\begin{align}
&\partial_t c_i + u\cdot\nabla c_i
=
D_i\nabla \cdot
\left(
 \nabla c_i
+
 \frac{z_i e}{k_B T} c_i \nabla \phi
+
c_i \log c_i \nabla \log T
\right), \quad
 i=1,\dots,N, \label{N-NPNS-system-original-1}
 \\
&- \varepsilon \Delta \phi= \rho
=
\sum_{i=1}^N z_i e\, c_i, \label{N-NPNS-system-original-2}
\\
&\partial_t u + u\cdot\nabla u + \nabla p - \nu \Delta u = -\rho\nabla \phi + g\alpha_T(T-T_r) \vec{k}, \label{N-NPNS-system-original-3}
\\
&\nabla \cdot u = 0, \label{N-NPNS-system-original-4}
\\
&\partial_t T + u\cdot\nabla T- \kappa \Delta T
= 0, \label{N-NPNS-system-original-5}
\end{align}
\end{subequations}
where $u$ is the divergence-free velocity field, $p$ is the pressure, $\nu>0$ is the kinematic viscosity, $g$ is the gravity acceleration, $\alpha_T>0$ is the thermal expansion coefficient, $T_r$ is a reference temperature, and $\vec{k}$ is the vertical unit vector. The term $g\alpha_T (T-T_r)\vec{k}$ represents the buoyancy force due to the nonhomogeneous temperature. When the temperature $T \equiv T_r$ is constant, the buoyancy and thermodiffusion terms vanish, and the system reduces to the classical isothermal NPNS system.

\smallskip

\noindent
{\bf Literature review.} The NPNS system has been extensively studied in the isothermal setting. In three dimensions, the existence of global weak solutions under blocking boundary conditions was established in \cite{fischer2017global,jerome2009global}. Global existence for small initial data and external forces was proved in \cite{ryham2009existence,schmuck2009analysis}. For the system with Robin boundary condition for the electric potential and blocking boundary condition for the ionic concentrations, global existence and stability of solutions in two dimensions were obtained in \cite{bothe2014global}. In \cite{constantin2019nernst}, the authors investigated the two-dimensional case under various boundary conditions for $c_i$ and proved the existence and uniqueness of global strong solutions, as well as their convergence to Boltzmann steady states. The existence of strong solutions in three dimensions was later established in \cite{constantin2021nernst} for the case of two ionic species and for systems with multiple ionic species having equal diffusivities. Regarding the stability of Boltzmann states, \cite{constantin2022nernst} proved nonlinear stability in both two- and three-dimensional bounded domains under suitable boundary conditions. By contrast, instabilities in simplified models have been studied both analytically and numerically in \cite{rubinstein2000electro,zaltzman2007electro}, while physical observations of such instabilities under selective boundary conditions were reported in \cite{rubinstein2008direct}. In the periodic setting, the long-time dynamics of the NPNS system were analyzed in \cite{abdo2021long,abdo2024long2}, where exponential decay of solutions was proved. Electrodiffusion phenomena have also been studied in other fluid models, including ideal fluids, leading to Nernst--Planck--Euler systems \cite{ignatova2021global,abdo2023global,shen2022stability,zhang2015global,zhang2020inviscid}, and porous media, leading to Nernst--Planck--Darcy systems \cite{ignatova2022global,abdo2023global, abdo2025long, abdo2026long}.

In the non-isothermal setting, the authors in \cite{abdo2024three} studied the Nernst-Planck-Boussinesq (NPB) system and established the global existence of weak solutions and addressed their long-time dynamics in three dimensions. The NPB system is closely related to the non-isothermal NPNS system \eqref{N-NPNS-system-original} introduced in this paper. The main difference is that the logarithmic thermodiffusion term
$c_i \log c_i \nabla \log T$ appearing in the ionic flux is omitted in the NPB system. Such an omission may be justified in regimes where temperature gradients are sufficiently small and Soret-type effects have a minor influence on ionic transport. In regimes where thermal gradients have a non-negligible effect, however, the thermodiffusion term should be retained. From the analytical viewpoint, this term introduces substantial difficulties, since it couples the ionic concentration to the temperature gradient through the nonlinear factor $c_i \log c_i$ and the derivative $\nabla \log T$.

\smallskip
\noindent
{\bf Main results.} We consider system \eqref{N-NPNS-system-original} with initial condition $(c_i(0), u(0), T(0)) = (c_{i0}, u_0, T_0)$ in a two-dimensional torus $\TT^2$ with unit volume. In the periodic case, the reference temperature $T_r = \int_{\TT^2} T_0 dx$ is the initial average, and the initial ionic concentrations satisfy the compatibility condition
\begin{equation}\label{eqn:compa}
    \int_{\TT^2} \sum_{i=1}^N z_i c_{i0} = 0.
\end{equation}

We denote by $H$ and $V$ the $L^2$ and $H^1$ closures of the space of divergence-free smooth vector fields. Let $A=-\mathbb P\Delta$ be the Stokes operator, where $\mathbb P$ is the Leray projection, and denote by $\mathcal D(A) = V\cap H^2$. 

Our first result is the local well-posedness of system \eqref{N-NPNS-system-original}. 

\begin{Thm}\label{thm:lwp-intro}
    Let $c_{i0} \in H^1, u_0 \in V, T_0 \in H^{\frac{5}{2}}$, with $T_0\geq T^*>0$ and $c_{i0}\geq a_0>0$. Suppose the initial concentrations $c_{i0}$ satisfy the compatibility condition \eqref{eqn:compa} and the initial velocity $u_0$ has zero mean. There exists a time $\mathcal{T}_0$ such that the non-isothermal NPNS system has a unique local solution on $[0, \mathcal{T}_0]$ obeying 
    \begin{equation}\label{lwp-regularity}
        \begin{split}
        &c_i \in L^\infty(0,\mathcal{T}_0; H^1)\cap L^2(0,\mathcal{T}_0; H^2),
        \\
        &u \in L^\infty(0,\mathcal{T}_0; V)\cap L^2(0,\mathcal{T}_0; \mathcal D(A)),
        \\
        &T \in L^\infty(0,\mathcal{T}_0; H^{\frac52})\cap L^2(0,\mathcal{T}_0; H^{\frac72}).
        \end{split}
    \end{equation}
    Moreover, $u$ has zero mean, $T\geq T^*$, and $c_i\geq a$ for some constant $a>0$.
\end{Thm}

Our second result extends the local solutions to global ones, provided that the initial temperature is close to a homogeneous state in the $L^\infty$ norm.
\begin{Thm}\label{thm:gwp-intro}
Let $(u, c_i, T)$ be a smooth local solution of the non-isothermal NPNS system on some time interval $[0, \mathcal{T}]$. If $\|T_0 - T_r\|_{L^\infty}$ is sufficiently small and satisfies conditions \eqref{small-1} and \eqref{small-2},
then for any $t\in[0,\mathcal T]$, it holds that
\be 
\sum\limits_{i=1}^{N}\|c_i(t)\|_{H^1} + \|u(t)\|_{H^1} + \|T(t)-T_r\|_{H^{\frac{5}{2}}}  \le \Gamma_0.
\ee  
Here $\Gamma_0$ depends only on the initial conditions and the parameters of the system.
In particular, this allows the extension of local solutions to global ones. 
\end{Thm}

We emphasize that the smallness condition is only on $\|T_0 - T_r\|_{L^\infty}$, which holds when the initial temperature distribution is close to equilibrium. Without such an assumption, the thermodiffusive effect can dominate the dynamics, and the global well-posedness is unclear.

\smallskip
\noindent
{\bf Mathematical challenges and strategies.} One of the principal difficulties in the analysis results from the presence of the logarithmic drift term $c_i\log c_i$ in the Nernst--Planck equations. Unlike polynomial nonlinearities, the logarithm exhibits a low-order singular behavior near zero, which creates substantial obstacles in both the local and global well-posedness theories. In particular, standard energy methods do not directly provide sufficient control to handle the logarithmic contributions, and uniqueness becomes unclear when the ionic concentrations are merely assumed to be nonnegative. These difficulties are further amplified by the strong coupling between the ionic concentrations, the fluid velocity, and the temperature.

A second major challenge concerns the global well-posedness of the system for an arbitrary number of ionic species with different diffusivities and valences. This problem remained unresolved in previous works. The main reason is that the entropy structure traditionally exploited in the analysis of Nernst--Planck systems is significantly altered by the presence of a variable temperature. In the isothermal setting, several classical cancellation mechanisms allow one to derive entropy dissipation inequalities that yield global control of the solutions. However, once temperature variations are incorporated, these cancellations break down. Furthermore, when the ionic species have different diffusivities, the entropy balance becomes even more delicate, and the standard arguments available in the literature are no longer applicable.

To overcome these difficulties and obtain the global well-posedness, we introduce a modified entropy functional that explicitly depends on the temperature and captures the thermodynamic structure of the model. The construction of this entropy constitutes one of the key contributions of the paper. The analysis then focuses on understanding the evolution of this modified entropy and exploiting it to derive a priori estimates for the system. 

After establishing the existence of local solutions, our objective is to show that these solutions can be extended beyond their local existence time. This requires obtaining bounds that remain uniform on arbitrary time intervals. The derivation of such bounds is highly nontrivial. The modified entropy depends explicitly on the temperature, and therefore its time evolution generates several nonlinear coupling terms involving both the ionic concentrations and the temperature. As a consequence, controlling the entropy cannot be achieved in isolation. Instead, one must simultaneously investigate the evolution of the temperature in higher-order Sobolev spaces. This leads naturally to a hierarchy of coupled energy estimates in which the control of one quantity depends on the control of several others. In particular, uniform bounds for the entropy require $H^1$-estimates for the ionic concentrations and the velocity field, and $H^{3/2}$-estimates for the temperature. To guarantee uniqueness and compatibility with the local well-posedness theory, higher-order energies must also be analyzed. The resulting argument involves a carefully designed chain of energy inequalities whose interactions must be balanced precisely in order to prevent the growth of the entropy.

The global extension result is obtained under the assumption that the initial temperature remains sufficiently close to its initial spatial average. This hypothesis is physically meaningful. In many practical electrochemical and fluid systems, temperature fluctuations are relatively small compared to the average operating temperature. Significant thermal gradients generally require strong external heating mechanisms or highly localized energy sources, whereas most realistic configurations operate in regimes where the temperature remains nearly homogeneous. Consequently, assuming that the initial temperature is a small perturbation of its average is consistent with a wide range of physical applications and allows the entropy structure of the system to remain sufficiently stable throughout the evolution.

The proof of local well-posedness presents a different set of challenges. As mentioned above, the logarithmic term creates serious difficulties for uniqueness. Since the derivative of the logarithm becomes singular near zero, standard uniqueness arguments fail when the concentrations are allowed to approach zero. To address this issue, we establish uniqueness under the assumption that the ionic concentrations remain uniformly bounded away from zero. This naturally raises the question of whether such a lower bound persists throughout the evolution. We therefore assume that the initial concentrations are bounded below by a positive constant and prove that this property is propagated in time.

The proof of this positivity preservation property is itself highly nontrivial. In fact, we address a family of drift-diffusion equations containing logarithmic forcing terms similar to those appearing in our original system. We establish a general result showing that positive lower bounds are preserved under suitable regularity assumptions. However, these assumptions involve higher-order norms of the concentrations, which means that the local existence theory must simultaneously provide sufficient regularity to justify the argument. As a result, even on short time intervals, the construction of a suitable chain of energy estimates becomes a delicate task.

Our strategy begins with the study of basic energy levels and their coupling. We then proceed to establish local bounds for the $H^1$-norms of the ionic concentrations. A crucial ingredient in this analysis is the study of the evolution of the quantity $\|\log c_i\|_{L^2}$. Remarkably, this evolution exhibits a previously unnoticed dissipation structure that plays a fundamental role in the analysis. The dissipation generated by this energy provides additional control over the logarithmic terms and allows us to close the higher-order estimates required for the local theory. This observation is one of the key technical innovations of the paper and serves as an essential bridge between the positivity arguments and the Sobolev regularity estimates.

Finally, the construction of local solutions requires the development of an approximation scheme compatible with all the singular features of the model. This step is particularly delicate because the temperature appears in the denominator of several terms, making standard Galerkin approximations difficult to implement. Moreover, the approximation procedure must remain meaningful at a stage where the positivity of the ionic concentrations has not yet been established, despite the presence of logarithmic nonlinearities.

To overcome these obstacles, we introduce a specialized regularization procedure combining classical mollification techniques with suitable cutoff functions. The mollifiers are used to regularize the nonlinear interactions, while the cutoff functions provide a rigorous treatment of the logarithmic terms in regions where the concentrations may become small. We prove that the resulting approximate system possesses global smooth solutions through an iterative scheme and derive estimates that are uniform with respect to the approximation parameters. These uniform bounds constitute the fundamental ingredients of classical compactness arguments. By combining compactness theorems with lower semicontinuity properties of the relevant norms, we are able to pass to the limit in the approximation scheme and recover local solutions of the original system.

Overall, the analysis requires overcoming several interconnected difficulties arising from the logarithmic nonlinearities, the temperature-dependent entropy structure, the coupling between multiple ionic species with different diffusivities, and the preservation of positivity. The resolution of these challenges relies on the construction of a new temperature-dependent entropy, the discovery of a useful logarithmic dissipation mechanism, the development of a carefully structured hierarchy of energy estimates, and the design of a robust approximation framework adapted to the singular nature of the equations.

\smallskip
\noindent
{\bf Organization of the paper.} The paper is organized as follows. In Section~\ref{sec2}, we state and prove new logarithmic estimates that will be frequently used in the paper. In Section~\ref{Sec3}, we establish a general criterion for the positive lower boundedness of solutions to drift-diffusion equations with logarithmic nonlinearities. In Section~\ref{sec4}, we construct a regularization scheme and establish Sobolev bounds of its solutions that are uniform in the regularization parameter. Section~\ref{sec5} is devoted to the local well-posedness of the non-isothermal NPNS system, whereas Section~\ref{sec6} is dedicated to its global well-posedness. In Appendix~\ref{app1}, we present some useful auxiliary lemmas that are needed for the paper. Finally, the local and global existence of solutions to the regularized scheme are established in Appendices~\ref{sec:appendix-b} and \ref{sec:appendix-c} via an iterative approximation.

Throughout the paper, $C$ denotes a positive universal constant that may change from line to line. When necessary, we denote by $C_{a,b,\dots}$ a positive constant depending on $a,b,\dots$.

\section{Logarithmic product estimates} \la{sec2}

The thermodiffusive effects on the ionic fluxes are dominated by a logarithm of the ionic concentrations, from which many challenges arise in the analysis of the problem. 

In this section, we present new logarithmic product  estimates that will be used later to establish the local and global well-posedness of the non-isothermal NPNS system. The estimates in this section are over $\mathbb T^2$.

For any $\eta > 0$, let $\chi(x)$ be a bounded smooth cutoff function such that $\chi(x) = 0$ when $x \le \eta$. With slight abuse of notation, we write $\chi(f)\log f$ to represent the cutoff extension, i.e., 
\[
\chi(f)\log f = \begin{cases}
    \chi(f)\log f, \quad & f>0,
    \\
    0, \quad &f\leq 0.
\end{cases}
\]
For any $s\in\mathbb R$, let $\Lambda^s= (-\Delta)^{\frac{s}{2}}$ be the periodic fractional Laplacian of order $s$.

\beg{lem} \la{Step1}
Let $m\geq 1$ be an integer. For any $\eta > 0$, any bounded smooth cutoff function $\chi(x)$ that vanishes when $x \le \eta$, any integer $j \ge 1$, and any function $f \in H^{m+1}$, it holds that 
\be \label{prop-logproduct-step1}
\sum\limits_{i=1}^{2} \left\|\frac{(\chi \circ f) \pa_i f}{f^j}\right\|_{H^m}   \le C\left(\|f\|_{H^{m+1}} + \|f\|_{H^{m}}^{m} \|f\|_{H^{m+1}} \right)
\ee  for some positive constant $C$ that depends on $\eta$, $j$, and $\chi$. 
\end{lem}

\begin{proof}
We employ an induction argument in $m$. When $m=1$, we have
\be 
\begin{aligned}
&\sum\limits_{i=1}^{2} \left\|\frac{(\chi \circ f) \pa_i f}{f^j}\right\|_{H^1}
\\ \le & \sum\limits_{i=1}^{2} \left\|\frac{(\chi \circ f) \pa_i f}{f^j}\right\|_{L^2} + \sum\limits_{i=1}^{2} \left\|\frac{\na (\chi \circ f) \pa_i f}{f^j}\right\|_{L^2} + \sum\limits_{i=1}^{2} \left\|\frac{ (\chi \circ f)\na \pa_i f}{f^j}\right\|_{L^2} 
+ j\sum\limits_{i=1}^{2} \left\|\frac{ (\chi \circ f)\pa_i f}{f^j}\frac{\na f}{f}\right\|_{L^2}
\\ \le &C (\|f\|_{H^{2}} + \|\na f\|_{L^4}^2)
\le C(\|f\|_{H^2} +  \|f\|_{H^1}\|f\|_{H^2})
\end{aligned}
\ee for any $f \in H^2$. Here we have used the upper boundedness of $\chi$ and $\chi'$ on $\R$, and the upper boundedness of $\frac1f$ by $\frac1\eta$ due to the vanishing of $\chi \circ f$ when $f < \eta$.  

Next, suppose that our desired statement holds for a positive integer $m$. We will show that for any $\eta > 0$, any bounded smooth cutoff function $\psi(x)$ that vanishes when $x \le \eta$, any integer $j \ge 1$, and any function $g \in H^{m+2}$, we have 
% \be 
% \left\|\frac{\psi \circ g}{g} \right\|_{H^{m+1}}
% \le C\left(1 + \|g\|_{H^{m}} + \|g\|_{H^{m+1}} + \|g\|_{H^{m}}^{m}\|g\|_{H^{m+1}} \right)
% \ee and 
\be \label{iu}
\sum\limits_{i=1}^{2}\left\|\frac{(\psi \circ g) \pa_i g}{g^j}\right\|_{H^{m+1}} \le C\left(\|g\|_{H^{m+2}} + \|g\|_{H^{m+1}}^{m+1} \|g\|_{H^{m+2}} \right)
\ee for some positive constant $C$ that depends on $\eta$, $j$, and $\psi$. We start by writing the $H^{m+1}$ norm of $(\psi \circ g)/g^j$ as the sum of the $L^2$ norm of $(\psi \circ g)/g^j$ and the $H^m$ norm of $\na((\psi \circ g)/g^j)$ which, after applying the Minkowski inequality and the Chain Rule, yields 
\be 
\begin{aligned}
\left\|\frac{\psi \circ g}{g^j} \right\|_{H^{m+1}}
&\le \left\|\frac{\psi \circ g}{g^j} \right\|_{L^2} + \left\| \frac{\na (\psi \circ g)}{g^j} \right\|_{H^{m}} + j\left\| \frac{(\psi \circ g) \na g}{g^{j+1}} \right\|_{H^{m}}
\\&\le \left\|\frac{\psi \circ g}{g^j} \right\|_{L^2} + \left\| \frac{\psi'(g)  \na g}{g^j} \right\|_{H^{m}} + j\left\| \frac{(\psi \circ g) \na g}{g^{j+1}} \right\|_{H^{m}}
\end{aligned}
\ee Since $\psi$ and $\pa_i \psi$ are smooth functions that vanish when $x \le \eta$, we can use the induction hypotheses to deduce that 
\be 
\left\|\frac{\psi \circ g}{g^j} \right\|_{H^{m+1}} \le \left\|\frac{\psi \circ g}{g^j}\right\|_{L^2} + C\left( \|g\|_{H^{m+1}} + \|g\|_{H^m}^{m} \|g\|_{H^{m+1}}\right).
\ee But $\psi$ is a bounded smooth cutoff function on $\TT^2$ so it attains its maximum, and thus 
\be \label{yt}
\left\|\frac{\psi \circ g}{g^j} \right\|_{H^{m+1}} \le  C\left(1 + \|g\|_{H^{m+1}} + \|g\|_{H^m}^{m}\|g\|_{H^{m+1}}\right)
\ee where $C$ depends on $\eta$, $j$, and $\psi$. Now we investigate the behavior of the $H^{m+1}$ norm of $(\psi \circ g) \pa_i g /g^j$. Since $H^{m+1}$ is a Banach algebra, it follows that 
\be 
\left\|\frac{(\psi \circ g) \pa_i g}{g^j} \right\|_{H^{m+1}}
\le \left\|\frac{\psi \circ g}{g^j} \right\|_{H^{m+1}}\|\na g\|_{H^{m+1}},
\ee and thus we obtain 
\be 
\left\|\frac{(\psi \circ g) \pa_i g}{g^j} \right\|_{H^{m+1}}
\le C\left(1+ \|g\|_{H^{m+1}} + \|g\|_{H^m}^{m} \|g\|_{H^{m+1}}\right)\|g\|_{H^{m+2}}
\ee in view of \eqref{yt}. Finally, due to the continuous Sobolev embedding of $H^{k}$ into $H^n$ for any $k \ge n$ and Young's inequality for products, we conclude that the desired inequality \eqref{iu} holds. By induction, we obtain \eqref{prop-logproduct-step1}.

\end{proof}

\beg{lem} \label{Step2}
Let $m \ge 1$ be an integer. For any $\eta > 0$, any bounded smooth cutoff function $\chi(x)$ that vanishes when $x \le \eta$ and such that $\chi'(x) = 0$ when $x \ge 2\eta$, and any function $f \in H^{m}$, it holds that 
\be \label{important}
\left\|(\chi \circ f) \log f\right\|_{H^m}   \le C\left(1 + \|f\|_{H^{m}} + \|f\|_{H^{m-1}}^{m-1} \|f\|_{H^{m}}   \right)
\ee  for some positive constant $C$ that depends on $\eta$, $m$ and $\chi$. 
\end{lem}

\begin{proof}
    We employ an induction argument in $m$. When $m=1$, we expand, apply H\"older's inequality, use the fact that $\chi' (f) = 0$ everywhere except when $\eta \le f \le 2\eta$, and obtain
\be 
\begin{aligned}
\|(\chi \circ f) \log f\|_{H^1}
&\le \|(\chi \circ f) \log f\|_{L^2} + \|\na (\chi \circ f) \log f \|_{L^2} + \left\|\frac{(\chi \circ f) \na f}{f} \right\|_{L^2} 
\\&\le C(1 + \|f\|_{L^2}) + C\|\na f\|_{L^2} 
\le C(1+\|f\|_{H^1})
\end{aligned}
\ee for any $f \in H^1$. 

Next, suppose our desired statement holds for a positive integer $m$. We prove that for any $\eta > 0$, any bounded smooth cutoff function $\psi$ such that $\psi = 0$ when $x \le \eta$ and $\psi' = 0$ when $x \ge 2\eta$, and any function $g \in H^{m+1}$, we have 
\be 
\left\|(\psi \circ g) \log g\right\|_{H^{m+1}} \le C\left(1 + \|g\|_{H^{m+1}} + \|g\|_{H^{m}}^{m} \|g\|_{H^{m+1}} \right)
\ee for some positive constant $C$ that depends on $\eta$ and $\psi$. 
If $m=1$, we have 
\be 
\begin{aligned}
\|(\psi \circ g) \log g\|_{H^{2}}
&\le \|(\psi \circ g) \log g \|_{L^2} + \|\na (\psi \circ g) \log g \|_{H^{1}} + \left\|\frac{(\psi \circ g)\na g}{g}\right\|_{H^{1}} 
\\&\le C(1+\|g\|_{L^2}) + \|\psi' (g) \na g \log g\|_{H^1} + C(\|g\|_{H^1}+1)\|g\|_{H^2}
\end{aligned}
\ee in view of Lemma \ref{Step1} and the vanishing properties of $\psi$ when $x \le \eta$ and of $\psi'$ when $x \ge 2\eta$ . But applications of the Chain Rule, H\"older's inequality, and Ladyzhenskaya's interpolation inequality give rise to 
\be 
\beg{aligned}
\|\psi'(g) \na g \log g\|_{H^1}
&= \|\psi'(g) \na g \log g\|_{L^2}
+ \|\na \left(\psi'(g) \na g \log g \right)\|_{L^2}
\\&\le C\|g\|_{H^1}+ C\|\na g\|_{L^4}^2 + C\|g\|_{H^2} 
\\&\le C\|g\|_{H^2} + C\|g\|_{H^1}\|g\|_{H^2}
\end{aligned}
\ee and consequently
\be 
\begin{aligned}
\|(\psi \circ g) \log g\|_{H^2}
\le C(1+  \|g\|_{H^2} + \|g\|_{H^1} \|g\|_{H^2}),
\end{aligned}
\ee which proves the case $m=1$.
Now suppose that $m \ge 2$. By Minkowski's inequality, we estimate 
\be 
\begin{aligned}
\|(\psi \circ g) \log g\|_{H^{m+1}}
&\le \|(\psi \circ g) \log g \|_{L^2} + \|\na (\psi \circ g) \log g \|_{H^{m}} + \left\|\frac{(\psi \circ g)\na g}{g}\right\|_{H^{m}} 
\end{aligned}
\ee Since $m \ge 2$, the space $H^{m}$ is Banach algebra, and thus
\be 
\|\na (\psi \circ g) \log g\|_{H^{m}} 
= \|\psi'(g)  \na g\log g\|_{H^{m}}
\le C\|\psi'(g) \log g\|_{H^{m}} \|\na g\|_{H^{m}}.
\ee As $\psi'$ is a smooth function that vanishes when $x \le \eta$ and its derivative vanishes when $x \ge 2\eta$, the induction hypothesis can be used to bound 
\be 
\|\psi'(g) \log g\|_{H^{m}}
\le C\left(1 +\|g\|_{H^{m}} + \|g\|_{H^{m-1}}^{m-1}\|g\|_{H^{m}} \right)
\ee and thus 
\be 
\|\na (\psi \circ g) \log g\|_{H^{m}} 
\le C \|g\|_{H^{m+1}} \left(1  + \|g\|_{H^{m}} + \|g\|_{H^{m-1}}^{m-1}\|g\|_{H^{m}} \right).
\ee Consequently, and in view of Lemma \ref{Step1}, we deduce that 
\be 
\begin{aligned}
&\|(\psi \circ g) \log g\|_{H^{m+1}} 
\\&\le C\left(1 + \|g\|_{H^{m+1}} + \|g\|_{H^{m+1}}\|g\|_{H^{m}}^{m} + \|g\|_{H^m}\|g\|_{H^{m+1}} + \|g\|_{H^m}\|g\|_{H^{m+1}}\|g\|_{H^{m-1}}^{m-1} \right)
\\&\le C\left(1  + \|g\|_{H^{m+1}} + \|g\|_{H^m}^{m}\|g\|_{H^{m+1}} \right)
\end{aligned}
\ee after applications of Young's inequality. 
\end{proof}

With Lemma~\ref{Step1} and Lemma~\ref{Step2}, we further establish the following propositions.

\begin{prop}\label{prop:log}
Let $m \ge 1$ be an integer. For any $\eta > 0$, let $\chi(x)$ be a bounded smooth cutoff function such that $\chi(x) = 0$ when $x \le \eta$ and $\chi(x) = 1$ when $x \ge 2\eta$. For any function $f \in H^{m} \cap L^{\infty}$, it holds that 
\be 
\begin{aligned}
\left\|(\chi \circ f) f \log f\right\|_{H^1}   
\le C(\|f\|_{H^1} + \|f\|_{L^{\infty}}\|f\|_{H^1}),
\end{aligned}
\ee  and for $m\geq 2$,
\be 
\begin{aligned}
\left\|(\chi \circ f) f \log f\right\|_{H^m}   
\le C\left(\|f\|_{H^m} + \|f\|_{H^m}^2 + \|f\|_{H^{m-1}}^{m-1}\|f\|_{H^m}^2\right).
\end{aligned}
\ee Here $C$ is a positive constant that depends on $\eta$ and $\chi$. 
\end{prop}

\begin{proof} 
 For any function $f \in H^{1}\cap L^\infty$, it holds that 
\be
\begin{aligned}
&\|(\chi \circ f) f \log f\|_{H^1}
\\=& \|(\chi \circ f) f \log f\|_{L^2} + \|\na((\chi \circ f) f \log f)\|_{L^2}
\\=& \|(\chi \circ f) f \log f\|_{L^2} + \|(\chi \circ f) \na f \log f\|_{L^2} + \|(\chi \circ f) \na f \|_{L^2} + \|\chi' (f) \na f  f \log f \|_{L^2}
\\\le & C\|f\|_{L^2} (1+ \|f\|_{L^{\infty}}) + C\|\na f\|_{L^2} (1+\|f\|_{L^{\infty}}) + C\|\na f\|_{L^2} + C\|\na f\|_{L^2}\|f\|_{L^{\infty}}
\\\le & C\|f\|_{H^1} + C\|f\|_{H^1}\|f\|_{L^{\infty}},
\end{aligned}
\ee due to the vanishing of $\chi'$ when $f$ does not take values in $[\eta, 2\eta]$. For $m \ge 2$ and $f\in H^{m}\cap L^\infty$, we have 
\be 
\begin{aligned}
\left\|(\chi \circ f) f \log f\right\|_{H^m}   
&\le C\|f\|_{H^m} \|(\chi \circ f) \log f\|_{H^m}
\\&\le C\left( \|f\|_{H^m} + \|f\|_{H^m}^2 + \|f\|_{H^{m-1}}^{m-1}\|f\|_{H^m}^2\right).
\end{aligned}
\ee Here $C$ is a positive constant that depends on $\eta$ and $\chi$. The latter follows directly from the fact that $H^m$ is a Banach algebra and from Lemma \ref{Step2}. 
\end{proof}

\begin{prop} \la{prop1}
Let $m \ge 1$ be an integer. Let $f\in H^m$ be a nonnegative real-valued function. For each $\eta > 0$, let $\chi$ be a smooth cutoff function such that $\chi(x) = 0$ when $x < \eta$ and $\chi(x) = 1$ when $x >2\eta$. Then it holds that 
\be \label{logprod}
\|\chi(f) \log f\|_{H^m} 
\le C (1+ \|f\|_{L^1}^m + \|\l^m f\|_{L^2}^m)
\ee for some positive constant $C$ depending only on $\eta$, $m$, $\chi$, and some universal constants. 
\end{prop}

\begin{proof} The proof follows directly from \eqref{important} and applications of classical continuous Sobolev embeddings, Poincar\'e's inequality, and Young's inequality for products.  
\end{proof}

\begin{prop}\label{prop:est-logterm}
Let $m \ge 2$ be an integer. Let $g$ be a smooth function. Let $f_n\in H^m$ be a sequence of nonnegative real-valued functions. For each $\eta > 0$, let $\chi$ be a smooth cutoff function such that $\chi(x) = 0$ when $x < \eta$ and $\chi(x) = 1$ when $x >2\eta$. Then it holds that 
\be \label{logprod-2}
\begin{aligned}
&\|\l^m(g \chi(f_n) f_{n+1} \log f_n)\|_{L^2} 
\\&\quad\le  C \left(\|g\|_{L^{\infty}} + \|\l^m  g\|_{L^2 }\right)\left(1 + \|f_n\|^m_{L^1} + \|\l^m f_n\|_{L^2}^m \right) \left(\|f_{n+1}\|_{L^1} + \|\l^m f_{n+1}\|_{L^2} \right)
\end{aligned}
\ee for some positive constant $C$ depending only on $\eta$, $m$, $\chi$, and some universal constants. 
\end{prop}

\begin{proof} Using classical periodic product estimates, we have 
\be 
\begin{aligned}
 &\|\l^m(g \chi(f_n) f_{n+1} \log f_n)\|_{L^2}
 \\&\quad\quad\le C\|\l^m (\chi(f_n) \log f_n)\|_{L^2} \|f_{n+1}\|_{L^{\infty}}\|g\|_{L^{\infty}}
 + C\|\chi(f_n) \log f_n\|_{L^{\infty}} \|\l^m(gf_{n+1})\|_{L^2}
 \\&\quad\quad\le C\|\l^m (\chi(f_n) \log f_n)\|_{L^2} \|f_{n+1}\|_{L^{\infty}}\|g\|_{L^{\infty}}
 \\&\quad\quad\quad\quad+ C(|\log \eta|+\|f_n\|_{L^\infty}) \left(\|\l^mf_{n+1}\|_{L^2}\|g\|_{L^{\infty}} + \|\l^m g\|_{L^2}\|f_{n+1}\|_{L^{\infty}} \right),
 \end{aligned}
 \ee 
 which, after interpolation, gives
 \be 
 \begin{aligned}
&\|\l^m(g \chi(f_n) f_{n+1} \log f_n)\|_{L^2}
\\&\quad\quad\le C\|\l^m (\chi(f_n) \log f_n)\|_{L^2} \left(\|f_{n+1}\|_{L^1} + \|\l^m f_{n+1}\|_{L^2} \right) \|g\|_{L^{\infty}}
\\&\quad\quad\quad+ C(|\log \eta|+\|f_n\|_{L^1} + \|\l^m f_n\|_{L^2}) \|g\|_{L^{\infty}} \|\l^m f_{n+1}\|_{L^2} 
\\
&\quad\quad\quad
+ C(|\log \eta|+\|f_n\|_{L^1} + \|\l^m f_n\|_{L^2})  \|\l^m g\|_{L^2} \left(\|f_{n+1}\|_{L^1} + \|\l^m f_{n+1}\|_{L^2} \right).
 \end{aligned}
 \ee An application of Proposition \ref{prop1} yields the desired estimate \eqref{logprod-2}. 
\end{proof}

\section{Positive Lower Boundedness of Solutions to Drift-Diffusion Equations with Logarithmic Nonlinearities} \label{Sec3}

In this section, we establish a general criterion that guarantees a uniform positive lower bound for solutions to certain drift-diffusion equations with logarithmic nonlinearities. More precisely, let $v$, $\mathcal{F}$, and $\mathcal{G}$ be vector fields defined on $\TT^2$, with $\nabla \cdot v = 0$, let $p$ be a smooth periodic scalar function, and let $\chi$ be a nonnegative smooth scalar cutoff function satisfying $\chi(x)=0$ for $x \le \eta$ and $\chi(x)=1$ for $x \ge 2\eta$, where $\eta>0$ is fixed. We ask: what regularity assumptions on $v$, $\mathcal{F}$, and $\mathcal{G}$ ensure that smooth solutions to equations of the form
\be \la{eqeq}
\pa_t q +  v \cdot \na q - \mu \Delta q
=  \na \cdot \left( C_1 q \mathcal{F}
+ (C_2 \chi(q)  \log q + C_3 \chi(p)\log p)q\mathcal{G}   \right)
\ee
maintain a strictly positive spatial infimum? Here, $\mu, C_1, C_2, C_3$ are prescribed constants with $\mu > 0$.

\beg{prop} (Criteria for nonnegativity) \label{nonneg} Let $\mathcal{T} > 0$ be arbitrary, and let $\eta>0$ be a fixed constant. %Suppose $\chi(x) = 0$ when $x \le \eta$ and $\chi(x)=1$ when $x\geq 2\eta$ for some positive constant $\eta$. 
Let $q$ be a smooth solution to \eqref{eqeq} on $[0, \mathcal{T}]$. Suppose the regularity criterion 
\be \label{nonneg-condition}
\int_{0}^{\mathcal{T}} C_1^2\|\mathcal{F}\|_{L^{\infty}}^2 + C_3^2(|\log \eta|^2 
+ \|p\|_{L^{\infty}}^2) \|\mathcal{G}\|_{L^{\infty}}^2 dt < \infty 
\ee holds. If $q_0 \ge 0$, then $q(x,t) \ge 0$ for every $t \in [0, \mathcal T]$ and a.e. $x \in \TT^2$.   

\end{prop}
\begin{proof} Denote the  negative part of $q$ by $q^-$, i.e., $q^- =\max\{-q,0\}$. Multiply \eqref{eqeq} by $-q^-$ and integrate over $\TT^2$. 
As $v$ is divergence-free, integrating by parts yields
\be 
\int_{\TT^2}v \cdot \na q q^- dx = -\int_{\TT^2}v \cdot \na q^- q^- dx = 0.
\ee 
In addition, using the following identities
\be 
\int_{\TT^2}\Delta q q^- dx = -\int_{\TT^2}\na q \na q^- dx = \int_{\TT^2}\na q^- \na q^- dx, 
\ee 
\be 
\chi(q) q^{-} = 0,
\ee we obtain the energy equality 
\be 
\frac{1}{2} \fr{d}{dt}\|q^-\|_{L^2}^2 + \mu \|\na q^-\|_{L^2}^2
= - C_1 \int_{\TT^2} q^- \mathcal{F} \cdot \na q^- dx- C_3\int_{\TT^2}\chi(p) q^- \log p \mathcal{G} \cdot \na q^- dx.
\ee 
Applications of the Cauchy-Schwarz and Young inequalities give rise to the differential inequality 
\be 
\frac{d}{dt} \|q^-\|_{L^2}^2
+ \mu \|\na q^-\|_{L^2}^2
\le C\left( C_1^2\|\mathcal{F}\|_{L^{\infty}}^2 +  C_3^2\|\mathcal{G}\|_{L^{\infty}}^2 \|\chi(p) \log p\|_{L^{\infty}}^2 \right)\|q^-\|_{L^2}^2.
\ee Since $\chi(p)$ vanishes when $p \le \eta$, the logarithmic term can be controlled by 
\be 
\|\chi(p) \log p\|_{L^{\infty}} 
\le C(|\log \eta| + \|p\|_{L^{\infty}}),
\ee and thus we deduce that 
\be 
\|q^-(t)\|_{L^2}^2 
\le \|q^-(0)\|_{L^2}^2 \exp \left\{ C\int_{0}^{t} \left(C_1^2\|\mathcal{F}\|_{L^{\infty}}^2 + C_3^2(|\log \eta|^2 + \|p\|_{L^{\infty}}^2) \|\mathcal{G}\|_{L^{\infty}}^2\right) ds \right\}
\ee as a consequence of Gr\"onwall's inequality. But the initial data $q_0$ is nonnegative, which forces its negative part to vanish. Therefore, $q^-(t) = 0$ in $L^2$ and so $q^- = 0$ a.e. on $\TT^2$. As a conclusion, we have $q(t) \ge 0$ for all $t \in [0,\mathcal{T}]$ and a.e. $x \in \TT^2$. 
\end{proof}

\begin{prop}\label{prop:pos-lower-bdd}
(Criteria for positive lower boundedness) 
 Let $\mathcal{T} > 0$ be arbitrary, and let $\eta>0$ be a fixed constant. %Suppose $\chi(x) = 0$ when $x \le \eta$ and $\chi(x)=1$ when $x\geq 2\eta$ for some positive constant $\eta$. 
 Let $q$ be a nonnegative smooth solution to \eqref{eqeq} on $[0, \mathcal{T}]$. Suppose the regularity criterion 
\be \label{positive-condition}
\mathcal{O}= \int_{0}^{\mathcal{T}} \Big[|C_1| \|\na \cdot \mathcal{F}\|_{L^{\infty}} + |C_2| \|\na \cdot \mathcal{G}\|_{L^{\infty}} + |C_3| \left(\|\chi(p) \log p\|_{L^{\infty}} + \|\na (\chi(p) \log p) \|_{L^{\infty}} \right)\|\mathcal{G}\|_{W^{1,\infty}} \Big] dt < \infty 
\ee 
holds. If $q_0(x) \ge a_0 >0$ on $\TT^2$, then $$q(x,t) \ge \min \left\{\frac{a_0}{2}, \exp \left\{1-\left(1+ \log \frac{2}{a_0}\right)\exp{\mathcal{O}}  \right\} \right\}$$ for every $(x,t) \in \TT^2\times[0, \mathcal T]$.
\end{prop}

\begin{proof}  We denote the minimum of $q_0$ on $\TT^2$ by $m_0$. For each $t \in [0, \mathcal{T}]$, let 
\be 
m(t) = \inf_{x \in \TT^2} q(x,t).
\ee For each $t \in [0, \mathcal{T}]$, there exists $x_t \in \TT^2$ such that 
\be
m(t) = q (x_t, t) \quad \text{and} \quad m'(t) = \partial_t q(x_t,t)
\ee 
hold simultaneously due to the smoothness properties satisfied by $q$ (see Appendix B in \cite{constantin2015long}). 
At $(x_t, t)$, we have 
\be 
\na q (x_t, t) = 0, \quad \Delta q (x_t, t) \ge 0,
\ee from which we infer that 
\be 
\nabla [\chi(q) \log q](x_t,t)=0.
\ee 
Consequently, $m(t)$ obeys the differential inequality 
\be 
\begin{aligned}
m'(t) 
&\ge C_1 m(t) \na \cdot \mathcal{F}(x_t, t) + (C_2 m(t)\chi(m(t))  \log m(t) + C_3m(t) \chi(p(x_t, t)) \log p(x_t, t))\na \cdot \mathcal{G}(x_t, t) 
\\&\quad\quad\quad\quad+ C_3 m(t) \mathcal{G}(x_t,t) \cdot \nabla\left(\chi(p) \log(p) \right)(x_t,t).
\end{aligned}
\ee 
Denote the instantaneous supremum of $|\chi(p)\log p|$  and $|\na (\chi(p) \log p)|$ over $\TT^2$ by $K(t)$. Using the upper boundedness of $\chi$ by $1$, we infer that 
\be \la{ty}
m'(t) + \left(|C_1| \|\na \cdot \mathcal{F}\|_{L^{\infty}} + |C_3| K \|\mathcal{G}\|_{W^{1,\infty}}\right) m(t) \ge - |C_2| \|\na \cdot \mathcal{G}\|_{L^{\infty}} m(t) |\log m(t)|
\ee Denote by $F(t)$ and $G(t)$ the following quantities:
\be 
F(t) = |C_1| \|\na \cdot \mathcal{F}\|_{L^{\infty}} + |C_3| K \|\mathcal{G}\|_{W^{1,\infty}},
\ee
\be 
G(t) = |C_2| \|\na \cdot \mathcal{G}\|_{L^{\infty}}, 
\ee in which case \eqref{ty} yields  
\be \la{tr}
m'(t) \ge -m(t) \left(1 +  |\log m(t)| \right) (F(t) +G(t))
\ee due to the nonnegativity of both $F$ and $G$.  Since $m_0 \ge a_0 >0$ and $m$ is continuous in time, we deduce the existence of a time $t_0$ such that $m(t) \ge \frac{a_0}{2}$ for all $t \in [0,t_0]$. Since $a_0$ is only a positive lower bound for $m_0$, replacing it by $\min\{a_0,1\}$ if necessary, we may
assume without loss of generality that $0<a_0\leq 1$. Next, we introduce the set of times 
\be 
\mathcal{S} = \left\{t \in [0, \mathcal{T}]: m(t) = a_0/2\right\}.
\ee  If $\mathcal S$ is empty, i.e., $m(t) > \frac{a_0}{2}$ for all times $t \in [0, \mathcal{T}]$, then we are done. Otherwise, there is a time $\tau \in [0, \mathcal{T}]$ such that $m(\tau) = a_0/2$.  Let $t \in (0, \mathcal{T}]$. If $m(t) \ge \frac{a_0}{2}$, we disregard this time. If $m(t) < \frac{a_0}{2}$, then by continuity $m(t) \le \frac{a_0}{2}$ on $[t_{\beta}, t]$  for some $t_{\beta} \in \mathcal{S}$. In the latter situation, it holds that $|\log m(t)| = - \log m(t)$ for any $t \in [t_{\beta}, t]$, and due to \eqref{tr}, we obtain
\be 
m'(t) \ge - m(t) (1- \log m(t)) (F(t) + G(t)).
\ee Dividing both sides by $m(t) (1 - \log m(t))$ and integrating in time from $t_{\beta}$ to $t$, we deduce that 
\be 
\log (1 - \log m(t))\le \log (1+ \log \frac{2}{a_0}) + \int_{0}^{\mathcal{T}} (F(s) + G(s)) ds. 
\ee Exponentiating twice gives rise to the lower bound 
\be 
m(t) \ge \exp \left\{1-\left(1+ \log \frac{2}{a_0}\right)e^{\int_{0}^{T}(F(s) + G(s)) ds}  \right\}.
\ee Therefore, $m(t)$ is bounded from below and obeys 
\be 
m(t) \ge \min \left\{\frac{a_0}{2}, \exp \left\{1-\left(1+ \log \frac{2}{a_0}\right)e^{\int_{0}^{T}(F(s) + G(s)) ds}  \right\} \right\}
\ee for any $t \in [0, \mathcal{T}]$. 
\end{proof}

\section{Mollified system} \label{sec4}

The local well-posedness of the non-isothermal NPNS is a challenging problem due to the nonlinear and nonlocal aspects of the model, in addition to the logarithmic thermodiffusive effects. In order to address the latter, we need to construct regularized systems, guaranteed to have global smooth solutions that converge, locally and uniformly in time, to a solution of the original  model. 

To this end, we consider a family of standard periodic mollifiers $\left\{\phi_{\eta}  \right\}_{\eta \in (0,1)}$ with $\int_{\TT^2} \phi_{\eta}(x) dx = 1$ for all $\eta \in  (0,1)$, and we denote by $\mathcal J_{\eta}$ the convolution mollification operator $\mathcal{J}_{\eta} f = \phi_{\eta} * f$. We recall that  $\mathcal{J}_{\eta}$ is bounded on all $L^p$ spaces for any $p \in [1,\infty]$ and uniformly in $\eta$. Moreover,  $\mathcal{J}_{\eta}$ is self-adjoint and obeys $\int_{\TT^2} \mathcal J_\eta f dx = \int_{\TT^2} f dx$ for any $\eta \in (0,1)$.

Let $\chi(x)$ be a cutoff function such that $0\leq \chi(x)\leq 1$, $\chi(x)=0$ when $x\leq 1$, and $\chi(x) =1$ when $x\geq 2$. 
For each $\eta \in (0,1)$, let 
\begin{equation}\label{chi-1}
    \chi^{\eta}(x) = \chi(x/\eta).
\end{equation}
Then  $\chi^\eta(x) = 1$ when $x \geq 2\eta$ and $\chi^\eta(x) = 0$ when $x\leq \eta$. With this notation, we 
consider the $\eta$-regularized system  
\noeqref{N-NPNS-mo-1}
\noeqref{N-NPNS-mo-2}
\noeqref{N-NPNS-mo-3}
\noeqref{N-NPNS-mo-4}
\noeqref{N-NPNS-mo-5}
\begin{subequations}\label{N-NPNS-mo-system}
    \begin{align}
        &\partial_t c_i^\eta + \j  u^\eta\cdot \nabla c_i^\eta = D_i\nabla \cdot
\left(
 \nabla c^\eta_i
+
 \frac{z_i e}{k_B T^\eta} c^\eta_i \nabla \j \phi^\eta
+
 \chi^\eta (c_i^{\eta}) c_i^\eta \log c_i^\eta \j \nabla \log  T^\eta
\right),
 i=1,\dots,N,\label{N-NPNS-mo-1}
        \\
        &-\varepsilon\Delta \phi^\eta = \rho^\eta = \sum\limits_{i=1}^N e  z_i c_i^\eta, \label{N-NPNS-mo-2}
        \\
        &\partial_t u^\eta + \j u^\eta\cdot \nabla u^\eta - \nu\Delta u^\eta + \nabla p^\eta 
        = g\alpha_T(T^\eta-T^\eta_r)  \vec{k} -\j  \left( \rho^\eta \nabla \j    \phi^\eta\right), \label{N-NPNS-mo-3}
        \\
        &\nabla\cdot u^\eta = 0,\label{N-NPNS-mo-4} 
        \\
        &\partial_t T^\eta +\j   u^\eta\cdot \nabla T^\eta - \kappa\Delta T^\eta = 0,\label{N-NPNS-mo-5}
    \end{align}
\end{subequations}
with initial conditions 
\[(c_i^\eta(0),u^\eta(0),T^\eta(0))=(\j c_{i0},\j u_0,\j T_0):=(c^\eta_{i0},u^\eta_0,T^\eta_0).
\]
Note that these mollified initial data are all smooth (in $C^\infty(\mathbb T^2)$) provided that $c_{i0}, u_0, T_0$ are $L^1$ integrable. The reference temperature is given by $T_r^\eta = \overline{T^\eta}$.
According to the properties of mollifiers, the conditions $c_{i0}\geq a_0 > 0$, $T_0\geq T^*>0$, and $\overline{u_0}=0$ imply that $c_{i0}^\eta\geq a_0 >0$, $T^\eta_0\geq T^*>0$, and $\overline{u^\eta_0}=0$. Also, $\overline{T^\eta} = \overline{T^\eta_0} = \overline{T_0}$ and $\overline{c_i^\eta} = \overline{c_{i0}^\eta}= \overline{c_{i0}}$ are invariant in time, and thus $T_r^\eta = \overline{T_0}$. Moreover, $\overline{u^\eta}=\overline{u^\eta_0}=0$ and  $u^\eta \in H$.

\begin{rem}
   The term $\log c_i$ in the original system makes physical sense, since ionic concentrations cannot be negative. From a mathematical standpoint, however, it is not well-defined until the nonnegativity of $c_i$ has been established. This creates a difficulty, because one must first prove the existence of solutions before proving such a property. To overcome this issue, we introduce a cutoff function and replace the logarithmic term by $\chi^\eta(c_i^\eta)\log c_i^\eta$. This expression is well-defined, since $\chi^\eta(c_i^\eta)$ vanishes whenever $c_i^\eta \le \eta$, so the truncated logarithmic term can be extended by zero in that region. Thus, the mollified system is a well-defined mathematical model, and one can then study the existence of its solutions. 
\end{rem}

The following proposition addresses the global existence of smooth solutions to the $\eta$-regularized system:

\begin{prop} \label{ap} 
Let $\eta \in (0,1)$ and let $\mathcal T > 0$ be arbitrary. Suppose $c_{i0}, u_0, T_0 \in L^1(\TT^2)$ with $c_{i0}\geq 0$ and $T_0\geq T^*>0$. Assume $c_{i0}$ satisfies the compatibility condition \eqref{eqn:compa}, and $u_0$ has zero mean and is divergence-free. The $\eta$-regularized system \eqref{N-NPNS-mo-system} has a solution $(c^\eta_i, u^\eta, T^\eta)$ on $[0, \mathcal T]$ such that
\be \label{eta-regularity}
(c^\eta_i, u^\eta, T^\eta) \in C([0,\mathcal{T}]; H^m(\TT^2))
\ee for any $m \in \N$. Moreover, $T^\eta\geq T^*$ and $c^\eta_i\geq 0$. 
\end{prop}

The proof of Proposition~\ref{ap} is established in Appendices \ref{sec:appendix-b} and \ref{sec:appendix-c}, where Appendix \ref{sec:appendix-b} addresses the local well-posedness of the $\eta$-regularized system and  Appendix \ref{sec:appendix-c} extends any local solution to a global solution based on classical energy methods. In addition, we point out that the uniqueness of smooth solutions also holds. We do not state this explicitly, since uniqueness of solutions to the mollified system is not needed for proving the well-posedness of the original system.

\begin{rem}
    In Proposition~\ref{ap}, we only require the initial concentrations to be nonnegative. However, we will require $c_{i0}\geq a_0>0$ to be strictly positive in order to prove the existence and uniqueness of solutions to the original system. The preservation of a strictly positive lower bound for the ionic concentrations is much more involved, which will be established in the rest of this section.
\end{rem}

%The rest of this section is dedicated to establishing uniform-in-$\eta$ bounds that are needed for the construction of unique local solutions to the original system. 

In order to prove existence and uniqueness of solutions to the original system, we need not only uniform-in-$\eta$ bounds of solutions $(c_i^\eta, u^\eta, T^\eta)$, but also the uniform-in-$\eta$ strictly positive lower bound of $c_i^\eta$ provided that initially $c^\eta_{i0} \geq a_0>0$. These two together require uniform-in-$\eta$ bounds of $c_i^\eta, u^\eta, T^\eta$ in $L_{t}^{\infty}H_x^1\cap L_t^2 H_x^2$, $L_{t}^{\infty}H_x^1\cap L_t^2 H_x^2$, and $L_{t}^{\infty}H_x^{5/2}\cap L_t^2 H_x^{7/2}$, respectively. In particular, the high regularity requirement for the temperature is due to the need to prove a strict positive lower bound of $c_i^\eta$ (see Remark~\ref{rem:T}). However, we are not able to directly prove the uniform-in-$\eta$ bounds of $c_i^\eta$ in $L_{t}^{\infty}H_x^1\cap L_t^2 H_x^2$. Instead, one needs to implement the following process. 
\begin{enumerate}[Step 1.]
    \item Show that $c_i^{\eta}$, $u^\eta$, and $T^\eta$ are uniformly bounded in $\eta$ in $L_{t}^{\infty}L_x^2\cap L_t^2 H_x^1$, $L_{t}^{\infty}H_x^1\cap L_t^2 H_x^2$, and $L_{t}^{\infty}H_x^{5/2}\cap L_t^2 H_x^{7/2}$, respectively, on some short time interval $[0,\mathcal T_1]$. 
    \item Show that for each $\eta\in(0,1)$, there exists some positive constant $a_\eta>0$ such that the global regularized concentrations $c_i^{\eta}$ obey $c_i^{\eta} \ge a_{\eta}$ for all $i=1,\dots,N$ and for all times $t \in [0,\mathcal T_1]$. Then show that $\log c_i^{\eta}$ are uniformly bounded in $L_{t}^{\infty}L_x^2$ and $L_t^2 H_x^1$ on some short time interval $[0,\mathcal{T}_2]$ with $\mathcal T_2\leq \mathcal T_1$.
    \item Show that $c_i^{\eta}$ are uniformly bounded in $L_t^{\infty}H_x^1$ and $L_t^2 H_x^2$ on $[0,\mathcal{T}_0]$ with $\mathcal T_0 \leq \mathcal T_2$.
    \item Show that there exists some constant $a>0$, independent of $\eta$, such that $c_i^\eta \geq a$ on $[0,\mathcal{T}_0]$ for any $\eta$.
\end{enumerate}
We summarize these results in the following theorem.
\begin{Thm}\label{thm:uniform-bound}
    Let $c_{i0} \in H^1$, $u_0 \in V$, $T_0 \in H^{\frac{5}{2}}$, with $T_0\geq T^*>0$ and $c_{i0}\geq a_0> 0$. Suppose $c_{i0}$ satisfy the compatibility condition \eqref{eqn:compa} and $u_0$ has zero mean. There exists a time $\mathcal T_0>0$ such that $(u^\eta, T^\eta, c_i^\eta)$ satisfy the following uniform-in-$\eta$ bounds: 
    \begin{align*}
        &u^\eta \,\,  \text{are uniformly bounded in} \,\,  L^\infty(0,\mathcal T_0; V)\cap L^2(0,\mathcal T_0; \mathcal D(A)),
        \\
        & T^\eta \,\,  \text{are uniformly bounded in} \,\, L^\infty(0,\mathcal T_0; H^{\frac52})\cap L^2(0,\mathcal T_0; H^{\frac72}),
        \\
        &c_i^\eta   \,\,  \text{are uniformly bounded in} \,\,  L^\infty(0,\mathcal T_0; H^1) \cap L^2(0,\mathcal T_0;  H^{2}).
    \end{align*}
    Moreover, there exists a constant $a>0$, independent of $\eta$, such that $c_i^\eta \geq a >0$ for $t\in[0,\mathcal T_0]$ and a.e. $x\in\TT^2$.
\end{Thm}
The rest of this section is dedicated to establishing the four steps above.

\medskip
{\bf Step 1.}
We start by showing that $c_i^{\eta}$, $u^\eta$, and $T^\eta$ are uniformly bounded in $\eta$ in $L_{t}^{\infty}L_x^2\cap L_t^2 H_x^1$, $L_{t}^{\infty}H_x^1\cap L_t^2 H_x^2$, and $L_{t}^{\infty}H_x^{5/2}\cap L_t^2 H_x^{7/2}$, respectively, on some short time interval $[0,\mathcal T_1]$.

\begin{prop}\label{prop:uni-est-sol} 
    Suppose $c_{i0} \in L^2$, $u_0 \in V$, $T_0 \in H^{\frac{5}{2}}$, with $T_0\geq T^*>0$ and $c_{i0}\geq 0$. Let $c_{i0}$ satisfy the compatibility condition \eqref{eqn:compa} and $u_0$ has zero mean. For each $\eta\in(0,1)$, let $(u^\eta, c_i^\eta, T^\eta)$ be any smooth solution to the mollified system.
    There exists a time $\mathcal{T}_1$, such that $(u^\eta, T^\eta, c_i^\eta)$ satisfy the following uniform-in-$\eta$ bounds: 
    \begin{align*}
        &u^\eta \,\,  \text{are uniformly bounded in} \,\,  L^\infty(0,\mathcal T_1; V)\cap L^2(0,\mathcal T_1; \mathcal D(A)),
        \\
        & T^\eta \,\,  \text{are uniformly bounded in} \,\, L^\infty(0,\mathcal T_1; H^{\frac52})\cap L^2(0,\mathcal T_1; H^{\frac72}),
        \\
        &c_i^\eta   \,\,  \text{are uniformly bounded in} \,\,  L^\infty(0,\mathcal T_1; L^{2}) \cap L^2(0,\mathcal T_1;  H^{1}).
    \end{align*}
\end{prop}
\begin{proof} 
    The regularized temperatures $T^{\eta}$ obey 
    \begin{equation}
    \|T^\eta (t)\|_{L^2}^2 + 2\kappa \int_0^t \|\nabla T^\eta(s)\|_{L^2}^2 ds = \|T^\eta_0\|_{L^2}^2 \leq \|T_0\|_{L^2}^2, \quad \|T^\eta(t)\|_{L^p} \leq \|T^\eta_0\|_{L^p} \leq \|T_0\|_{L^p}
\end{equation}
for any $p > 2$. This implies that
$\{T^\eta\}$ are uniformly bounded in $ L^\infty(0,\mathcal T; L^p)\cap L^2(0,\mathcal T; H^1)$ for any $\mathcal T>0$.
\begin{comment}
As $T^\eta_r=\overline{T^\eta}$ is invariant in time, we have the Poincar\'e inequality
$
\|T-T_r\|_{L^2}^2 \leq C_p \|\nabla T\|_{L^2}^2.
$
Taking the $L^2$ inner product of the equation obeyed by $T$ with $T-T_r$ and using the Poincar\'e inequality above, we obtain
\[
\frac{1}{2} \frac{d}{d t}\|T-T_r\|_{L^2}^2+ \frac{\kappa}{C_p}\|T-T_r\|_{L^2}^2 %=
\le0,
\]
yielding that for any $t>0$,
\begin{equation}\label{est:T}
    \|T(t)-T_r\|_{L^2}^2 \leq \|T_0-T_r\|_{L^2}^2 e^{-\frac{2\kappa}{C_p} t}.
\end{equation}
\end{comment}

Next, we couple and study the evolutions of $\|c^\eta_i\|_{L^2}$, $\|u^\eta\|_{H^1}$, and $\|T^\eta\|_{H^{\frac32}}$. In fact, the $L^2$ norm of each regularized concentration $c_i^{\eta}$ satisfies the following energy inequalities 
\begin{equation}
\begin{aligned}
&\hspace{0.3cm}\frac{1}{2}\frac{d}{dt}\|c^\eta_i\|_{L^2}^2 + D_i\|\na c^\eta_i\|_{L^2}^2 
\\
&= -z_i D_i k_{B}^{-1} e\int_{\TT^2}\frac1{T^\eta} c^\eta_i \na \j \phi^\eta \cdot \na c^\eta_i dx
-D_i \int_{\TT^2} \chi^\eta(c^\eta_i) (c^\eta_i \log c^\eta_i) \j\nabla \log T^\eta \cdot \na c^\eta_i dx
\\
&\le D_i C\left(  \|c^\eta_i\|_{L^4}\|\na \phi^\eta\|_{L^4} + \|\chi^\eta(c^\eta_i) (c^\eta_i \log c^\eta_i) \j \na \log  T^\eta\|_{L^2} \right)\|\na c^\eta_i\|_{L^2} 
\\
&\le D_i C\left((\|c^\eta_i\|_{L^2}^{\frac{1}{2}}\|\na c^\eta_i\|_{L^2}^{\frac{1}{2}} + \|c^\eta_i\|_{L^2} )\|\na \phi^\eta\|_{L^2}^{\frac{1}{2}}\|\rho^\eta\|_{L^2}^{\frac{1}{2}} + \|((c^\eta_i)^{\frac{9}{8}} +1)  \j \na \log T^\eta\|_{L^2} \right)\|\na c^\eta_i\|_{L^2}
\\
&\le \frac{D_i}{4}\|\na c^\eta_i\|_{L^2}^2 + C\|c^\eta_i\|_{L^2}^2 \|\na \phi^\eta\|_{L^2}^2 \|\rho^\eta\|_{L^2}^2 
+ C \|c^\eta_i\|_{L^2}^2\|\rho^\eta\|_{L^2}^2 + C\|\na T^\eta\|_{L^2}^2 + C\|(c^\eta_i)^{\frac{9}{8}} \j \na \log  T^\eta\|_{L^2}^2.
\end{aligned}
\end{equation} 
Here we have used the logarithmic inequality $|c^\eta_i \log c^\eta_i| \le C(1 +  (c^\eta_i)^{\frac{9}{8}})$ that holds due to the nonnegativity of $c^\eta_i.$ 
Using classical Gagliardo-Nirenberg interpolation inequalities, we have
\be 
\begin{aligned}
&C\|(c^\eta_i)^{\frac{9}{8}}\j \na \log  T^\eta\|_{L^2}^2
\le C\|c^\eta_i\|_{L^4}^2\| (c^\eta_i)^{\frac{1}{8}}\|_{L^8}^2 \|\na T^\eta\|_{L^{8}}^2
\\
\le &C(\|c^\eta_i\|_{L^2}\|\na c^\eta_i\|_{L^2} + \|c^\eta_i\|_{L^2}^2)\|c^\eta_i\|_{L^1}^{\frac{1}{4}}\|\Lambda^{\frac32} T^\eta\|_{L^{\frac83}}^2
\\
\leq &C(\|c^\eta_i\|_{L^2}\|\na c^\eta_i\|_{L^2} + \|c^\eta_i\|_{L^2}^2)\|c^\eta_i\|_{L^1}^{\frac{1}{4}}\|\Lambda^{\frac32} T^\eta\|_{L^{2}}^{\frac32} \|\Lambda^{\frac52} T^\eta\|_{L^{2}}^{\frac12}
\\
\le &\frac{D_i}{4}\|\na c^\eta_i\|_{L^2}^2 + \frac\kappa{4N} \|\Lambda^{\frac52} T^\eta\|_{L^2}^2+ C\|c^\eta_i\|_{L^1}\|\l^{\frac{3}{2}} T^\eta\|_{L^2}^{6} \|c^\eta_i\|_{L^2}^4 +  C\|c^\eta_i\|_{L^1}^{\frac{1}{3}} \|\l^{\frac{3}{2}}T^\eta\|_{L^2}^2\|c^\eta_i\|_{L^2}^{\frac83}.
\end{aligned}
\ee 
As $\|c^\eta_i\|_{L^1} = \|c^\eta_{i0}\|_{L^1}$ is conserved in time, the latter gives rise to the differential inequality 
\be 
\begin{aligned}
\frac{d}{dt}\|c^\eta_i\|_{L^2}^2
+ D_i\|\na c^\eta_i\|_{L^2}^2 
\le &\frac\kappa{2N} \|\Lambda^{\frac52} T^\eta\|_{L^2}^2 +  C\|\na T^\eta\|_{L^2}^2 + C\left(  \|\na \phi^\eta\|_{L^2}^2 \|\rho^\eta\|_{L^2}^2 
+ \|\rho^\eta\|_{L^2}^2\right)\|c^\eta_i\|_{L^2}^2
\\
&+  C\|\l^{\frac{3}{2}} T^\eta\|_{L^2}^{6} \|c_i^\eta\|_{L^2}^4 +  C \|\l^{\frac{3}{2}}T^\eta\|_{L^2}^2\|c^\eta_i\|_{L^2}^{\frac83},
\end{aligned}
\ee
and therefore,
\be 
\begin{aligned}
\frac{d}{dt}\sum_{i=1}^N\|c^\eta_i\|_{L^2}^2
+ \sum_{i=1}^N D_i\|\na c^\eta_i\|_{L^2}^2 
\le &\frac\kappa2 \|\Lambda^{\frac52} T^\eta\|_{L^2}^2 + C\|\na T^\eta\|_{L^2}^2 + C\left(1+\sum_{i=1}^N\|c^\eta_i\|_{L^2}^2\right)^3
\\
&+C(1+\|\l^{\frac{3}{2}} T^\eta\|_{L^2}^{2} )^3\left(1+\sum_{i=1}^N\|c^\eta_i\|_{L^2}^2\right)^2 .
\end{aligned}
\ee

Now we move on to the $H^1$ evolution of the regularized velocities. However, and due to the Poincar\'e inequality,  we only need to address the $L^2$ norm of their gradients, which evolve according to 
\begin{align*}
    \frac{1}{2}\frac{d}{dt}\|\na u^\eta\|_{L^2}^2
+ \nu \|\Delta u^\eta\|_{L^2}^2 = - \int_{\TT^2} (\j u^\eta\cdot \nabla) u^\eta \cdot \Delta u^\eta dx  &-g\alpha_T \int_{\TT^2}(T^\eta - T^\eta_r)e_2 \Delta u^\eta dx 
\\
&+ \int_{\TT^2}\j(\rho^\eta \na \j \phi^\eta) \Delta u^\eta dx .
\end{align*}
Here the nonlinear term will not vanish since $\j u^\eta\neq u^\eta$. We integrate by parts and use the Ladyzhenskaya interpolation inequality to bound
\begin{align*}
    -\int_{\TT^2} (\j u^\eta\cdot \nabla) u^\eta \cdot \Delta u^\eta dx = \int_{\TT^2} \j \nabla u^\eta\cdot \nabla u^\eta \cdot \nabla u^\eta dx \leq C \|\nabla u^\eta\|_{L^2}^2 \|\Delta u^\eta\|_{L^2}.
\end{align*}
In view of the H\"older, Young, and Ladyzhenskaya interpolation inequalities, together with classical elliptic estimates, we have
\begin{align*}
    \frac12\frac{d}{dt}\|\na u^\eta\|_{L^2}^2
+ \nu \|\Delta u^\eta\|_{L^2}^2 \leq &C (\|\nabla u^\eta \|_{L^2}^2+\|T^\eta - T^\eta_r\|_{L^2} + \|\rho^\eta\|_{L^4} \|\nabla\phi^\eta\|_{L^4})\|\Delta u^\eta \|_{L^2}
\\
\leq & C(\|\nabla u^\eta \|_{L^2}^4+\|T^\eta - T^\eta_r\|_{L^2}^2 + \|\rho^\eta\|^3_{L^2} \|\nabla\rho^\eta\|_{L^2}) + \frac12\nu \|\Delta u^\eta\|_{L^2}^2,
\end{align*}
where the zero-mean condition $\overline{\rho^\eta}=0$ is exploited.
Therefore,
\begin{align*}
    \frac{d}{dt}\|\na u^\eta\|_{L^2}^2
+ \nu \|\Delta u^\eta\|_{L^2}^2 \leq &C(\|\nabla u^\eta \|_{L^2}^4+\|T^\eta - T^\eta_r\|_{L^2}^2 + \|\rho^\eta\|^3_{L^2} \|\nabla \rho^\eta\|_{L^2})
\\
\leq & C(\|\nabla u^\eta \|_{L^2}^4+\|T^\eta_0 - T^\eta_r\|_{L^2}^2) + C\left(\sum_{i=1}^N \|c^\eta_i\|_{L^2}^2\right)^{\frac32} \left(\sum_{i=1}^N \|\nabla c^\eta_i\|_{L^2}\right)
\\
\leq & C(\|\nabla u^\eta \|_{L^2}^4+\|T^\eta_0 - T^\eta_r\|_{L^2}^2) +C\left(\sum_{i=1}^N \|c^\eta_i\|_{L^2}^2\right)^3 + \frac12 \sum_{i=1}^N D_i\|\nabla c^\eta_i\|_{L^2}^2.
\end{align*}

Regarding the $H^{\frac{3}{2}}$ norm of the regularized temperatures, and due to the Poincar\'e inequality, it is enough to consider the norms $\|\Lambda^{\frac32} T^\eta\|_{L^2}$, which obey 
\begin{equation}
\begin{aligned}
\frac{1}{2}\frac{d}{dt}\|\Lambda^{\frac{3}{2}}T^\eta\|_{L^2}^2 + \kappa \|\Lambda^{\frac{5}{2}}T^\eta\|_{L^2}^2
= - \int_{\TT^2}\Lambda (\j u^\eta \cdot \na T^\eta) \Lambda^{2}T^\eta dx
\le \|\na (\j u^\eta \cdot \na T^\eta)\|_{L^{\frac{4}{3}}} \|\Delta T^\eta\|_{L^4}.
\end{aligned}
\end{equation}
By making use of continuous Sobolev embeddings and interpolation inequalities, we obtain 
\begin{equation}
\begin{aligned}
&\frac{1}{2}\frac{d}{dt}\|\Lambda^{\frac{3}{2}}T^\eta\|_{L^2}^2 + \kappa \|\Lambda^{\frac{5}{2}}T^\eta\|_{L^2}^2
\\&\le C(\|\na u^\eta\|_{L^2}\|\na T^\eta\|_{L^4} + \|u^\eta\|_{L^4} \|\Delta T^\eta\|_{L^2})\|\Lambda^{\frac{5}{2}} T^\eta\|_{L^2}
\\&\le C\|\na u^\eta\|_{L^2} \|\Lambda^{\frac{3}{2}}T^\eta\|_{L^2} \|\Lambda^{\frac{5}{2}}T^\eta\|_{L^2} 
+C \|\na u^\eta\|_{L^2}\|\Lambda^{\frac{3}{2}}T^\eta\|_{L^2}^{\frac{1}{2}}\|\Lambda^{\frac{5}{2}}T^\eta\|_{L^2}^{\frac{3}{2}}
\\&\le C\|\na u^\eta\|_{L^2}^2 \left(\|\nabla u^\eta\|_{L^2}^2 \|\Lambda^{\frac{3}{2}}T^\eta\|_{L^2}^2 + \|\Lambda^{\frac{3}{2}}T^\eta\|_{L^2}^2 \right) + \frac{\kappa}{2} \|\Lambda^{\frac{5}{2}}T^\eta\|_{L^2}^2.
\end{aligned}
\end{equation} 
This yields the differential inequality 
\begin{equation}
\frac{d}{dt}\|\Lambda^{\frac{3}{2}}T^\eta\|_{L^2}^2 + \kappa \|\Lambda^{\frac{5}{2}}T^\eta\|_{L^2}^2
\le C\|\na u^\eta\|_{L^2}^2 \left( \|\nabla u^\eta\|_{L^2}^2 \|\Lambda^{\frac{3}{2}}T^\eta\|_{L^2}^2 + \|\Lambda^{\frac{3}{2}}T^\eta\|_{L^2}^2 \right).
\end{equation}

Finally, we combine the energy inequalities for $\|c^\eta_i\|_{L^2}$, $\|\nabla u^\eta\|_{L^2}$, and $\|\Lambda^{\frac32}T^\eta\|_{L^2}$, and we obtain 
\begin{align*}
    &\frac{d}{dt} \left(1+\sum_{i=1}^N\|c^\eta_i\|_{L^2}^2 + \|\nabla u^\eta\|_{L^2}^2 + \|\Lambda^{\frac32}T^\eta\|_{L^2}^2 \right) + \frac12\left(\sum_{i=1}^N D_i\|\na c^\eta_i\|_{L^2}^2  + \nu \|\Delta u^\eta \|_{L^2}^2 + \kappa \|\Lambda^{\frac52}T^\eta\|_{L^2}^2\right) 
    \\
    \leq &C\|\na T^\eta\|_{L^2}^2 + C\left(1+\sum_{i=1}^N\|c^\eta_i\|_{L^2}^2\right)^3
+C(1+\|\l^{\frac{3}{2}} T^\eta\|_{L^2}^{2})^3 \left(1+\sum_{i=1}^N\|c_i^\eta\|_{L^2}^2\right)^2 
\\
&+C(\|\nabla u^\eta \|_{L^2}^4+\|T^\eta_0 - T^\eta_r\|_{L^2}^2) + C\|\na u^\eta\|_{L^2}^2 \left( \|\nabla u^\eta\|_{L^2}^2 \|\Lambda^{\frac{3}{2}}T^\eta\|_{L^2}^2 + \|\Lambda^{\frac{3}{2}}T^\eta\|_{L^2}^2 \right).
\end{align*}

Denoting by 
\begin{align*}
    &y(t) = 1+\sum_{i=1}^N\|c^\eta_i\|_{L^2}^2 + \|\nabla u^\eta\|_{L^2}^2 + \|\Lambda^{\frac32}T^\eta\|_{L^2}^2,
    \\
    &z(t) = \sum_{i=1}^N D_i\|\na c^\eta_i\|_{L^2}^2  + \nu \|\Delta u^\eta \|_{L^2}^2 + \kappa \|\Lambda^{\frac52}T^\eta\|_{L^2}^2,
\end{align*}
we have
\begin{align}\label{ine:uni-bdd}
    \frac{dy}{dt} + \frac12 z \leq C y^5.
\end{align}
From $\frac{dy}{dt} \leq Cy^5$, we conclude that there exists a time $\mathcal T_1>0$ such that $y(t) \leq 2y(0)$ for any $t\in[0,\mathcal{T}_1]$. Moreover, integrating \eqref{ine:uni-bdd} from $0$ to $\mathcal{T}_1$, we deduce that  $\int_0^{\mathcal{T}_1} z(t) dt \leq 2y(0)+2C\mathcal T_1(2y(0))^5<\infty$. Therefore, it follows that for each $i=1,2,\dots,N$, $\{c_i^\eta\}$ (and thus $\{\rho^\eta=\sum_{i=1}^N ez_i c_i^\eta\}$) are uniformly bounded in $L^\infty(0,\mathcal T_1; L^2) \cap L^2(0,\mathcal T_1; H^1)$, $\{u^\eta\}$ are uniformly bounded in $L^\infty(0,\mathcal T_1; V) \cap L^2(0,\mathcal T_1; \mathcal D(A))$, and $\{T^\eta\}$ are uniformly bounded in $L^\infty(0,\mathcal T_1; H^{\frac32}) \cap L^2(0,\mathcal T_1; H^{\frac52})$.

Finally, it holds that 
\begin{align*}
    &\frac{1}{2}\frac{d}{dt}\|\Lambda^{\frac{5}{2}}T^\eta\|_{L^2}^2 + \kappa \|\Lambda^{\frac{7}{2}}T^\eta\|_{L^2}^2
\\= &- \int_{\TT^2}\Lambda^{\frac32} (\j u^\eta \cdot \na T^\eta) \Lambda^{\frac72}T^\eta dx
\le \|\Lambda^{\frac32} (\j u^\eta \cdot \na T^\eta)\|_{L^{2}} \|\Lambda^{\frac72}T^\eta\|_{L^2}
\\
\leq &C\|\Lambda^{\frac32} u^\eta\|_{L^2} \|\Lambda^{\frac52} T^\eta\|_{L^2} \|\Lambda^{\frac72}T^\eta\|_{L^2} \leq C\|\Lambda^{\frac32} u^\eta\|_{L^2}^2 \|\Lambda^{\frac52} T^\eta\|_{L^2}^2 + \frac\kappa2 \|\Lambda^{\frac72}T^\eta\|_{L^2}^2,
\end{align*}
where we have used the fact that $H^\frac32$ is a Banach algebra.  Consequently,
\begin{align*}
    \frac{d}{dt}\|\Lambda^{\frac{5}{2}}T^\eta\|_{L^2}^2 + \kappa \|\Lambda^{\frac{7}{2}}T^\eta\|_{L^2}^2 \leq C\|\Lambda^{\frac32} u^\eta\|_{L^2}^2 \|\Lambda^{\frac52} T^\eta\|_{L^2}^2.
\end{align*}
By the Gr\"onwall inequality, and due to the fact that $u^\eta\in  L^\infty(0,\mathcal T_1; V) \cap L^2(0,\mathcal T_1; \mathcal D(A))$, we conclude that $\{T^\eta\}$ are uniformly bounded in $L^\infty(0,\mathcal T_1; H^{\frac52}) \cap L^2(0,\mathcal T_1; H^{\frac72})$.

\end{proof}

{\bf Step 2.}
For the rest of this section, in addition to the assumptions in Proposition~\ref{prop:uni-est-sol}, we assume that $c_{i0}\geq a_0>0$. With this assumption, we have a strictly positive lower bound of $c_i^\eta$ for each $\eta\in(0,1)$. Indeed, by matching \eqref{eqeq} and \eqref{N-NPNS-mo-1}, we know $c_i^\eta$ satisfies \eqref{eqeq} with
\begin{equation}\label{choice-positivity}
    C_1 \mathcal F = D_i\frac{z_i e}{k_B T^\eta} \nabla \j \phi^\eta, \quad C_2 \mathcal G = D_i \j\nabla \log T^\eta, \quad \text{and} \quad C_3=0.
\end{equation}
    As $(c_i^\eta,u^\eta,T^\eta)\in L^\infty(0,\mathcal T; H^m)$ for any $m\in\mathbb N$ and $\mathcal T>0$, we have that condition \eqref{positive-condition} holds for any $\mathcal T>0$, and in particular for $\mathcal T=\mathcal T_1$. By applying Proposition~\ref{prop:pos-lower-bdd} we conclude that 
    \begin{equation}\label{eqn:non-uni-lower-bound}
        c_i^\eta(x,t)\geq a_\eta>0
    \end{equation}
   for all $t\in[0,\mathcal T_1]$ and a.e. $x\in\mathbb T^2$. Note that as we do not have uniform-in-$\eta$ bound of $\int_0^{\mathcal T_1}\|\nabla\cdot \mathcal F\|_{L^\infty} dt$, 
   the bound in \eqref{positive-condition} will depend on $\eta$, and therefore $a_\eta$ also depends on $\eta$. Nonetheless, with the positive lower bound, one can guarantee that the term $\log c_i^\eta$ makes sense.

    Next, we derive uniform-in-$\eta$ local bounds for the $L^2$ norms of $\log c_i^\eta$.

    \begin{prop}\label{prop:uni-est-sol-log}
        Suppose the assumption of Proposition~\ref{prop:uni-est-sol} holds, and in addition, assume $c_{i0} \geq a_0>0.$ Then there exists a time $\mathcal T_2>0$ such that $\log c_i^\eta$ are uniformly bounded in $L^\infty(0,\mathcal T_2; L^{2}) \cap L^2(0,\mathcal T_2;  H^{1})$ with $\mathcal T_2 \leq \mathcal T_1$. {Moreover, $\mathscr D_i^\eta := - D_i \int_{\left\{c_i^{\eta} \le 1\right\}} \frac{|\na c_i^{\eta}|^2}{(c_i^{\eta})^2} \log c_i^{\eta}\geq 0$ is uniformly bounded in $L^1(0,\mathcal T_2)$.}
    \end{prop}

    \begin{proof}
    %Since each $c_i^{\eta}$ is bounded from below by a positive constant, the strict positivity of the family of $\eta$-regularized concentrations $c_i^{\eta}$ follows, and thus $\log c_i^{\eta}$ is well-defined for all $i \in \left\{1, \dots, N\right\}$ and any $\eta > 0$. 
    The time evolution of $\log c_i^{\eta}$ is described by 
    \be 
    \begin{aligned}
&\pa_t \log c_i^{\eta} + \mathcal{J}_{\eta} u^{\eta} \cdot \na \log c_i^{\eta} - D_i \frac{\Delta c_i^{\eta}}{c_i^{\eta}} \\&\quad\quad= \frac{D_iz_i e}{k_B} \na \cdot \left(\frac{c_i^{\eta}\nabla\mathcal{J}_{\eta}\phi^{\eta}}{T^{\eta}} \right)\frac{1}{c_i^{\eta}}
+ D_i \na \cdot \left( \chi^\eta (c_i^{\eta}) c_i^\eta \log c_i^\eta \mathcal{J}_{\eta}  \nabla \log  T^\eta\right) \frac{1}{c_i^{\eta}}.
    \end{aligned}
    \ee Multiply the latter by $\log c_i^{\eta}$ and integrate spatially over $\TT^2$. In view of the divergence-free property satisfied by $J_{\eta}u^{\eta}$, the cancellation law
    \be 
\int_{\TT^2} \mathcal{J}_{\eta}u^{\eta} \cdot \na \log c_i^{\eta} \log c_i^{\eta} = 0
    \ee holds, and thus we obtain the energy equality
\be
\begin{aligned}
&\frac{1}{2} \frac{d}{dt}\|\log c_i^{\eta}\|_{L^2}^2 - D_i \int_{\TT^2} \frac{\Delta c_i^{\eta}}{c_i^{\eta}} \log c_i^{\eta} 
\\&\quad= \frac{D_iz_i e}{k_B} \int_{\TT^2} \na \cdot \left(\frac{c_i^{\eta}\nabla\mathcal{J}_{\eta}\phi^{\eta}}{T^{\eta}} \right)\frac{1}{c_i^{\eta}} \log c_i^{\eta}
+ D_i \int_{\TT^2} \na \cdot \left( \chi^\eta (c_i^{\eta}) c_i^\eta \log c_i^\eta \mathcal{J}_{\eta}  \nabla \log  T^\eta\right) \frac{1}{c_i^{\eta}} \log c_i^{\eta}.
\end{aligned}
\ee Integration by parts and differentiation yield the identity
\be 
\begin{aligned}
-D_i \int_{\TT^2}\frac{\Delta c_i^{\eta}}{c_i^{\eta}} \log c_i^{\eta} 
&= D_i\int_{\TT^2} \na c_i^{\eta} \cdot \frac{\na c_i^{\eta}}{(c_i^{\eta})^2} - D_i\int_{\TT^2} \na c_i^{\eta} \cdot \na c_i^{\eta} \frac{\log c_i^{\eta}}{(c_i^{\eta})^2}
\\&= D_i \|\na \log c_i^{\eta}\|_{L^2}^2
- D_i\int_{\TT^2} \frac{|\na c_i^{\eta}|^2}{(c_i^{\eta})^2} \log c_i^{\eta}.
\end{aligned}
\ee Splitting the domain of integration into $\TT^2 \cap \left\{c_i^{\eta} \le 1\right\}$ and $\TT^2 \cap \left\{c_i^{\eta} > 1 \right\}$, we have  
\be 
-D_i \int_{\TT^2}\frac{\Delta c_i^{\eta}}{c_i^{\eta}} \log c_i^{\eta} 
= D_i \|\na \log c_i^{\eta}\|_{L^2}^2 - D_i \int_{\left\{c_i^{\eta} \le 1\right\}} \frac{|\na c_i^{\eta}|^2}{(c_i^{\eta})^2}\log c_i^{\eta} - D_i \int_{\left\{c_i^{\eta} > 1\right\}} \frac{|\na c_i^{\eta}|^2}{(c_i^{\eta})^2}\log c_i^{\eta}.
\ee We note that 
\be 
{\mathscr D_i^\eta:=}- D_i \int_{\left\{c_i^{\eta} \le 1\right\}} \frac{|\na c_i^{\eta}|^2}{(c_i^{\eta})^2}\log c_i^{\eta} \ge 0
\ee whereas 
\be 
\left|D_i \int_{\left\{c_i^{\eta} > 1\right\}} \frac{|\na c_i^{\eta}|^2}{(c_i^{\eta})^2}\log c_i^{\eta}\right|
\le D_i\|\na c_i^{\eta}\|_{L^2}^2.
\ee Now we estimate the electromigration term. Integrating by parts gives
\be 
\begin{aligned}
&\frac{D_i z_i e}{k_B} \int_{\TT^2} \na \cdot \left(\frac{c_i^{\eta} \nabla\mathcal{J}_{\eta} \phi^{\eta}}{T^{\eta}} \right) \frac{1}{c_i^{\eta}} \log c_i^{\eta} 
\\&\quad\quad= \frac{D_i z_i e}{k_B} \int_{\TT^2} \frac{\log c_i^{\eta} \na  \mathcal{J}_{\eta} \phi^{\eta} \cdot \na \log c_i^{\eta}}{T^{\eta}}
- \frac{D_i z_i e}{k_B}\int_{\TT^2} \frac{\na  \mathcal{J}_{\eta} \phi^{\eta} \cdot \na \log c_i^{\eta}}{T^{\eta}},
\end{aligned}
\ee which can be bounded by
\be 
\begin{aligned}
    &\left|\frac{D_i z_i e}{k_B} \int_{\TT^2} \na \cdot \left(\frac{c_i^{\eta} \nabla \mathcal{J}_{\eta} \phi^{\eta}}{T^{\eta}} \right) \frac{1}{c_i^{\eta}} \log c_i^{\eta}  \right|
    \\&\le \frac{D_i}{16} \|\na \log c_i^{\eta}\|_{L^2}^2 
    + C  \left(\|\log c_i^{\eta}\|_{L^4}^2 \|\na \mathcal{J}_{\eta} \phi^{\eta}\|_{L^{4}}^2 + \|\na  \mathcal{J}_{\eta} \phi^{\eta}\|_{L^{2}}^2 \right)
    \\&\le \frac{D_i}{16}\|\na \log c_i^{\eta}\|_{L^2}^2
    + C  (\|\log c_i^{\eta}\|_{L^2}\|\na \log c_i^{\eta}\|_{L^2} + \|\log c_i^{\eta}\|_{L^2}^2+1 )\|\rho^{\eta}\|_{L^2}^2
    \\&\le \frac{D_i}{8} \|\na \log c_i^{\eta}\|_{L^2}^2 + C (\|\log c_i^{\eta}\|_{L^2}^2+1) (\|\rho^{\eta}\|_{L^2}^4 + 1)
\end{aligned}
\ee in view of H\"older's inequality, the uniform boundedness of the regularized temperatures from below, Ladenzhenskaya's interpolation inequality, the uniform boundedness of mollifiers in $L^4$, elliptic regularity estimates, and Young's inequality for products. As for the logarithmic term, it holds that 
\be 
\begin{aligned}
  &D_i \int_{\TT^2} \na \cdot \left( \chi^\eta (c_i^{\eta}) c_i^\eta \log c_i^\eta \mathcal{J}_{\eta}  \nabla \log  T^\eta\right) \frac{1}{c_i^{\eta}} \log c_i^{\eta} 
  \\&= D_i \int_{\TT^2} \chi^\eta (c_i^{\eta}) c_i^\eta \log c_i^\eta \mathcal{J}_{\eta}  \nabla \log  T^\eta  \cdot \left(\frac{\na c_i^{\eta}}{(c_i^{\eta})^2}\log c_i^{\eta} - \frac{\na c_i^{\eta}}{(c_i^{\eta})^2} \right).
\end{aligned}
\ee By applying H\"older and Young inequalities, interpolating in Sobolev spaces, and using the continuous Sobolev embedding of $H^{\frac{1}{2}}$ in $L^4$, we estimate 
\be 
\begin{aligned}
&\left|D_i \int_{\TT^2} \chi^\eta (c_i^{\eta}) c_i^\eta \log c_i^\eta \mathcal{J}_{\eta}  \nabla \log  T^\eta \cdot \frac{\na c_i^{\eta}}{(c_i^{\eta})^2}  \right|
\\&\le D_i \int_{\TT^2} |\log c_i^{\eta}| |\mathcal{J}_{\eta} \na \log T^{\eta}| |\na \log c_i^{\eta}|
\\&\le \frac{D_i}{16}\|\na \log c_i^{\eta}\|_{L^2}^2 + C\|\na \log T^{\eta}\|_{L^4}^2 \|\log c_i^{\eta}\|_{L^4}^2
\\&\le \frac{D_i}{16}\|\na \log c_i^{\eta}\|_{L^2}^2 + C(T^*)^{-2}\|\na T^{\eta}\|_{L^4}^2 \left(\|\log c_i^{\eta}\|_{L^2}^2 +\|\log c_i^{\eta}\|_{L^2}\|\na\log c_i^{\eta}\|_{L^2} \right)
\\&\le \frac{D_i}{8} \|\na \log c_i^{\eta}\|_{L^2}^2 + C(\|T^{\eta}\|_{H^{\frac{3}{2}}}^4 +1)(\|\log c_i^{\eta}\|_{L^2}^2 + 1).
\end{aligned}
\ee In order to estimate the quadratic logarithmic term, we employ again a  splitting technique by which we decompose the domain of integration into two regions $\left\{c_i^{\eta} \le 1\right\}$ and $\left\{c_i^{\eta} > 1\right\}$, and we obtain 
\be 
\begin{aligned}
&\left|D_i \int_{\TT^2} \chi^\eta (c_i^{\eta}) c_i^\eta \log c_i^\eta \mathcal{J}_{\eta}  \nabla \log  T^\eta  \cdot \frac{\na c_i^{\eta}}{(c_i^{\eta})^2}\log c_i^{\eta}\right|
\le D_i\int_{\TT^2}|\na \log c_i^{\eta}|(\log c_i^{\eta})^2|\mathcal{J}_{\eta} \nabla \log T^{\eta}|
\\&\le D_i\int_{\left\{c_i^{\eta} \ge 1 \right\}} |\na \log c_i^{\eta}|(\log c_i^{\eta})^2|\mathcal{J}_{\eta}\na \log T^{\eta}| + D_i\int_{\left\{c_i^{\eta} \le 1 \right\}}|\na \log c_i^{\eta}|(\log c_i^{\eta})^2|\mathcal{J}_{\eta} \na \log T^{\eta}|.
\end{aligned}
\ee Over the set $\left\{c_i^{\eta} \ge 1 \right\}$, we make use of the logarithmic inequality $|\log c_i^{\eta}| \le C(c_i^{\eta})^{\frac{1}{8}}$ and bound
\be 
\begin{aligned}
D_i\int_{\left\{c_i^{\eta} \ge 1 \right\}} |\na \log c_i^{\eta}|(\log c_i^{\eta})^2|\mathcal{J}_{\eta}\na \log T^{\eta}|
&\le \frac{D_i}{8}\|\na \log c_i^{\eta}\|_{L^2}^2 + C\|(c_i^{\eta})^{\frac{1}{4}}\|_{L^4}^2\|\na \log T^{\eta}\|_{L^4}^2
\\&\le \frac{D_i}{8}\|\na \log c_i^{\eta}\|_{L^2}^2 + C\|c_i^{\eta}\|_{L^1}^{\frac{1}{2}}\| T^{\eta}\|_{H^{\frac{3}{2}}}^2.
\end{aligned}
\ee Over the set $\left\{c_i^{\eta} <1 \right\}$, we use the fact that $-\log c_i^{\eta} = |\log c_i^{\eta}|$ and deduce that 
\be 
\begin{aligned}
&D_i\int_{\left\{c_i^{\eta} \le 1 \right\}}|\na \log c_i^{\eta}|(\log c_i^{\eta})^2|\mathcal{J}_{\eta} \na \log T^{\eta}|
= D_i\int_{\left\{c_i^{\eta} \le 1 \right\}}|\na \log c_i^{\eta}| |\log c_i^{\eta}|^{\frac{3}{2}} |\log c_i^{\eta}|^{\frac{1}{2}}|\mathcal{J}_{\eta} \na \log T^{\eta}|. 
\\&\quad\quad\le -\frac{D_i}{4} \int_{\left\{c_i^{\eta}\le 1 \right\}} \log c_i^{\eta} |\na \log c_i^{\eta}|^2 + C\int_{\left\{c_i^{\eta} \le 1\right\}} |\log c_i^{\eta}|^3 |\mathcal{J}_{\eta} \na \log T^{\eta}|^2.
\end{aligned}
\ee But an application of the Ladyzhenskaya interpolation inequality gives
\be 
\begin{aligned}
&C\int_{\left\{c_i^{\eta} \le 1\right\}} |\log c_i^{\eta}|^3 |\mathcal{J}_{\eta} \na \log T^{\eta}|^2
\le C\|\log c_i^{\eta}\|_{L^4}^3\|\na \log T^{\eta}\|_{L^8}^2
\\&\le C\left(\|\log c_i^{\eta}\|_{L^2}^{\frac{3}{2}} \|\na \log c_i^{\eta}\|_{L^2}^{\frac{3}{2}} + \|\log c_i^{\eta}\|_{L^2}^3 \right) \|\na \log T^{\eta}\|_{L^8}^2
\\&\le \frac{D_i}{8}\|\na \log c_i^{\eta}\|_{L^2}^2 + C(\|\log c_i^{\eta}\|_{L^2}^6 +1)(\|T^{\eta}\|_{H^{2}}^8+1)
\end{aligned}
\ee Putting all these estimates together, we infer that 
\be \label{logci-est}
\begin{aligned}
&\frac{d}{dt} \|\log c_i^{\eta}\|_{L^2}^2
+ D_i \|\na \log c_i^{\eta}\|_{L^2}^2 +{\mathscr D_i^\eta}
\\&\quad\quad\le C(\|\log c_i^{\eta}\|_{L^2}^6 +1)(\|T^{\eta}\|_{H^{2}}^8+ \|\rho^{\eta}\|_{L^2}^4 + \|c_i^{\eta}\|_{L^1}+ 1) {+2D_i\|\nabla c_i^\eta\|_{L^2}^2}.
\end{aligned}
\ee Let
\be 
\mathcal{M} = \sup\limits_{0<\eta<1 } \sup\limits_{t \in [0, \mathcal T_1]} \left(\|T^{\eta}\|_{H^{2}}^8+ \|\rho^{\eta}\|_{L^2}^4 + \|c_i^{\eta}\|_{L^1}+1\right).
\ee Then $1 \le \mathcal{M} < \infty$ as a consequence of the uniform boundedness of the regularized concentrations in $L^{\infty}(0,\mathcal T_1; L^2)$ and temperatures in $L^{\infty}(0,\mathcal T_1; H^{2})$. Therefore, we deduce that the $L^2$ norm of $\log c_i^{\eta}$ obeys the differential inequality 
\be 
\frac{d}{dt}\left(\|\log c_i^{\eta}\|_{L^2}^2 + \mathcal{M}\right) \le C\left(\mathcal{M} + \|\log c_i^{\eta}\|_{L^2}^2\right)^{{4}}{+2D_i\|\nabla c_i^\eta\|_{L^2}^2}
\ee on $[0, \mathcal T_1]$. 
In view of \eqref{ine:uni-bdd}, the gradients of the regularized concentrations are uniformly bounded in $\eta$ in $L^{2}(0, \mathcal T_1, L^2)$. Thus, we can apply Lemma~\ref{newloc} and Lemma~\ref{molb} to infer that 
\be 
\|\log c_i^{\eta}(t)\|_{L^2}^2 \le K
\ee  on a short time interval $[0, \mathcal{T}_2]$, where $\mathcal{T}_2\leq \mathcal T_1$ is independent of $\eta$, and $K$ depends only on the initial data and the parameters, but is independent of $\eta$ and time. 
Therefore, we deduce that $\log c_i^{\eta}$ is uniformly bounded in $\eta$ in the Lebesgue space $L^{\infty}(0,\mathcal{T}_2, L^2)$. {In addition, by integrating \eqref{logci-est} from $0$ to $\mathcal T_2$ and using uniform bounds of $\log c_i^\eta$, we conclude that $\mathscr D_i^\eta$ are uniformly bounded in $L^1(0,\mathcal T_2)$.}
\end{proof}

\begin{rem}
The evolution of $\log c_i^{\eta}$ in $L^2$ gives rise to an interesting dissipative structure dominated by the term $\mathscr D_i^\eta := - D_i \int_{\left\{c_i^{\eta} \le 1\right\}} \frac{|\na c_i^{\eta}|^2}{(c_i^{\eta})^2} \log c_i^{\eta}$, which is locally integrable in time. The latter is crucial to obtain good control of the $c_i^{\eta} \log c_i^{\eta}$ term in the evolution of the $H^1$ norm of the regularized concentrations, as shown in the subsequent step.  
\end{rem}

{\bf Step 3.}
 With the uniform bounds on $\log c_i^\eta$, we can proceed to establish  $L_t^{\infty}H_x^1$ and $L_t^2 H_x^2$ bounds for $c_i^\eta$ provided that $c_{i0}\in H^1$. 

\begin{prop}\label{prop:uni-est-sol-cihigh}
     Suppose the assumptions of Proposition~\ref{prop:uni-est-sol-log} hold, and in addition, assume $c_{i0}\in H^1$. Then there exists a time $\mathcal T_0>0$ such that $c_i^\eta$ are uniformly bounded in $L^\infty(0,\mathcal T_0; H^1) \cap L^2(0,\mathcal T_0;  H^{2})$ with $\mathcal T_0 \leq \mathcal T_2$.
\end{prop}
\begin{proof}
      We multiply the $\eta$-regularized ionic concentration equation by $-\Delta c_i^{\eta}$ and we integrate over $\TT^2$. We obtain  the evolution equation  
    \be 
    \begin{aligned}
\frac{1}{2}\frac{d}{dt} \|\na c_i^{\eta}\|_{L^2}^2 + D_i \|\Delta c_i^{\eta}\|_{L^2}^2
&=  \int_{\TT^2} \mathcal{J}_{\eta} u^{\eta} \cdot \na c_i^{\eta} \Delta c_i^{\eta}  -\frac{z_iD_ie}{k_B}\int_{\TT^2} \na \cdot \left( \frac{1}{T^{\eta}} c_i^{\eta} \na \mathcal{J}_{\eta} \phi^{\eta} \right) \Delta c_i^{\eta}
\\&\quad\quad-D_i \int_{\TT^2} \na \cdot \left(\chi^\eta (c_i^{\eta}) c_i^\eta \log c_i^\eta \mathcal{J}_{\eta}  \nabla \log  T^\eta
\right)\Delta c_i^{\eta}.
\end{aligned}
    \ee
In order to estimate the logarithmic term, we apply the divergence product rule to expand the following term, 
\be 
\begin{aligned}
 \na \cdot \left(\chi^\eta (c_i^{\eta}) c_i^\eta \log c_i^\eta \mathcal{J}_{\eta}  \nabla \log  T^\eta \right)
 &= \chi^{\eta} (c_i^{\eta}) c_i^{\eta} \log c_i^{\eta} \mathcal{J}_{\eta}  \Delta \log T^{\eta} 
 \\&\quad\quad+ (\chi^{\eta})'(c_i^{\eta}) \na c_i^{\eta} c_i^{\eta} \log c_i^{\eta} \cdot \mathcal{J}_{\eta} \na \log T^{\eta}
 \\&\quad\quad\quad+ \chi^{\eta}(c_i^{\eta}) \na c_i^{\eta} \log c_i^{\eta} \cdot \mathcal{J}_{\eta} \na \log T^{\eta}
 \\&\quad\quad\quad\quad+\chi^{\eta}(c_i^{\eta}) \na c_i^{\eta}  \cdot \mathcal{J}_{\eta} \na \log T^{\eta},
 \end{aligned}
\ee then we bound the latter in $L^2$ using H\"older's inequality and the boundedness of mollifier in $L^p$ spaces (with $1 \le p \le \infty$) to obtain  
\be 
\begin{aligned}
&\| \na \cdot \left(\chi^\eta (c_i^{\eta}) c_i^\eta \log c_i^\eta \mathcal{J}_{\eta}  \nabla \log  T^\eta \right)\|_{L^2}
\\&\quad\quad\le C\|c_i^{\eta} \log c_i^{\eta}\|_{L^4} \|\Delta \log T^{\eta}\|_{L^4}
\\&\quad\quad\quad\quad+ C\|\na c_i^{\eta} \log c_i^{\eta}\|_{L^2} \|\na \log T^{\eta}\|_{L^{\infty}}
+ C\|\na c_i^{\eta}\|_{L^2} \|\na \log T^{\eta}\|_{L^{\infty}}
\\
&\quad\quad=:\mathcal{I}_1 + \mathcal{I}_2 + \mathcal{I}_3,
\end{aligned}
\ee 
where we have used the bound $|(\chi^\eta)'(c_i^\eta) c_i^\eta| \leq C$, which follows from the construction of $\chi^\eta$ given by \eqref{chi-1}.
 We define the quantity $\mathcal{A}$ to be
\be 
\mathcal{A} = \sup\limits_{0<\eta<1} \sup\limits_{0 \le t \le \mathcal{T}_2} \left[\sum\limits_{i=1}^{N} \left(\|c_i^{\eta}\|_{L^2}^2  + \|\log c_i^{\eta}\|_{L^2}^2\right) +  \|T^{\eta}\|_{H^{\frac{5}{2}}}^2 + \|u^{\eta}\|_{H^1}^2\right],
\ee where $\mathcal{T}_2$ is the local time constructed in Proposition \ref{prop:uni-est-sol-log}. 
In view of the temperature bound
\be 
\begin{aligned}
\|\Delta \log T^{\eta}\|_{L^4}
= \left\|\frac{\Delta T^{\eta}}{T^{\eta}} - \frac{|\na T^{\eta}|^2}{(T^{\eta})^2} \right\|_{L^4}
\le C\|\Delta T^{\eta}\|_{L^4} + C\|\na T^{\eta}\|_{L^8}^2
\le C(1+\|T^{\eta}\|_{H^{\frac{5}{2}}}^2) \le C(1+\mathcal{A})
\end{aligned}
\ee that holds due to the lower boundedness of the regularized temperatures by $T^*$, classical Sobolev embeddings, and Young's inequality, and the concentration bound
\be \label{wei}
\begin{aligned}
&\|c_i^{\eta} \log c_i^{\eta}\|_{L^4}
\le C(1 + \|(c_i^{\eta})^{2}\|_{L^4})
\le C(1 + \|c_i^{\eta}\|_{L^8}^2)
\le C(1+ \|c_i^{\eta}\|_{H^1}^2)
\le C(1+ \mathcal{A} + \|\na c_i^{\eta}\|_{L^2}^2)
\end{aligned}
\ee that follows from the logarithmic estimate $|x \log x| \le C(1+x^2)$ for $x \ge 0$ and the continuous embedding of $H^1$ in $L^8$ , we infer that 
\be 
\mathcal{I}_1 \le C(1+\mathcal{A})(1+ \mathcal{A} + \|\na c_i^{\eta}\|_{L^2}^2) 
\ee 
% The treatment of the term $\mathcal{I}_2$ is somehow similar. Indeed, we make use of the temperature estimate 
% \be 
% \|\na \log T^{\eta}\|_{L^{\infty}}
% \le \left\|\frac{\na T^{\eta}}{T^{\eta}}\right\|_{L^{\infty}}
% \le C\left\|\na T^{\eta}\right\|_{L^{\infty}}
% \le C\|T^{\eta}\|_{H^{\frac{5}{2}}} \le C\sqrt{\mathcal{A}},
% \ee the concentration estimate \eqref{wei}, and the Ladyzhenskaya interpolation inequality applied to the mean-free vector field $\na c_i$ to deduce that 
% \be 
% \mathcal{I}_2 \le C\sqrt{\mathcal{A}}(1+ \mathcal{A} + \|\na c_i^{\eta}\|_{L^2}^2)\|\na c_i^{\eta}\|_{L^2}^{\frac{1}{2}}\|\Delta c_i^{\eta}\|_{L^2}^{\frac{1}{2}}. 
% \ee We point out that the power of $L^2$ norm of $\Delta c_i$ does not reach 1, which allows control by the dissipation at the last stage. 
In order to estimate the term $\mathcal{I}_2$, we split the domain of integration into two parts: $\left\{c_i^{\eta} < 1 \right\}$ and $\left\{c_i^{\eta} \ge 1 \right\}$. On $\left\{c_i^{\eta} < 1 \right\}$, it holds that 
\be 
\int_{ \left\{c_i^{\eta} < 1 \right\}} |\na c_i^{\eta} \log c_i^{\eta}|^2 
= -\int_{ \left\{c_i^{\eta} < 1 \right\}} |\na c_i^{\eta}|^2 \log c_i^{\eta} \log\frac{1}{c_i^{\eta}} 
\le -\int_{ \left\{c_i^{\eta} <1 \right\}} \frac{|\na c_i^{\eta}|^2}{|c_i^{\eta}|^2} \log c_i^\eta
\ee in view of the identity $-\log c_i^{\eta} = \log \frac{1}{c_i^{\eta}}$ and the nonnegativity of the quantity $-|\na c_i^{\eta}|^2 \log c_i^{\eta}$ on the integration domain $\left\{c_i^{\eta} < 1 \right\}$. 
On $\left\{c_i^{\eta} \ge 1 \right\}$, we have 
\be 
\int_{ \left\{c_i^{\eta} \ge 1 \right\}} |\na c_i^{\eta} \log c_i^{\eta}|^2 
\le C\int_{ \left\{c_i^{\eta} \ge 1 \right\}} |\na c_i^{\eta}|^2  |c_i^{\eta}| \le C\|c_i^{\eta}\|_{L^2}\|\na c_i^{\eta}\|_{L^4}^2
\le C\sqrt{\mathcal{A}}\|\na c_i^{\eta}\|_{L^2}\|\Delta c_i^{\eta}\|_{L^2}.
\ee Consequently, the term $\mathcal{I}_2$ can be bounded by 
\be 
\mathcal{I}_2
\le C\sqrt{\mathcal{A}} \left(\mathcal{A}^{\frac{1}{4}}\|\na c_i^{\eta}\|_{L^2}^{\frac{1}{2}} \|\Delta c_i^{\eta}\|_{L^2}^{\frac{1}{2}} + \sqrt{- \int_{\left\{c_i^{\eta} <1 \right\}} \frac{|\na c_i^{\eta}|^2}{|c_i^{\eta}|^2} \log c_i^{\eta}}\right).
\ee As for the last term, we have 
\be 
\mathcal{I}_3\le C\sqrt{\mathcal{A}} \|\na c_i^{\eta}\|_{L^2}.
\ee Applying the Cauchy-Schwarz inequality followed by several applications of Young's inequality gives rise to 
\be 
\begin{aligned}
&D_i \int_{\TT^2} \na \cdot \left(\chi^\eta (c_i^{\eta}) c_i^\eta \log c_i^\eta \mathcal{J}_{\eta}  \nabla \log  T^\eta
\right)\Delta c_i^{\eta}
\\&\le D_i \|\na \cdot \left(\chi^\eta (c_i^{\eta}) c_i^\eta \log c_i^\eta \mathcal{J}_{\eta}  \nabla \log  T^\eta
\right)\|_{L^2}\|\Delta c_i^{\eta}\|_{L^2}
\\&\le \frac{D_i}{8}\|\Delta c_i^{\eta}\|_{L^2}^2
+ C(1 +\mathcal{A}^{10})(1+ \|\na c_i^{\eta}\|_{L^2}^{10})  - C\mathcal{A} \int_{\left\{c_i^{\eta} <1 \right\}} \frac{|\na c_i^{\eta}|^2}{|c_i^{\eta}|^2} \log c_i^{\eta}.
\end{aligned}
\ee The estimates for the advection and electromigration terms are more classical. In fact, using H\"older, Ladyzhenskaya, and Young inequalities, we have 
\be 
\beg{aligned}
&\left|\int_{\TT^2} \mathcal{J}_{\eta} u^{\eta} \cdot \na c_i^{\eta} \Delta c_i^{\eta}\right| 
\le C\|u^{\eta}\|_{L^4}\|\na c_i^{\eta}\|_{L^4} \|\Delta c_i^{\eta}\|_{L^2}
\\&\le \frac{D_i}{8}\|\Delta c_i^{\eta}\|_{L^2}^2 + C\|u^{\eta}\|_{H^1}^4 \|\na c_i^{\eta}\|_{L^2}^2
\le \frac{D_i}{8}\|\Delta c_i^{\eta}\|_{L^2}^2 + C\mathcal{A}^2\|\na c_i^{\eta}\|_{L^2}^2 
\end{aligned}
\ee and 
\be 
\beg{aligned}
&\left|\frac{z_iD_ie}{k_B}\int_{\TT^2} \na \cdot \left( \frac{1}{T^{\eta}} c_i^{\eta} \na \mathcal{J}_{\eta} \phi^{\eta} \right) \Delta c_i^{\eta}\right|
\\&\le \frac{D_i}{8} \|\Delta c_i^{\eta}\|_{L^2}^2 + C\|\na \phi^{\eta}\|_{L^{\infty}}^2\|c_i^{\eta}\|_{L^2}^2\|\na T^{\eta}\|_{L^{\infty}}^2 + C\|\na \phi^{\eta}\|_{L^{\infty}}^2\|\na c_i^{\eta}\|_{L^2}^2 + C\|c_i^{\eta}\|_{L^4}^2\|\Delta \phi^{\eta}\|_{L^4}^2
\\&\le \frac{D_i}{8} \|\Delta c_i^{\eta}\|_{L^2}^2 + C\mathcal{A}^2 \sum\limits_{j=1}^{N}\|\na c_j^{\eta}\|_{L^2}^2 + C\|\na c_i^{\eta}\|_{L^2}^2 \sum\limits_{j=1}^{N} \|\na c_j^{\eta}\|_{L^2}^2 + C\left(\mathcal{A} + \|\na c_i^{\eta}\|_{L^2}^2\right)  \sum\limits_{j=1}^{N}\|\na c_j^{\eta}\|_{L^2}^2.
\end{aligned}
\ee Putting all these estimates together, we obtain the differential inequality 
\be 
\frac{d}{dt}\|\na c_i^{\eta}\|_{L^2}^2 + \|\Delta c_i^{\eta}\|_{L^2}^2 
\le  C\left(\mathcal{A}^{10} + 1 \right)\left(\sum\limits_{j=1}^{N} \|\na c_j^{\eta}\|_{L^2}^{10}+1\right)
 - C\mathcal{A} \int_{\left\{c_i^{\eta} <1 \right\}} \frac{|\na c_i^{\eta}|^2}{|c_i^{\eta}|^2} \log c_i^{\eta}.
 \ee For each $t \ge 0$ and $\eta > 0$, we consider the instantaneous energies 
 \be 
\mathcal{E}^{\eta}(t) = \sum\limits_{i=1}^{N} \|\na c_i^{\eta}\|_{L^2}^2 + \mathcal{A}^2 + 1
 \ee and 
 \be 
 \mathscr{D}^{\eta}(t) = -\sum\limits_{i=1}^{N} \int_{\left\{c_i^{\eta} <1 \right\}} \frac{|\na c_i^{\eta}|^2}{|c_i^{\eta}|^2} \log c_i^{\eta}.
 \ee Then it holds that  
 \be 
\frac{d\mathcal{E}^{\eta}}{dt}
\le C (\mathcal{E}^{\eta})^{10}+ C \mathcal{A} \mathscr{D}^{\eta}. 
 \ee Applying Lemma~\ref{newloc} and using the uniform boundedness of $\mathscr D^{\eta}$ in $L^1(0, \mathcal{T}_2)$ that comes from Proposition~\ref{prop:uni-est-sol-log}, we deduce the existence of a time $\mathcal{T}_0 \leq \mathcal T_2$ depending only on the initial data (and independent of $\eta$) such that $c_i^{\eta} \in L^{\infty}(0, \mathcal{T}_0; H^1) \cap L^2(0,\mathcal{T}_0 ;H^2)$.
\end{proof}

{\bf Step 4.}
To conclude this section, we will improve \eqref{eqn:non-uni-lower-bound} to a lower bound that is independent of $\eta$ thanks to the uniform bound obtained in Proposition~\ref{prop:uni-est-sol-cihigh}.
\begin{prop}\label{prop:uniform-positive-bound}
 Suppose that the assumptions of Proposition~\ref{prop:uni-est-sol-log} hold. Let $\mathcal T_0$ be the time appearing in Proposition~\ref{prop:uni-est-sol-cihigh}. Then there exists a constant $a>0$, independent of $\eta$, such that $c_i^\eta\geq a >0$ for $t\in[0,\mathcal T_0]$ and a.e. $x\in\mathbb T^2$.
\end{prop}
\begin{proof}
 Similarly to the proof of \eqref{eqn:non-uni-lower-bound}, we match \eqref{eqeq} and \eqref{N-NPNS-mo-1} and observe that $c_i^\eta$ satisfies \eqref{eqeq} with \eqref{choice-positivity}. Compared to the proof of \eqref{eqn:non-uni-lower-bound} where we do not have the uniform-in-$\eta$ bound of $\int_0^{\mathcal T_0} \|\nabla\cdot \mathcal F\|_{L^\infty} ds$, here we have it thanks to the uniform-in-$\eta$ bound of $c_i^\eta$ in $L^\infty(0,\mathcal T_0, H^1)\cap L^2(0,\mathcal T_0, H^2)$. Consequently, the bound in \eqref{positive-condition} holds uniformly-in-$\eta$, and thus the lower bound $a>0$ is independent of $\eta$.
\end{proof}

\begin{rem}[Choice of $H^{\frac52}$ for $T^\eta$.]\label{rem:T}
    From the proof of Proposition~\ref{prop:uniform-positive-bound}, we see that one needs $\|\Delta T^\eta\|_{L^\infty}$ to be integrable in time with a uniform-in-$\eta$ bound. This suggests one needs to do $H^{2+}$ estimate for $T$, and for simplicity we consider $H^{\frac52}$.
\end{rem}

\section{Local well-posedness} \label{sec5}
In this section, we prove Theorem~\ref{thm:lwp-intro}, concerning the local well-posedness of the non-isothermal NPNS system \eqref{N-NPNS-system-original}. We denote by $\mathcal{T}_0$ the time constructed in Proposition \ref{prop:uni-est-sol-cihigh}.

We divide the proof into two parts: existence of solutions and uniqueness. 

\subsection{Existence of solutions} In order to prove the existence of solutions, we make use of the Aubin-Lions compactness theorem to pass to the limit. 
With the uniform bounds of $(c_i^\eta, u^\eta, T^\eta)$ in Theorem~\ref{thm:uniform-bound}, it is straightforward to show that $\partial_t c^\eta_i$, $\partial_t u^\eta$, and $\partial_t T^\eta$ are uniformly bounded in $L^2(0,\mathcal T_0; L^2)$.

By virtue of the uniform-in-$\eta$ boundedness of the family of solutions $(c_i^\eta, u^\eta, T^\eta)$ and their time derivatives $(\partial_t c_i^\eta, \partial_t u^\eta, \partial_t T^\eta)$, we can apply the Aubin-Lions and Banach-Alaoglu theorems to obtain the existence of a subsequence $\{c_i^{\eta_k},u^{\eta_k},T^{\eta_k}\}$ converging to a limit $\{c_i,u,T\}$ that satisfies the desired regularity \eqref{lwp-regularity}, and the convergence happens in the following sense:
\begin{align}
        &u^{\eta_k} \rightharpoonup u \,\, \text{in} \,\, L^2(0,\mathcal T_0; \mathcal D(A)),\quad  u^{\eta_k}\rightarrow u \,\, \text{in} \,\, L^2(0,\mathcal T_0; V),
        \\
        &T^{\eta_k} \rightharpoonup T \,\, \text{in} \,\, L^2(0,\mathcal T_0; H^{\frac72}),\quad T^{\eta_k}\rightarrow T \,\, \text{in} \,\, L^2(0,\mathcal T_0; H^{3}),
        \\
        &c_i^{\eta_k} \rightharpoonup c_i \,\, \text{in} \,\, 
        L^2(0,\mathcal T_0; H^2),\quad c_i^{\eta_k}\rightarrow c_i \,\, \text{in} \,\, L^2(0,\mathcal T_0; H^1).
     \end{align}
Recall that $T^{\eta_k}\geq T^*>0$ and $c_i^{\eta_k}\geq a >0$, uniformly in $\eta_k$. This together with the strong convergence of $T^{\eta_k}$ to $T$ and $c_i^{\eta_k}$ to $c_i$ imply that $T\geq T^*>0$ and $c_i\geq a >0$ on $[0, \mathcal{T}_0]$. Since $u^{\eta_k}$ has zero mean, it follows that $u$ has zero mean.

The rest of the proof follows the proof of \cite[Theorem 2.1]{abdo2024three}, and the only technical part we need to address is the convergence of the  $c_i \log c_i$ nonlinear term in the evolution of $c_i$. Indeed, we have stronger regularity here, thus the convergence of the remaining terms should be easier. To be more specific, we want to show that for any test function $\psi\in C_c^\infty ([0,\mathcal T_0)\times \TT^2)$,
     \begin{align}\label{convergence-LWP}
         \int_0^{\mathcal T_0} \left\langle \chi^{\eta_k}(c_i^{\eta_k})c_i^{\eta_k} \log c^{\eta_k}_i \nabla \jk \log T^{\eta_k}, \nabla \psi\right\rangle  dt \to \int_0^{\mathcal T_0} \left\langle c_i \log c_i \nabla \log T, \nabla \psi\right\rangle dt
     \end{align}
     as $k \to \infty$. We first notice that as $\lim_{k\to \infty} \eta_k = 0$, there exists some $K\in\mathbb N$ such that for all $k\geq K$ we have $\eta_k \leq \frac{a}{2}$. Then for $k\geq K$ we have $c_i^{\eta_k} \geq a \geq 2{\eta_k}$, and thus $\chi^{\eta_k}(c_i^{\eta_k}) =1$. Therefore, for $k\geq K$, it holds that
     \begin{align*}
         &\left|\left\langle \chi^{\eta_k}(c_i^{\eta_k})c_i^{\eta_k} \log c^{\eta_k}_i \nabla \jk \log T^{\eta_k}-  c_i \log c_i \nabla \log T, \nabla \psi\right\rangle \right|
         \\
         = &\left|\left\langle c_i^{\eta_k} \log c^{\eta_k}_i \nabla \jk \log T^{\eta_k}-  c_i \log c_i \nabla \log T, \nabla \psi\right\rangle \right|
         \\
         \leq &\left|\left\langle (c_i^{\eta_k} - c_i) \log c^{\eta_k}_i \nabla \jk \log T^{\eta_k}, \nabla \psi\right\rangle \right| + \left|\left\langle c_i (\log c_i^{\eta_k} - \log c_i) \nabla \jk \log T^{\eta_k}, \nabla \psi\right\rangle \right|
         \\
         &+\left|\left\langle c_i \log c_i \nabla \jk  (\log T^{\eta_k} - \log T), \nabla \psi\right\rangle \right| + 
         \left|\left\langle c_i \log c_i \nabla ( \jk \log T - \log T), \nabla \psi\right\rangle \right| 
         \\
         := &I_1 + I_2 + I_3+I_4.
     \end{align*}
      By the H\"older and Sobolev inequalities, we can bound each as follows:
      \begin{align*}
          I_1 \leq &\|c_i^{\eta_k}-c_i\|_{L^4} \|\log c_i^{\eta_k}\|_{L^2} \|\nabla \log T^{\eta_k}\|_{L^4} \|\nabla \psi\|_{L^\infty} 
          \\
          \leq &C \|c_i^{\eta_k}-c_i\|_{H^1}\|\log c_i^{\eta_k}\|_{L^2} \|\frac1{T^{\eta_k}}\|_{L^\infty} \|T^{\eta_k}\|_{H^2} \|\nabla \psi\|_{L^\infty},
          \\
          I_2 \leq & \|c_i\|_{L^4} \| \log c_i^{\eta_k} - \log c_i\|_{L^2} \|\nabla \log T^{\eta_k}\|_{L^4} \|\nabla \psi\|_{L^\infty} 
          \\
          \leq & C \|c_i\|_{H^1} \sup_{\xi\geq a}| \log'(\xi)| \|c_i^{\eta_k} - c_i\|_{L^2} \|\frac1{T^{\eta_k}}\|_{L^\infty} \|T^{\eta_k}\|_{H^2} \|\nabla \psi\|_{L^\infty},
          \\
          I_3 \leq & \|c_i\|_{L^4} \|\log c_i\|_{L^4} \|\frac{\nabla T^{\eta_k}}{T^{\eta_k}} - \frac{\nabla T}{T}\|_{L^2} \|\nabla \psi\|_{L^\infty} 
          \\
          \leq &C\|c_i\|_{H^1} (|\log a|+\|c_i\|_{L^2})(\|\frac1{T^{\eta_k}}\|_{L^\infty} \|\nabla T^{\eta_k} -\nabla T\|_{L^2} +  \|\frac1{T}\|_{L^\infty} \|\nabla T\|_{H^{\frac32}} \|T-T^{\eta_k}\|_{L^2})\|\nabla\psi\|_{L^\infty} , 
          \\
          I_4 \leq & C \|c_i\|_{L^4} \|\log c_i\|_{L^4} \|(\jk-I) \nabla \log T\|_{L^2} \|\nabla\psi\|_{L^\infty}
          \\
          \leq & C\|c_i\|_{H^1} (|\log a|+\|c_i\|_{L^2})\|(\jk-I) \nabla \log T\|_{L^2} \|\nabla\psi\|_{L^\infty},
      \end{align*}
      where in $I_2$ we have used the fact that both $c_i^{\eta_k}$ and $c_i$ are bounded below by $a>0$. Using the strong convergence of $c_i^{\eta_k}\to c_i$ and $T^{\eta_k}\to T$, the uniform bounds of $c_i^{\eta_k}$ and $T^{\eta_k}$, the regularity of $c_i$ and $T$, the lower boundedness of $T^{\eta_k}$ and $T$, and the convergence property of $\jk$ together with the fact that $\nabla\log T\in L^2$, we can conclude that 
      \[
      \int_0^{\mathcal T_0} (I_1+I_2+I_3+I_4)(t) dt \to 0 \quad \text{as}\quad k\to \infty.
      \]
      This finishes the proof of \eqref{convergence-LWP}.

\subsection{Uniqueness} Assume there are two solutions $(c_{i,1}, u_1, T_1)$ and $(c_{i,2}, u_2, T_2)$ to system \eqref{N-NPNS-system-original} satisfying \eqref{lwp-regularity}, with initial conditions $(c_{i0,1}, u_{01}, T_{01})$ and $(c_{i0,2}, u_{02}, T_{02})$. Suppose that $c_{i0,j}\geq c>0$ for $j=1,2$. Let $\phi_j$ and $\rho_j$, $j=1,2$, be the corresponding potential and charge density to each solution, i.e., $\rho_j = \sum_{i=1}^N z_i e c_{i,j}$ and $- \varepsilon \Delta \phi_j= \rho_j$. Denote by $(c_i,u,T) = (c_{i,1} - c_{i,2}, u_1 - u_2, T_1 - T_2)$, and let $\rho = \sum_{i=1}^N z_i e c_i$ and $- \varepsilon \Delta \phi= \rho$.
    We have
    \begin{subequations}\label{N-NPNS-system-uniqueness}
\begin{align}
&\partial_t c_i + u_1\cdot\nabla c_i + u\cdot \nabla c_{i,2}
=
D_i\nabla \cdot
\Big(
 \nabla c_i
+
 \frac{z_i e}{k_B } \left(c_i \frac{1}{T_1}\nabla \phi_1 + c_{i,2}\frac1{T_1} \nabla \phi + c_{i,2} \nabla \phi_2 \frac{-T}{T_1T_2}\right)\nonumber
\\
&\hspace{4cm}+(c_{i,1} \log c_{i,1} - c_{i,2} \log c_{i,2})\nabla \log T_1 + c_{i,2} \log c_{i,2}\nabla \log\frac{T_1}{T_2}
\Big), \label{N-NPNS-system-original-1-uniqueness}
 \\
&\partial_t u + u_1\cdot\nabla u + u\cdot\nabla u_2+ \nabla p - \nu \Delta u = -\rho_1\nabla \phi - \rho\nabla\phi_2 + g\alpha_T(T-\overline{T_0}) \vec{k}, \label{N-NPNS-system-original-2-uniqueness}
\\
&\nabla \cdot u = 0, \label{N-NPNS-system-original-3-uniqueness}
\\
&\partial_t T + u_1\cdot\nabla T + u\cdot\nabla T_2- \kappa \Delta T
= 0, \label{N-NPNS-system-original-4-uniqueness}
\end{align}
\end{subequations}
with initial conditions $(c_{i0}, u_0, T_0) = (c_{i0,1}-c_{i0,2}, u_{01}-u_{02}, T_{01}-T_{02})$. Thanks to the regularity of $(c_{i,j}, u_j, T_j)$ and $(\partial_t c_{i,j},\partial_t u_j, \partial_t T_j)$, we can take $L^2$ inner product of \eqref{N-NPNS-system-original-1-uniqueness} with $c_i$, \eqref{N-NPNS-system-original-2-uniqueness} with $u$, and  \eqref{N-NPNS-system-original-4-uniqueness} with $T$ and $-\Delta T$. For $c_i$, we obtain the following
\begin{align*}
    \frac12 \frac{d}{dt} \|c_i\|_{L^2}^2 + D_i \|\nabla c_i\|_{L^2}^2 = &- \langle u\cdot \nabla c_{i,2}, c_i\rangle - D_i \Big\langle \frac{z_i e}{k_B } \left(c_i \frac{1}{T_1}\nabla \phi_1 + c_{i,2}\frac1{T_1} \nabla \phi + c_{i,2} \nabla \phi_2 \frac{-T}{T_1T_2}\right), \nabla c_i\Big\rangle
    \\
    &- D_i\Big\langle  c_i \log c_{i,1} \nabla \log T_1 + c_{i,2} \log\frac{c_{i,1}}{c_{i,2}} \nabla \log T_1 + c_{i,2} \log c_{i,2}\nabla \log\frac{T_1}{T_2} , \nabla c_i\Big\rangle.
\end{align*}
By the H\"older, Sobolev, and Gagliardo-Nirenberg interpolation inequalities, we estimate term by term as follows:
\begin{align*}
    &\langle u\cdot \nabla c_{i,2}, c_i\rangle \leq C\|u\|_{L^4} \|\nabla c_{i,2}\|_{L^2} \|c_i\|_{L^4} \leq C\|\nabla c_{i,2}\|_{L^2}\|u\|_{L^2}^{\frac12}\|\nabla u\|_{L^2}^{\frac12} \|c_i\|_{L^2}^{\frac12} \|c_i\|_{H^1}^{\frac12},
    \\
    &D_i\langle \frac{z_i e}{k_B }c_i \frac{1}{T_1}\nabla \phi_1, \nabla c_i\rangle \leq C\|c_i\|_{L^2}\|\nabla\phi_1\|_{L^\infty} \|\nabla c_i\|_{L^2} \leq
    C\sum_{k=1}^N\|\nabla c_{k,1}\|_{L^2}\|c_i\|_{L^2} \|\nabla c_i\|_{L^2},
    \\
    & D_i\langle \frac{z_i e}{k_B }c_{i,2}\frac1{T_1} \nabla \phi, \nabla c_i\rangle \leq C\|c_{i,2}\|_{L^4}\|\nabla \phi\|_{L^4} \|\nabla c_i\|_{L^2}\leq C\|c_{i,2}\|_{H^1}\sum_{k=1}^N\|c_k\|_{L^2}\|\nabla c_i\|_{L^2},
    \\
     &D_i\langle \frac{z_i e}{k_B } c_{i,2}\nabla\phi_2\frac{T}{T_1T_2},\nabla c_i\rangle \leq C\|c_{i,2}\|_{L^4} \|\nabla \phi_2\|_{L^\infty}\|T\|_{L^4} \|\nabla c_i\|_{L^2} 
     \\&\hspace{4.3cm}\leq C \|c_{i,2}\|_{H^1}\sum_{k=1}^N  \|c_{k,2}\|_{H^1} \|T\|_{H^1}\|\nabla c_i\|_{L^2}.
\end{align*}
For the logarithmic terms, we recall that as $c_{i0,j}\geq c>0$ for $j=1,2$, one has $c_{i,j}\geq a >0$. For the function $f(x) = \log(x)$ with $x\geq a$, we have $|f'|\leq \frac1a$. Therefore,
\begin{align*}
    D_i\Big\langle  c_i \log c_{i,1} \nabla \log T_1, \nabla c_i\Big\rangle &\leq C\|c_i\|_{L^4} \|\log c_{i,1}\|_{L^4} \|\nabla \log T_1\|_{L^\infty} \|\nabla c_i\|_{L^2} 
    \\
    &\leq C\|c_i\|_{L^2}^{\frac12} \|c_i\|_{H^1}^{\frac12} \|\nabla c_i\|_{L^2} (1+\|c_{i,1}\|_{L^2}) \|T_1\|_{H^{\frac52}},
    \\
    D_i\Big\langle c_{i,2} \log\frac{c_{i,1}}{c_{i,2}} \nabla \log T_1  , \nabla c_i\Big\rangle &\leq C\|c_{i,2}\|_{L^4} \frac1a \|c_i\|_{L^4} \|\nabla \log T_1\|_{L^\infty} \|\nabla c_i\|_{L^2}
    \\
    &\leq C\|c_i\|_{L^2}^{\frac12} \|c_i\|_{H^1}^{\frac12} \|\nabla c_i\|_{L^2}\|c_{i,2}\|_{H^1}\|T_1\|_{H^{\frac52}},
    \\
    D_i\Big\langle   c_{i,2} \log c_{i,2}\nabla \log\frac{T_1}{T_2} , \nabla c_i\Big\rangle & = D_i \Big\langle   c_{i,2} \log c_{i,2} \Big(\frac{\nabla T}{T_1} - \frac{T\nabla T_2}{T_1T_2} \Big)  , \nabla c_i\Big\rangle 
    \\
    &\leq C \|c_{i,2}\|_{L^8} \|\log c_{i,2}\|_{L^8} (\|\nabla T\|_{L^4} + \|T\|_{L^4} \|\nabla T_2\|_{L^\infty}) \|\nabla c_i\|_{L^2}
    \\
    &\leq C \|c_{i,2}\|_{H^1} (1+ \|c_{i,2}\|_{L^2}) (\|\nabla T\|_{L^2}^{\frac12} \|\Delta T\|_{L^2}^{\frac12} + \|T\|_{H^1} \|T_2\|_{H^{\frac52}}) \|\nabla c_i\|_{L^2}. 
\end{align*}

For $u$, we have
\begin{align*}
    \frac12 \frac{d}{dt} \|u\|_{L^2}^2 + \nu \|\nabla u\|_{L^2}^2 = -\langle u\cdot\nabla u_2, u\rangle  - \langle \rho_1 \nabla \phi, u\rangle - \langle \rho\nabla \phi_2 , u\rangle + g\alpha_T \langle (T-\overline{T_0})\vec{k}, u\rangle.
\end{align*}
Applications of the H\"older and Sobolev inequalities and classical interpolation yield the following:
\begin{align*}
   &\langle u\cdot\nabla u_2, u\rangle \leq C\|u\|_{L^4}^2 \|\nabla u_2\|_{L^2} \leq C\|u\|_{L^2} \|\nabla u\|_{L^2} \|\nabla u_2\|_{L^2},
   \\
   &\langle \rho_1 \nabla \phi,u\rangle \leq C \|\rho_1\|_{L^4} \|\nabla\phi\|_{L^4} \|u\|_{L^2} \leq C\sum_{k=1}^N \|c_{k,1}\|_{H^1} \sum_{k=1}^N \|c_k\|_{L^2} \|u\|_{L^2},
   \\
   &\langle \rho \nabla \phi_2 ,u\rangle \leq C \|\rho\|_{L^2} \|\nabla\phi_2\|_{L^\infty} \|u\|_{L^2} \leq C\sum_{k=1}^N \|c_k\|_{L^2} \sum_{k=1}^N \|c_{k,2}\|_{H^1}  \|u\|_{L^2},
   \\
   &g\alpha_T \langle (T-\overline{T_0})\vec{k}, u\rangle \leq C\|\nabla T\|_{L^2} \|u\|_{L^2}.
\end{align*}

For $T$, we have
\begin{align*}
    \frac12 \frac{d}{dt} \|T\|_{L^2}^2 + \kappa \|\nabla T\|_{L^2}^2 = -\langle u\cdot \nabla T_2, T\rangle \leq C\|\nabla T_2\|_{L^\infty} \|u\|_{L^2} \|T\|_{L^2} \leq C\|T_2\|_{H^{\frac52}}\|u\|_{L^2} \|T\|_{L^2},
\end{align*}
and
\begin{align*}
    \frac12 \frac{d}{dt} \|\nabla T\|_{L^2}^2 + \kappa \|\Delta T\|_{L^2}^2 = &\langle u_1 \cdot\nabla T, \Delta T\rangle + \langle u\cdot\nabla T_2, \Delta T\rangle
    \\
    \leq &C\|u_1\|_{L^4} \|\nabla T\|_{L^4} \|\Delta T\|_{L^2} + C\|u\|_{L^2} \|\nabla T_2\|_{L^\infty} \|\Delta T\|_{L^2} 
    \\
    \leq &C \|u_1\|_{H^1} \|\nabla T\|_{L^2}^{\frac12}\|\Delta T\|_{L^2}^{\frac32} + C\|T_2\|_{H^{\frac52}}\|u\|_{L^2} \|\Delta T\|_{L^2} .
\end{align*}

Denoting by 
\[
F(t) := \sum_{i=1}^N \|c_i(t)\|_{L^2}^2 + \|u(t)\|_{L^2}^2 + \|T(t)\|_{H^1}^2,
\]
combining the above energy estimates, and using Young's inequality,  we deduce that 
\begin{align*}
    \frac{d}{dt} F \leq C\left(1+ \sum_{i=1}^N (\|c_{i,1}\|_{H^1}^8+\|c_{i,2}\|_{H^1}^8) + \|T_1\|_{H^{\frac52}}^8 + \|T_2\|_{H^{\frac52}}^8 + \|u_1\|_{H^1}^4 + \|u_2\|_{H^1}^4 \right) F.
\end{align*}
Thanks to the regularity properties obeyed by $c_{i,j}, T_j, u_j$ for $j=1,2$ from \eqref{lwp-regularity}, we infer that the time function
\[
 k(t) := 1+ \sum_{i=1}^N (\|c_{i,1}(t)\|_{H^1}^8+\|c_{i,2}(t)\|_{H^1}^8) + \|T_1(t)\|_{H^{\frac52}}^8 + \|T_2(t)\|_{H^{\frac52}}^8 + \|u_1(t)\|_{H^1}^4 + \|u_2(t)\|_{H^1}^4,
\]
is bounded in $L^\infty(0,\mathcal T_0)$. Therefore, by Gr\"onwall's inequality, we  conclude that 
\begin{align*}
    F(t) \leq F(0) \exp\left(t\sup_{t\in(0,\mathcal T_0)} k(t)\right)
\end{align*}
for all $t\in (0,\mathcal T_0)$. This implies the uniqueness and stability of solutions.

\section{Extension of the Local Solution} \label{sec6}

In this section, we prove Theorem~\ref{thm:gwp-intro}, i.e., any local smooth solution can be extended to any time interval $[0,\mathcal{T}]$ provided that $\|T_0-T_r\|_{L^\infty}$ is sufficiently small. The proof is divided into 7 major steps.

% \begin{Thm}
% Let $c_{i0} \in H^1, u_0 \in V, T_0 \in H^{\frac{5}{2}}$ with $T_0 \ge T^* > 0$ and $c_{i0} \ge a_0 >0$ for any $i \in \left\{1, \dots, N\right\}$. Let $(u, c_i, T)$ be a smooth solution of the non-isothermal NPNS system on some time interval $[0, \mathcal{T}]$. If $T_0 - T_r$ is sufficiently small in $L^{\infty}$, then there is a constant $\Gamma_0$ depending only on the $H^1$ norms of the initial concentrations and velocity, the $H^{\frac{5}{2}}$ norm of the initial temperature, the parameters of the system, and some universal constants such that 
% \be 
% \sum\limits_{i=1}^{N}\|c_i(t)\|_{H^1} + \|T(t)-T_r\|_{H^{\frac{5}{2}}} + \|u(t)\|_{H^1} \le \Gamma_0
% \ee for any $t \in [0, \mathcal{T}]$. 
% \end{Thm}

\medskip
{\bf{Step 1. Local Behavior.}} We consider the entropic contribution in the Helmholtz free energy $\mathcal{F}$:
\begin{equation}
\sum\limits_{i=1}^{N} \int_{\TT^2} k_BT(c_i \log c_i - c_i + 1) dx =: \sum\limits_{i=1}^{N} \int_{\TT^2} k_BTE_i dx,  
\end{equation}
where $E_i = c_i \log c_i - c_i + 1$. Let $\mathcal{Q}$ be the  energy defined by 
\begin{equation}
\mathcal{Q} = 1+\frac{\varepsilon}{2} \|\na \phi\|_{L^2}^2 + \sum\limits_{i=1}^{N} \int_{\TT^2} k_BTE_i  dx  + \frac{1}{2} \|u\|_{L^2}^2. 
\end{equation} By continuity, there exists a time $t_0>0$ such that $\|\Lambda^{\frac{3}{2}} T\|_{L^2}^2 \le \|\Lambda^{\frac{3}{2}}T_0\|_{L^2}^2 + 1$ and $\mathcal{Q} \le 2\mathcal{Q}_0$ for all $t \in [0,t_0]$. Our goal is to prove that the latter bounds hold at times $t \in [t_0, \mathcal{T}]$. 

\medskip
{\bf{Step 2. Uniform bounds for $\mathcal{Q}$.}} 
We multiply \eqref{N-NPNS-system-original-1} by $T \log c_i$, integrate over $\TT^2$, and obtain the energy evolution
\begin{equation}
\begin{aligned}
\frac{d}{dt} \int_{\TT^2} TE_i dx
&= \int_{\TT^2} (\pa_t T) E_i dx + \int_{\TT^2} T (\pa_t c_i) \log c_i dx
\\&= \int_{\TT^2} (- u \cdot \na T + \kappa\Delta T) E_i dx + \int_{\TT^2} \pa_t c_i T \log c_i dx,
\end{aligned}
\end{equation}
yielding
\begin{equation}\la{com1}
\int_{\TT^2}\pa_t c_i T \log c_i dx
= \frac{d}{dt} \int_{\TT^2}TE_i dx
+ \int_{\TT^2}( u \cdot \na T - \kappa\Delta T)E_i dx. 
\end{equation}
To deal with the advection term, we integrate by parts using the divergence-free condition obeyed by $ u$ and obtain the identity
\begin{equation}\la{com2}
\beg{aligned}
\int_{\TT^2}( u \cdot \na c_i) T \log c_i dx
= - \int_{\TT^2} uc_i \cdot \na T \log c_i dx - \int_{\TT^2} u T \cdot \na c_i dx.
\end{aligned}
\end{equation}Adding \eqref{com1} and \eqref{com2}
gives rise to 
\begin{equation}
\begin{aligned}
\int_{\TT^2}(\pa_t c_i +  u \cdot \na c_i) T \log c_i dx
&= \frac{d}{dt} \int_{\TT^2}TE_i dx - \int_{\TT^2}\kappa\Delta T E_i dx - \int_{\TT^2}( u \cdot \na T) c_i dx - \int_{\TT^2}( u \cdot \na c_i) T dx.
\end{aligned}
\end{equation}Another integration by parts shows that 
\begin{equation}
- \int_{\TT^2}( u \cdot \na T) c_i dx - \int_{\TT^2}( u \cdot \na c_i) T dx= 0,
\end{equation}
and consequently,
\begin{equation}\label{gwp-eqn1}
\begin{split}
    \frac{d}{dt} \int_{\TT^2}TE_i dx- \int_{\TT^2}\kappa\Delta T E_i dx&=\int_{\TT^2}(\pa_t c_i +  u \cdot \na c_i) T \log c_i dx
\\
&= \int_{\TT^2}D_i \na \cdot \left(\na c_i + \frac{z_i e}{k_B T} c_i \nabla   \phi
+
c_i \log c_i \nabla \log T \right) T \log c_i dx
\\
&= - D_i \int_{\TT^2} \left(\na c_i + \frac{z_i e}{k_B T} c_i \nabla  \phi
+
c_i \log c_i \nabla \log T \right) \cdot \frac{\na c_i}{c_i}T  dx
\\&\qquad- D_i \int_{\TT^2} \left(\na c_i + \frac{z_i e}{k_B T} c_i \nabla  \phi
+
c_i \log c_i \nabla \log T \right) \cdot \na T \log c_i dx.
\end{split}
\end{equation}
The first term amounts to 
\be \label{gwp-eqn2}
\begin{aligned}
&- D_i \int_{\TT^2} \left(\na c_i + \frac{z_i e}{k_B T} c_i \nabla \phi
+
c_i \log c_i \nabla \log T \right) \cdot \frac{\na c_i}{c_i}T  dx
\\
&= -D_i \int_{\TT^2}\frac{T}{c_i} \left|\na c_i + \frac{z_i e}{k_B T} c_i \nabla   \phi
+
c_i \log c_i \nabla \log T\right|^2  dx
\\&\hspace{2cm}+ D_i \int_{\TT^2}\left(\na c_i + \frac{z_i e}{k_B T} c_i \nabla   \phi
+
c_i \log c_i \nabla \log T \right) \cdot \left(\frac{z_i e}{k_B} \nabla   \phi
+
 \log c_i \nabla T\right) dx
\\&=- D_i \left\|\frac{\sqrt{T}}{\sqrt{c_i}} (\na c_i + \frac{z_i e}{k_B T} c_i \nabla \phi
+
c_i \log c_i \nabla \log T) \right\|_{L^2}^2- \frac{z_i e}{k_B} \int_{\TT^2}(\pa_t c_i +  u \cdot \na c_i) \phi dx
\\
&\hspace{2cm}+D_i \int_{\TT^2}\left(\na c_i + \frac{z_i e}{k_B T} c_i \nabla   \phi
+
c_i \log c_i \nabla \log T \right) \cdot \nabla T \log c_i dx. 
\end{aligned} 
\end{equation}
Notice that the last term in \eqref{gwp-eqn1} and the last term in \eqref{gwp-eqn2} get canceled. Consequently, we have
\begin{equation}
\begin{aligned}
&\frac{d}{dt} \int_{\TT^2}TE_i dx - \int_{\TT^2}\kappa\Delta T E_i dx
\\
=&- D_i \left\|\frac{\sqrt{T}}{\sqrt{c_i}} (\na c_i + \frac{z_i e}{k_B T} c_i \nabla \phi
+
c_i \log c_i \nabla \log T) \right\|_{L^2}^2- \frac{z_i e}{k_B} \int_{\TT^2}(\pa_t c_i + u \cdot \na c_i) \phi dx.
\end{aligned}
\end{equation}
Summing over all indices $i \in \left\{1, \dots, N \right\}$ and making use of \eqref{N-NPNS-system-original-2} yield
\begin{equation}
\sum\limits_{i=1}^{N} - \frac{z_i e}{k_B} \int_{\TT^2}(\pa_t c_i + u \cdot \na c_i) \phi dx 
= - \frac{1}{k_B}\int_{\TT^2}(\pa_t \rho + u \cdot \na \rho) \phi  dx
= - \frac{\varepsilon}{2k_B} \frac{d}{dt} \|\na \phi\|_{L^2}^2 + \frac1{k_B}\int_{\TT^2}\rho \na \phi \cdot u dx. 
\end{equation}Putting these together, we end up with the energy equality
\begin{equation}
\beg{aligned}
&\frac{d}{dt} \left(\frac{\varepsilon}{2} \|\na \phi\|_{L^2}^2 + \sum\limits_{i=1}^{N} \int_{\TT^2} k_BTE_i dx \right) + \sum\limits_{i=1}^{N} k_B D_i \left\|\frac{\sqrt{T}}{\sqrt{c_i}} (\na c_i + \frac{z_i e}{k_B T} c_i \nabla \phi
+
c_i \log c_i \nabla \log T) \right\|_{L^2}^2
\\&= \int_{\TT^2}\rho \na \phi \cdot u dx + \sum\limits_{i=1}^{N} \int_{\TT^2}\kappa k_B\Delta T E_i dx.
\end{aligned}
\end{equation}In order to cancel the term $\int_{\TT^2}\rho \na \phi \cdot u dx$, we couple the latter with the $L^2$ evolution of the velocity $u$ and obtain 
\begin{equation}
\beg{aligned}
&\frac{d}{dt} \left(\frac{\varepsilon}{2} \|\na \phi\|_{L^2}^2 + \sum\limits_{i=1}^{N} \int_{\TT^2} k_BTE_i  dx + \frac{1}{2} \|u\|_{L^2}^2 \right) 
\\
&\qquad+ \sum\limits_{i=1}^{N} k_BD_i \left\|\frac{\sqrt{T}}{\sqrt{c_i}} (\na c_i + \frac{z_i e}{k_B T} c_i \nabla \phi
+
c_i \log c_i \nabla \log T) \right\|_{L^2}^2 + \nu\|\na u\|_{L^2}^2
\\&= \sum\limits_{i=1}^{N} \int_{\TT^2}\kappa k_B\Delta T E_i dx+ g\alpha_T\int_{\TT^2}(T - T_r) e_2 \cdot u dx.
\end{aligned}
\end{equation}
%When there is no heat diffusion, i.e., $\kappa=0$, the right hand side vanishes.
Now we explore the structure of the ionic concentration dissipative term. Indeed, it holds that 

\begin{equation}
    \begin{aligned}
   &\hspace{0.5cm}\sum_{i=1}^{N} k_B D_i 
\left\|
\frac{\sqrt{T}}{\sqrt{c_i}}
\left(
\nabla c_i
+ \frac{z_i e}{k_B T} c_i \nabla \phi
+ c_i \log c_i \nabla \log T
\right)
\right\|_{L^2}^2
\\
&\geq \sum_{i=1}^{N} k_B D 
\left\|
\frac{\sqrt{T}}{\sqrt{c_i}}
\left(
\nabla c_i
+ \frac{z_i e}{k_B T} c_i \nabla \phi
+ c_i \log c_i \nabla \log T
\right)
\right\|_{L^2}^2
\\&=
\sum_{i=1}^{N} k_B D
\int_{\TT^2}
\Big(
\frac{T}{c_i}|\nabla c_i|^2
+ \frac{T}{c_i}\left|\frac{z_i e}{k_B T} c_i \nabla \phi + c_i \log c_i \nabla \log T\right|^2
+ 2\nabla c_i \cdot \frac{z_i e}{k_B}  \nabla \phi
+ 2\,\nabla c_i \cdot  \log c_i \nabla T
\Big)dx 
\\&= \sum\limits_{i=1}^{N} 4k_BD \left\|\sqrt{T} \na \sqrt{c_i} \right\|_{L^2}^2
+  \sum\limits_{i=1}^{N} k_B D \left\| \frac{\sqrt{T}}{\sqrt{c_i}} \Big(\frac{z_i e}{k_B T} c_i \nabla \phi + c_i \log c_i \nabla \log T \Big)\right\|_{L^2}^2
\\&\quad+ \int_{\TT^2}\sum\limits_{i=1}^{N}2 D \na(ez_i c_i) \cdot \na \phi dx 
+ \int_{\TT^2}\sum\limits_{i=1}^{N} 2k_BD \log c_i \na c_i \cdot \na T dx,
% \\
% &\quad
% + \int_{\TT^2}\sum\limits_{i=1}^{N}2z_iDeT^{-1}c_i \log c_i \na \phi \cdot \na T dx,
    \end{aligned}
\end{equation}
where $D$ is the minimum of $D_1, \dots, D_N$. 
The first two terms are nonnegative. The third term is also nonnegative, which can be shown using \eqref{N-NPNS-system-original-2} and integration by parts:
\be 
\int_{\TT^2}\sum\limits_{i=1}^{N} 2D \na (ez_i c_i) \cdot \na \phi dx
= 2\varepsilon^{-1} D \|\rho\|_{L^2}^2.
\ee 
% Regarding the sixth term, we can use the Cauchy-Schwarz and Young inequalities to bound it from above as follows, 
% \be 
% \begin{aligned}
% &\int_{\TT^2}\sum\limits_{i=1}^{N} 2z_iD_ieT^{-1} c_i \log c_i \na \phi \cdot \na T dx
% \le \sum\limits_{i=1}^{N} 2D_iz_i e \|T^{-\frac{1}{2}} \sqrt{c_i} \na \phi\|_{L^2}\|T^{-\frac{1}{2}}\sqrt{c_i} \log c_i \na T\|_{L^2} 
% \\&\quad\quad\le \sum\limits_{i=1}^{N} k_B^{-1} D_iz_i^2e^2 \|T^{-\frac{1}{2}}\sqrt{c_i} \na \phi\|_{L^2}^2
% + \sum\limits_{i=1}^{N} k_BD_i \|T^{-\frac{1}{2}} \sqrt{c_i} \log c_i \na T\|_{L^2}^2.
% \end{aligned}
%\ee 
Consequently, we obtain the differential inequality
\begin{equation}
\beg{aligned}
&\frac{d}{dt} \mathcal Q 
+\sum\limits_{i=1}^{N}4k_BD \|\sqrt{T} \na \sqrt{c_i}\|_{L^2}^2 +2\varepsilon^{-1} D\|\rho\|_{L^2}^2 + \nu\|\na u\|_{L^2}^2
\\\le &\sum\limits_{i=1}^{N} \int_{\TT^2}k_B\left(\kappa \Delta T E_i  - 2D \log c_i \na c_i \cdot \na T\right) dx + g\alpha_T\int_{\TT^2}(T - T_r) e_2 \cdot u dx
\\=& \sum\limits_{i=1}^{N} \int_{\TT^2}{-k_B(2D+\kappa)} \log c_i \na c_i \cdot \na T dx + g\alpha_T\int_{\TT^2}(T - T_r) e_2 \cdot u dx.  
\end{aligned}
\end{equation}
Let $\eta > 0$ be a sufficiently small exponent to be chosen later. When $c_i\leq 1$, we have the bound $|\sqrt{c_i} \log c_i| \leq \frac{2}{e}$. When $c_i >1$,
the logarithmic estimate 
$
|\sqrt{c_i} \log c_i| \le \eta^{-1} c_i^{\frac{1}{2} + \eta}
$ holds. 
Therefore, it follows that
\[
|\sqrt{c_i} \log c_i| \le 1 + \eta^{-1} c_i^{\frac{1}{2} + \eta}
\] for any $x \in \TT^2$ and any $t \ge 0$. 
Using the latter, we estimate 
 \begin{equation}
\beg{aligned}
& \hspace{0.3cm}\sum\limits_{i=1}^{N}  -k_B(2D+\kappa) \int_{\TT^2}\na T \cdot \na c_i \log c_i dx
=\sum\limits_{i=1}^{N}  -k_B(2D+\kappa) \int_{\TT^2}2\na T \cdot \na \sqrt{c_i} \sqrt{c_i}  \log c_i dx
\\&\le \sum\limits_{i=1}^{N} 4 k_B(2D+\kappa) \int_{\TT^2}|\na T| | \na \sqrt{c_i}| (\eta^{-1} c_i^{\frac{1}{2} + \eta} +1) dx
\\&\le \sum\limits_{i=1}^{N} 4\eta^{-1} k_B(2D+\kappa) \|\na T\|_{L^{4+\delta}} \|\na \sqrt{c_i}\|_{L^2}\|\sqrt{c_i}\|_{L^4}\|c_i^{\eta}\|_{L^{\frac{1}{\eta}}}
\\
&\hspace{2cm} +  \sum\limits_{i=1}^{N}4 k_B(2D+\kappa)\|\nabla T\|_{L^2}  \|\na \sqrt{c_i}\|_{L^2}.
\end{aligned}
 \end{equation}Here $\delta$ is chosen so that $4+\delta, 2, 4, \frac1\eta$ are H\"older exponents. Recall that $T_r = \int_{\TT^2} T dx$ is the spatial average. We choose $\eta$ and $\delta$ such that 
 \begin{equation}
\|\na T\|_{L^{4+\delta}}
\le C\|T - T_r\|_{L^{\infty}}^{\frac{1}{2}} \| T-T_r\|_{H^{\frac{5}{2}}}^{\frac{1}{2}}
\end{equation}
holds. This leads to the choice of $\delta = 4$ and $\eta = \frac{1}{8}$, yielding
\begin{equation}
\beg{aligned}
& \sum\limits_{i=1}^{N}  -k_B(2D+\kappa) \int_{\TT^2}\na T \cdot \na c_i \log c_i dx
\\\le &C\sum\limits_{i=1}^{N}  \|T-T_r\|_{L^{\infty}}^{\frac{1}{2}} \|T-T_r\|_{H^{\frac{5}{2}}}^{\frac{1}{2}} \|\na \sqrt{c_i}\|_{L^2}\|\sqrt{c_i}\|_{L^4}\|c_i^{1/8}\|_{L^8}
 +  C\sum\limits_{i=1}^{N}\|\nabla T\|_{L^2}  \|\na \sqrt{c_i}\|_{L^2}
 \\
 \leq &C\sum\limits_{i=1}^{N}  \|T-T_r\|_{L^{\infty}}^{\frac{1}{2}} \|T-T_r\|_{H^{\frac{5}{2}}}^{\frac{1}{2}} \|\na \sqrt{c_i}\|_{L^2}( \|\na \sqrt{c_i}\|_{L^2}^{\frac{1}{2}}\|c_i\|_{L^1}^{\frac{3}{8}} +\|c_i\|_{L^1}^{\frac{5}{8}}) + C\sum\limits_{i=1}^{N}\|\nabla T\|_{L^2}  \|\na \sqrt{c_i}\|_{L^2}
\end{aligned} 
\end{equation}  for some universal constant $C$, where in the last step we have used the Ladyzhenskaya interpolation inequality.
Applications of Young's inequality give rise to 
\begin{equation}
\beg{aligned}
& \sum\limits_{i=1}^{N}  -k_B(2D+\kappa) \int_{\TT^2}\na T \cdot \na c_i \log c_i dx
\\
\leq&\sum\limits_{i=1}^{N} 2k_{B}D \|\sqrt{T} \na \sqrt{c_i}\|_{L^2}^2 + C\sum\limits_{i=1}^{N}  \|c_i\|_{L^1}^{\frac{3}{2}}\|T-T_r\|_{L^{\infty}}^2 \|T -T_r\|_{H^{\frac{5}{2}}}^2  \\
&+ C \sum\limits_{i=1}^{N}  \|c_i\|_{L^1}^{\frac{5}{4}}\|T-T_r\|_{L^{\infty}}\|T -T_r\|_{H^{\frac{5}{2}}}
+ C\|\na T\|_{L^2}^2, 
\end{aligned} 
\end{equation} where the bound $\|\frac1T\|_{L^\infty} \leq \frac1{T^*}$ was exploited.
Denoting by $\mathcal{R}$ the following energy, 
\begin{equation}
\begin{aligned}
\mathcal{R} &=  \sum\limits_{i=1}^{N}2k_BD \|\sqrt{T} \na \sqrt{c_i}\|_{L^2}^2 +2\varepsilon^{-1} D\|\rho\|_{L^2}^2 + \nu\|\na u\|_{L^2}^2,
\end{aligned}
\end{equation}
it holds that 
\begin{equation}
\begin{aligned}
\frac{d}{dt}\mathcal{Q} + \mathcal{R} 
&\le C\sum\limits_{i=1}^{N}  \|c_i\|_{L^1}^{\frac{3}{2}}\|T-T_r\|_{L^{\infty}}^2 \|\Lambda^{\frac52}T \|_{L^2}^2 + C \sum\limits_{i=1}^{N}  \|c_i\|_{L^1}^{\frac{5}{4}}\|T-T_r\|_{L^{\infty}}\|T -T_r\|_{H^{\frac{5}{2}}}
\\
&\qquad + C\|\na T\|_{L^2}^2+ C \|u\|_{L^2} \|T-T_r\|_{L^2}
\\
&\leq C\left(\|T-T_r\|_{L^{\infty}}^2 \|\Lambda^{\frac52}T \|_{L^2}^2 + \|T-T_r\|_{L^{\infty}} (1+\|\Lambda^{\frac52}T \|_{L^2}^2)  + \|\na T\|_{L^2}^2 + \|\nabla u\|_{L^2} \|T-T_r\|_{L^2}\right).
\end{aligned}
\end{equation}
Here we have used the Poincar\'e inequality for $u$, and the fact that $\|c_i\|_{L^1} = \|c_{i0}\|_{L^1}$ is constant in time.
Accordingly, we need to derive uniform bounds for the temperature $T$ in $L^2(0,t_0; H^{\frac52})$. 

The $H^{\frac{3}{2}}$ evolution of $T$ obeys
\begin{equation}
\begin{aligned}
\frac{1}{2}\frac{d}{dt}\|\Lambda^{\frac{3}{2}}T\|_{L^2}^2 + \kappa \|\Lambda^{\frac{5}{2}}T\|_{L^2}^2
= - \int_{\TT^2}\Lambda (u \cdot \na T) \Lambda^{2}T dx
\le \|\na (u \cdot \na T)\|_{L^{\frac{4}{3}}} \|\Delta T\|_{L^4}.
\end{aligned}
\end{equation}
By making use of continuous Sobolev embeddings and interpolation inequalities, we obtain 
\begin{equation}
\begin{aligned}
\frac{1}{2}\frac{d}{dt}\|\Lambda^{\frac{3}{2}}T\|_{L^2}^2 + \kappa \|\Lambda^{\frac{5}{2}}T\|_{L^2}^2
&\le C(\|\na u\|_{L^2}\|\na T\|_{L^4} + \|u\|_{L^4} \|\Delta T\|_{L^2})\|\Lambda^{\frac{5}{2}} T\|_{L^2}
\\&\le C\|\na u\|_{L^2} \|\Lambda^{\frac{3}{2}}T\|_{L^2} \|\Lambda^{\frac{5}{2}}T\|_{L^2} 
+C \|u\|_{L^2}^{\frac{1}{2}}\|\na u\|_{L^2}^{\frac{1}{2}}\|\Lambda^{\frac{3}{2}}T\|_{L^2}^{\frac{1}{2}}\|\Lambda^{\frac{5}{2}}T\|_{L^2}^{\frac{3}{2}}
\\&\le C\|\na u\|_{L^2}^2 \left(\|u\|_{L^2}^2 \|\Lambda^{\frac{3}{2}}T\|_{L^2}^2 +  \|\Lambda^{\frac{3}{2}}T\|_{L^2}^2 \right) + \frac{\kappa}{2} \|\Lambda^{\frac{5}{2}}T\|_{L^2}^2.
\end{aligned}
\end{equation} 
This yields 
\begin{equation}
\frac{d}{dt}\|\Lambda^{\frac{3}{2}}T\|_{L^2}^2 + \kappa \|\Lambda^{\frac{5}{2}}T\|_{L^2}^2
\le C\|\na u\|_{L^2}^2 \left(\|u\|_{L^2}^2 \|\Lambda^{\frac{3}{2}}T\|_{L^2}^2 +  \|\Lambda^{\frac{3}{2}}T\|_{L^2}^2 \right).
\end{equation}

Recall that there exists a time $t_0>0$ such that $\|\Lambda^{\frac{3}{2}} T\|_{L^2}^2 \le \|\Lambda^{\frac{3}{2}}T_0\|_{L^2}^2 + 1$ and $\mathcal{Q} \le 2\mathcal{Q}_0$ for all $t \in [0,t_0]$. Thus, on the time interval $[0,t_0]$, it holds that 
\begin{equation}\label{T32}
\|\Lambda^{\frac32} T(t)\|_{L^2}^2+\kappa \int_0^t \|\Lambda^{\frac{5}{2}}T\|_{L^2}^2 ds 
\le \|\Lambda^{\frac{3}{2}}T_0\|_{L^2}^2 + C(\mathcal{Q}_0 \|\Lambda^{\frac{3}{2}}T_0\|_{L^2}^2 +  \|\Lambda^{\frac{3}{2}}T_0\|_{L^2}^2 +\mathcal{Q}_0 + 1)\int_{0}^{t}\|\na u\|_{L^2}^2 ds.
\end{equation}
Moreover, from the $L^2$ energy evolution of the temperature and the maximum principle, we have 
\begin{equation}
\kappa \int_0^{t} \|\na T\|_{L^2}^2 ds \le \|T_0-T_r\|_{L^2}^2
\le C\|T_0 - T_r\|_{L^{\infty}}^2, \quad \|T-T_r\|_{L^\infty} \leq \|T_0 - T_r\|_{L^{\infty}}.
\end{equation}
Consequently, by Young's inequality we obtain 
\begin{equation}
\begin{aligned}
\mathcal{Q}(t) + \int_{0}^{t} \mathcal{R}
 ds
 \le &\mathcal{Q}_0 
 +C(1+\|\Lambda^{\frac32} T_0\|_{L^2})(\|T_0-T_r\|_{L^{\infty}}+\|T_0-T_r\|_{L^{\infty}}^2) %+ C\|T_0-T_r\|_{L^\infty}  \left(  \int_{0}^{t}\|\na u\|_{L^2}^2 ds\right)^{\frac12}
\\&+
 C(\mathcal{Q}_0 \|\Lambda^{\frac{3}{2}}T_0\|_{L^2}^2 +  \|\Lambda^{\frac{3}{2}}T_0\|_{L^2}^2 +\mathcal{Q}_0 + 1)
   (\|T_0-T_r\|_{L^{\infty}}+\|T_0-T_r\|_{L^{\infty}}^2)  \int_{0}^{t}\|\na u\|_{L^2}^2 ds
% \\ \leq &\mathcal{Q}_0 
%  +C(1+\|\Lambda^{\frac32} T_0\|_{L^2})(\|T_0-T_r\|_{L^{\infty}}+\|T_0-T_r\|_{L^{\infty}}^2)  
%  \\
%  &+
%  C\left[(\mathcal{Q}_0 \|\Lambda^{\frac{3}{2}}T_0\|_{L^2}^2 +  \|\Lambda^{\frac{3}{2}}T_0\|_{L^2}^2 +\mathcal{Q}_0 + 1)
%    \|T_0-T_r\|_{L^{\infty}}^2 + \|T_0-T_r\|_{L^{\infty}}\right] \int_{0}^{t}\|\na u\|_{L^2}^2 ds
 \\=: &\mathcal{Q}_0 + A_1 + A_2 \int_{0}^{t}\|\na u\|_{L^2}^2 ds.
 \end{aligned}
\end{equation}
Now we choose $\|T_0 - T_r\|_{L^\infty}$ to be sufficiently small so that 
\be \label{small-1}
 A_1 \le 
 \frac{\mathcal{Q}_0}{2} \quad\text{ and } \quad A_2 \le \frac{\nu}{2}
\ee 
hold simultaneously. Then it follows that
\begin{equation} \label{squarerootc}
\mathcal{Q}(t) + \frac12\int_{0}^{t} \mathcal{R}(s) ds \leq \frac32 \mathcal{Q}_0 < 2 \mathcal{Q}_0
\end{equation}on $[0,t_0]$. In addition, this together with \eqref{T32} imply that $\Lambda^{\frac32}T$ and $\Lambda^{\frac52}T$ are bounded in $L^\infty(0,t_0;L^2)$ and $L^2(0,t_0;L^2)$ respectively, with bounds uniform in time. Note that the smallness of $\|T_0 - T_r\|_{L^\infty}$ depends on the parameters of the system and other initial conditions, but is independent of $t_0$.  

\medskip
   {\bf{Step 3. Uniform bounds for the concentrations in $L^2.$}} We address the $L^2$ evolution of each ionic concentration $c_i$. Indeed, we have 
\begin{equation}
\begin{aligned}
&\frac{1}{2}\frac{d}{dt}\|c_i\|_{L^2}^2 + D_i\|\na c_i\|_{L^2}^2 
\\
= &-z_i D_i k_{B}^{-1} e\int_{\TT^2}T^{-1} c_i \na \phi \cdot \na c_i dx
-D_i \int_{\TT^2}T^{-1} c_i \log c_i \nabla T \cdot \na c_i dx
\\
\leq & C\left(  \|c_i\|_{L^4}\|\na \phi\|_{L^4} + \|c_i \log  c_i \na T\|_{L^2} \right)\|\na c_i\|_{L^2}
\\
\leq & C \left( (\|c_i\|_{L^2}^{\frac{1}{2}}\|\na c_i\|_{L^2}^{\frac{1}{2}} + \|c_i\|_{L^2} )\|\na \phi\|_{L^2}^{\frac{1}{2}}\|\rho\|_{L^2}^{\frac{1}{2}} + C\|(c_i^{\frac{9}{8}} +1) \na T\|_{L^2} \right)\|\na c_i\|_{L^2}
\\\leq & \frac{D_i}{8}\|\na c_i\|_{L^2}^2 + C \|c_i\|_{L^2}^2 \|\na \phi\|_{L^2}^2 \|\rho\|_{L^2}^2 
+ C \|c_i\|_{L^2}^2\|\rho\|_{L^2}^2 + C\|\na T\|_{L^2}^2 + C\|c_i^{\frac{9}{8}} \na T\|_{L^2}^2.
\end{aligned}
\end{equation} 
Using interpolation inequalities and the continuous embeddings of $H^{\frac{3}{2}}$ and $H^{\frac{5}{2}}$ in $H^1$, we estimate the term
\be 
\begin{aligned}
\|c_i^{\frac{9}{8}} \na T\|_{L^2}^2
&\le \|c_i\|_{L^4}^2\| {c}_i^{\frac{1}{8}}\|_{L^8}^2 \|\na T\|_{L^{8}}^2
\\&\le C(\|c_i\|_{L^2}\|\na c_i\|_{L^2} + \|c_i\|_{L^2}^2)\|c_i\|_{L^1}^{\frac{1}{4}}\|\l T\|_{L^2}  \|\l^{\frac{5}{2}}T\|_{L^2}
\\&\le \frac{D_i}{8}\|\na c_i\|_{L^2}^2 + C\|\l^{\frac{3}{2}} T\|_{L^2}^2 \|\l^{\frac{5}{2}}T\|_{L^2}^2 \|c_i\|_{L^2}^2 
+  C \|\l^{\frac{5}{2}}T\|_{L^2}^2\|c_i\|_{L^2}^2.
\end{aligned}
\ee 
This gives rise to the differential inequality 
\be 
\begin{aligned}
&\frac{d}{dt}\|c_i\|_{L^2}^2
+ D_i\|\na c_i\|_{L^2}^2 
\\\leq & C\|\na T\|_{L^2}^2 + C\left(\|\l^{\frac{3}{2}} T\|_{L^2}^2 \|\l^{\frac{5}{2}}T\|_{L^2}^2 +  \|\l^{\frac{5}{2}}T\|_{L^2}^2+ \|\na \phi\|_{L^2}^2 \|\rho\|_{L^2}^2 
+  \|\rho\|_{L^2}^2\right)\|c_i\|_{L^2}^2
\\=: &C\|\na T\|_{L^2}^2 + F(t) \|c_i\|_{L^2}^2
\end{aligned}
\ee where $\int_{0}^{t} F(s) ds$ is bounded uniformly in time on $[0,t_0]$. Indeed, this follows from the uniform-in-time boundedness of $\rho$ and $\l^{\frac{5}{2}}T$ in $L^2(0,t_0; L^2)$ and $\l^{\frac{3}{2}}T$ and $\na \phi$ in $L^{\infty}(0,t_0; L^2)$. Moreover, the bound is independent of $t_0$.
Thus, an application of Gr\"onwall's inequality yields 
\be 
\begin{aligned}
\|c_i(t)\|_{L^2}^2 + D_i \int_{0}^{t} \|\na c_i\|_{L^2}^2 ds
&\le C\left(\|c_i(0)\|_{L^2}^2 + \int_{0}^{t}\|\na T(s)\|_{L^2}^2 ds\right) \exp \left(\int_{0}^{t} F(s) ds\right)
\\&\le C\left(\|c_i(0)\|_{L^2}^2 +  \|T_0 - T_r\|_{L^2}^2 \right) \exp \left(\int_{0}^{t} F(s) ds\right)
\end{aligned}
\ee for any $t \in [0,t_0].$ 
Note that these bounds are uniform in time and they depend only on the initial data and the parameters of the system. 

\medskip
   {\bf{Step 4. Uniform bounds for the velocity in $H^1.$}} The evolution of $\na u$ in $L^2$ is described by the energy equality
\begin{equation}
\frac{1}{2}\frac{d}{dt}\|\na u\|_{L^2}^2
+ \nu \|\Delta u\|_{L^2}^2 = -g\alpha_T \int_{\TT^2}(T - T_r)e_2 \Delta u dx+ \int_{\TT^2}\rho \na \phi \Delta u dx,
\end{equation}
where the cancellation 
\begin{equation*}
    \int_{\TT^2} (u\cdot\nabla u) \cdot \Delta u dx = 0
\end{equation*}
is exploited.
Applications of the Cauchy-Schwarz, Young, and Gr\"onwall inequalities yield
\begin{equation}
\begin{aligned}
\|\na u\|_{L^2}^2 
&\le \|\na u_0\|_{L^2}^2
+ C \|T_0 - T_r\|_{L^2}^2
+ \int_{0}^{t} \|\rho\|_{L^4}^2\|\na \phi\|_{L^4}^2 ds  
\\&\le \|\na u_0\|_{L^2}^2
+ C  \|T_0 - T_r\|_{L^2}^2
+ C \int_{0}^{t} \|\rho\|_{L^2}^2\|\na \rho\|^2_{L^2} ds,
\end{aligned}
\end{equation} for any $t \in [0,t_0].$ The latter is bounded by a constant $\Gamma$ that depends only on the initial data and does not depend on $t_0$. This follows from the uniform-in-time boundedness of the ionic concentrations and their gradients in $L^{\infty}(0,t_0; L^2)$ and $L^2(0,t_0; L^2)$ respectively. 

\medskip
      {\bf{Step 5. Uniform bounds for the temperature in $H^{\frac{3}{2}}.$}} In order to extend beyond time $t_0$, we need to show $\|\Lambda^{\frac{3}{2}}T\|^2_{L^2} < \|\Lambda^{\frac{3}{2}} T_0\|_{L^2}^2+1$ on $[0,t_0]$. To this end, we recall that on $[0,t_0]$ we have 
\begin{equation}
\begin{aligned}
\frac{d}{dt}\|\Lambda^{\frac{3}{2}}T\|_{L^2}^2 + \kappa \|\Lambda^{\frac{5}{2}}T\|_{L^2}^2
&\le C\|\na u\|_{L^2}^2 \left( \|u\|_{L^2}^2 \|\Lambda^{\frac{3}{2}}T\|_{L^2}^2 + \|\Lambda^{\frac{3}{2}}T\|_{L^2}^2 \right).
\\&\le C\|\na u\|_{L^2}^2 ( \mathcal{Q}_0 +1)\|\na T\|_{L^{2}}^{\frac{4}{3}} \|\Lambda^{\frac{5}{2}}T\|_{L^2}^{\frac{2}{3}}, 
\end{aligned}
\end{equation}where the last inequality follows from 
\begin{equation}
    \|\Lambda^{\frac{3}{2}}T\|_{L^2}^2 \lesssim {\|\nabla T\|^{\frac43}_{L^2}} \|\Lambda^{\frac{5}{2}}T\|_{L^2}^{\frac23}.
\end{equation}
By Young's inequality, we infer that 
\begin{equation}
    \frac{d}{dt}\|\Lambda^{\frac{3}{2}}T\|_{L^2}^2 +\frac{\kappa}{2} \|\Lambda^{\frac{5}{2}}T\|_{L^2}^2
\le C\|\na u\|_{L^2}^3 (\mathcal{Q}_0 +1)^\frac32{\|\nabla T\|^2_{L^2}}.
\end{equation} Since $\|\na u\|_{L^2}^2 \le \Gamma$ on $[0, t_0]$, we deduce that 
\begin{equation}
\|\Lambda^{\frac{3}{2}}T\|_{L^2}^2 {+ \frac{\kappa}2 \int_0^{t}\|\Lambda^{\frac{5}{2}}T\|_{L^2}^2 ds}
\le \|\Lambda^{\frac{3}{2}}T_0\|_{L^2}^2 + C\Gamma^{\frac{3}{2}} ( \mathcal{Q}_0+1)^\frac{3}{2}\|T_0 - T_r\|_{L^{\infty}}^2. 
\end{equation}
Choosing $\|T_0-T_r\|_{L^\infty}$ to be sufficiently small such that
\begin{equation}\label{small-2}
C\Gamma^{\frac{3}{2}} (\mathcal{Q}_0+1)^{\frac{3}{2}} \|T_0 - T_r\|_{L^{\infty}}^2
\le \frac{1}{2}
\end{equation}
Consequently, it yields the desired result $\|\Lambda^{\frac{3}{2}}T\|_{L^2}^2 < \|\Lambda^{\frac{3}{2}} T_0\|_{L^2}^2+1$. Note that the smallness condition imposed on $\|T_0-T_r\|_{L^\infty}$ is again independent of $t_0$.

\medskip
      {\bf{Step 6. Extension.}} 
As a conclusion of the previous steps, we have $\|\Lambda^{\frac{3}{2}}T\|_{L^2}^2$ and $\mathcal{Q}$ strictly less than $\|\Lambda^{\frac{3}{2}}T_0\|_{L^2}^2 + 1$ and $2\mathcal{Q}_0$ respectively, which allows us to extend the solution to the time interval $[0,\mathcal{T}]$. Indeed, suppose there is a time $t_1 < \mathcal{T}$ such that $\|\l^{\frac{3}{2}}T(t)\|_{L^2}^2 < \|\l^{\frac{3}{2}}T_0\|_{L^2}^2 + 1$ and $\mathcal{Q}(t) < 2\mathcal{Q}_0$ on $[0,t_1)$, but $\|\l^{\frac{3}{2}}T(t_1)\|_{L^2}^2 = \|\l^{\frac{3}{2}}T_0\|_{L^2}^2 + 1$ or $\mathcal{Q}(t_1) = 2\mathcal{Q}_0$. This implies that  $\|\l^{\frac{3}{2}}T(t)\|_{L^2}^2 \le \|\l^{\frac{3}{2}}T_0\|_{L^2}^2 + 1$ and $\mathcal{Q}(t) \le 2\mathcal{Q}_0$ on $[0,t_1]$, in which case we can reapply the above argument to deduce that $\|\l^{\frac{3}{2}}T(t)\|_{L^2}^2 < \|\l^{\frac{3}{2}}T_0\|_{L^2}^2 + 1$ and $\mathcal{Q}(t) < 2\mathcal{Q}_0$ on $[0,t_1]$, contradicting our assumption. Therefore, we can extend, not only slightly beyond $t_0$, but also up to the time $\mathcal{T}$. The only requirement is the continuity in time of the norms involved, which holds due to Lions-Magenes. 

\medskip
      {\bf{Step 7. Uniform bounds for the temperature in $H^{\frac{5}{2}}$ and concentrations in $H^1$.}} The goal of this step is to study the behavior of the smooth solution at time $\mathcal{T}$ in the aforementioned Sobolev spaces in which uniqueness is guaranteed. In fact, the $H^{\frac{5}{2}}$ norm of $T$ obeys 
\be
\begin{aligned}
    \frac{d}{dt}\|\Lambda^{\frac{5}{2}}T\|_{L^2}^2 + \kappa \|\Lambda^{\frac{7}{2}}T\|_{L^2}^2 \leq C\|\Lambda^{\frac32} u\|_{L^2}^2 \|\Lambda^{\frac52} T\|_{L^2}^2.
\end{aligned}
\ee The latter follows from integration by parts and applications of H\"older and Young inequalities. Due to the uniform-in-time boundedness of $u$ in $L^2(0, \mathcal{T}; H^2)$, together with Gr\"onwall's lemma, we deduce that the temperature $T$ belongs to $L^\infty(0, \mathcal{T}; H^{\frac{5}{2}})\cap L^2(0, \mathcal{T}; H^{\frac{7}{2}})$ with uniform-in-time bounds. 

As for the $H^1$ evolution of the ionic concentrations, we multiply each $c_i$-equation by $\Delta c_i$ and integrate over $\TT^2$. Our goal is to obtain uniform-in-time bounds that are independent of $\mathcal{T}$, which requires delicate analysis. The advection and electromigration terms can be estimated classically as follows, 
\be 
\begin{aligned}
&\left|\int_{\TT^2} u \cdot \na c_i \Delta c_i dx\right| + D_i|z_i|ek_{B}^{-1}\left|\int_{\TT^2} \na \cdot (c_i T^{-1} \na \phi) \Delta c_i dx \right|
\\&\quad\quad\le C\|u\|_{L^{\infty}}^2 \|\na c_i\|_{L^2}^2 
+ C\|\na c_i \|_{L^2}^2 \|\na \phi\|_{L^{\infty}}^2 
+ C\|c_i\|_{L^4}^2 \|\Delta \phi\|_{L^4}^2
\\&\quad\quad\quad\quad+  C\|\na T\|_{L^{\infty}}^2 \|\na \phi\|_{L^4}^2\|c_i\|_{L^4}^2
+ \frac{D_i}{4} \|\Delta c_i\|_{L^2}^2
\\&\quad\quad\le C\left(\|\Delta u\|_{L^2}^2 + \|\na \rho\|_{L^2}^2
+ \|\Lambda^{\frac{5}{2}}T\|_{L^2}^2 \|\na \rho\|_{L^2}^2  \right) \|\na c_i\|_{L^2}^2 
\\&\quad\quad\quad\quad+ C\overline{c_i}^2 \|\na \rho\|_{L^2}^2 \left(1+ \|\Lambda^{\frac{5}{2}}T\|_{L^2}^2 \right)+ \frac{D_i}{4} \|\Delta c_i\|_{L^2}^2.
\end{aligned}
\ee The logarithmic term can be decomposed into three parts: 
\be
\begin{aligned}
&D_i\left|\int_{\TT^2} \na \cdot \left(c_i \log c_i T^{-1} \na T \right) \Delta c_i \right|
\le D_i\left|\int_{\TT^2} \left(\na c_i T^{-1} \cdot \na T \right) \Delta c_i \right|
\\&\quad\quad\quad\quad+ D_i\left|\int_{\TT^2} \left(c_i \log c_i \na \cdot (T^{-1} \na T) \right) \Delta c_i \right|
+ D_i\left|\int_{\TT^2} \left(\na c_i \log c_i T^{-1} \cdot\na T \right) \Delta c_i \right|.
\end{aligned}
\ee The first two terms can be estimated as follows,
\be 
\begin{aligned}
&D_i\left|\int_{\TT^2} \left(\na c_i T^{-1} \cdot \na T \right) \Delta c_i \right|
+ D_i\left|\int_{\TT^2} \left(c_i \log c_i \na \cdot (T^{-1} \na T) \right) \Delta c_i \right| 
\\&\quad\quad\le C\|\na T\|_{L^{\infty}}^2 \|\na c_i\|_{L^2}^2
+  C\|\na T\|_{L^{\infty}}^4 \|c_i \log c_i\|_{L^2}^2
+ C\|\Delta T\|_{L^4}^2\|c_i \log c_i\|_{L^4}^2 + \frac{D_i}{8}\|\Delta c_i\|_{L^2}^2
\\&\quad\quad\le C\|\na T\|_{L^{\infty}}^2 \|\na c_i\|_{L^2}^2
+  C\|\na T\|_{L^{\infty}}^4 (1 + \|c_i\|_{L^3}^3)
+ C\|\Delta T\|_{L^4}^2(1 + \|c_i\|_{L^6}^3)  + \frac{D_i}{8}\|\Delta c_i\|_{L^2}^2
\\&\quad\quad\le  C\left(1 + \overline{c_i}^3  + \|\na c_i\|_{L^2}^2 + \|c_i - \overline{c_i}\|_{L^2} \|\na c_i\|_{L^2}^2 \right)\left(1 + \|\Lambda^{\frac{5}{2}}T\|_{L^2}^2\right) \|\Lambda^{\frac{5}{2}}T\|_{L^2}^2+ \frac{D_i}{8}\|\Delta c_i\|_{L^2}^2.
\end{aligned}
\ee 
The third logarithmic term is more delicate due to the change in behavior of $|\log c_i|$ near 1. In fact, we split the domain of integration into $\left\{c_i < 1 \right\} \cap \TT^2$ and $\left\{c_i \geq 1 \right\} \cap \TT^2$ and treat each term separately. On the one hand, we have 
\be 
\begin{aligned}
&\int_{\left\{c_i < 1 \right\} \cap \TT^2} \left(\na c_i \log c_i T^{-1} \cdot\na T \right) \Delta c_i 
\le  C\int_{\left\{c_i < 1 \right\} \cap \TT^2} |\na c_i| (\sqrt{c_i})^{-1} |T^{-1}| |\na T| \|\Delta c_i| 
\\&\quad\quad\le C\|\sqrt{T}\na \sqrt{c}_i\|_{L^2} \|\na T\|_{L^{\infty}}\|\Delta c_i\|_{L^2}
\le C\|\Lambda^{\frac{5}{2}}T\|_{L^2}^2 \|\sqrt{T} \na \sqrt{c}_i\|_{L^2}^2 + \frac{D_i}{8}\|\Delta c_i\|_{L^2}^2.
\end{aligned}
\ee 
On the other hand, it holds that
\be 
\begin{aligned}
&\int_{\left\{c_i \ge 1 \right\} \cap \TT^2} \left(\na c_i \log c_i T^{-1} \na T \right) \Delta c_i 
\le  C\int_{\left\{c_i \ge 1 \right\} \cap \TT^2} |\na c_i| \sqrt{c_i}|T^{-1}| |\na T| \|\Delta c_i| 
\\&\quad\quad\le C
\|\na c_i\|_{L^4}\|\sqrt{c_i}\|_{L^4} \|\na T\|_{L^{\infty}}\|\Delta c_i\|_{L^2}
\le C\|c_i\|_{L^2}^2\|\Lambda^{\frac{5}{2}}T\|_{L^2}^4 \| \na c_i\|_{L^2}^2 + \frac{D_i}{8}\|\Delta c_i\|_{L^2}^2.
\end{aligned}
\ee Putting all these estimates together, we obtain   
\be 
\frac{d}{dt}\|\na c_i\|_{L^2}^2 
\le C\mathcal{L}_1 \|\na c_i\|_{L^2}^2 
+ C\mathcal{L}_2
\ee where 
\be 
\mathcal{L}_1 = 1+\|\Delta u\|_{L^2}^2 + \|\na \rho\|_{L^2}^2 + \|\Lambda^{\frac{5}{2}}T\|_{L^2}^4  + \|\na \rho\|_{L^2}^2  \|\Lambda^{\frac{5}{2}}T\|_{L^2}^2 + \|c_i\|_{L^2}^2\|\Lambda^{\frac{5}{2}}T\|_{L^2}^4
\ee 
and 
\be 
\begin{aligned}
\mathcal{L}_2 &=  \|\na \rho\|_{L^2}^2 (1+ \|\Lambda^{\frac{5}{2}}T\|_{L^2}^2) +  \|\Lambda^{\frac{5}{2}}T\|_{L^2}^2 \|\sqrt{T} \na \sqrt{c}_i\|_{L^2}^2
\\&\quad\quad\quad\quad+\left(1   + \|\na c_i\|_{L^2}^2 + \|c_i - \overline{c_i}\|_{L^2} \|\na c_i\|_{L^2}^2 \right)\left(1 + \|\Lambda^{\frac{5}{2}}T\|_{L^2}^2\right) \|\Lambda^{\frac{5}{2}}T\|_{L^2}^2.
\end{aligned}
\ee  
Here we have used the fact that $\overline{c_i} = \overline{c_{i0}}$ is a constant.
By making use of the uniform bound \eqref{squarerootc}, together with the uniform boundedness of $u, c_i,  \nabla c_i, \nabla T$ in $L^2(0, \mathcal{T}, H^2),$ $L^{\infty}(0, \mathcal{T}, L^2)$, $L^2(0, \mathcal{T}, L^2)$, and $L^{\infty}(0, \mathcal{T}, H^{\frac{3}{2}}) \cap L^2(0, \mathcal{T}, H^{\frac{3}{2}})$, respectively and independently of $\mathcal{T}$, we infer that $\mathcal{L}_1$ and $\mathcal{L}_2$ are integrable in time from $0$ to $\mathcal{T}$, with bounds that do not depend on $\mathcal{T}$ but only on the initial data. Applying Gr\"onwall's inequality, we deduce that $\nabla c_i$ is bounded in $L^{\infty}(0, \mathcal{T}, L^2)\cap L^2(0,\mathcal T; H^1)$.

\appendix

\section{Auxiliary lemmas} \label{app1}
In this appendix, we collect some useful auxiliary lemmas.

\beg{lem} \label{molb}
Let $f \in L^2(\TT^2)$. Suppose that $f \geq a >0$ on $\TT^2$. Then it holds that 
\be 
\|\log \mathcal{J}_{\eta}f\|_{L^2}
\le \sqrt{|\TT^2|}|\log a| + \|f\|_{L^2}.
\ee
\end{lem}

\begin{proof}
  By the properties of the convolution operator $\mathcal{J}_{\eta}$, it holds that $\mathcal{J}_{\eta} f \ge a$ on $\TT^2$.  If $a<1$ and $\mathcal{J}_{\eta} f \in [a, 1]$, then $\log \mathcal{J}_{\eta}f \in [\log a, 0]$.  Consequently, it holds that 
  \be \la{o1p}
\int_{\left\{a \le \mathcal{J}_{\eta}f \le 1 \right\}} |\log \mathcal{J}_{\eta} f|^2 \le |\TT^2| |\log a|^2.
  \ee On the other hand, if $a<1$ and $\mathcal{J}_{\eta} f \ge 1$, or if $a\geq 1$ and $\j f \geq a \geq 1$, we have $|\log \j f| \le |\j f|$. As a result,
    \be \la{o2p}
\int_{\left\{\mathcal{J}_{\eta}f \ge 1 \right\}} |\log \mathcal{J}_{\eta} f|^2 \le \|\j f\|_{L^2}^2 \leq \|f\|_{L^2}^2.
  \ee Adding \eqref{o1p} and \eqref{o2p} gives the desired inequality.

\end{proof}

\begin{lem} \la{newloc}
Let $k \ge 2$ be an integer, and let
$C_1,C_2,C_3,C_4\geq0$. Suppose $y(t) \ge 0$ and $z(t)\geq 0$ obey 
\be 
y'(t) \le C_1y(t)^k + C_2 z(t), \qquad  \int_{0}^{t} z(s) ds \le C_3 + C_4t
\ee 
on a time interval $[0, \mathcal{T}]$.
If $y(0) \le c_0$, then there is a positive constant $K$ depending only on $k, c_0,  C_2, C_3$ and a time $\tilde{\mathcal{T}} \leq  \mathcal{T}$ depending only on $k, c_0,C_1, C_2, C_3, C_4$ such that 
\be \label{eqn:lemma-gronwall1}
y(t) \le K
\ee for any $t \in [0, \tilde{\mathcal{T}}]$. 
\end{lem}

\begin{proof}
   For each $t \in [0, \mathcal T]$, we define the modified quantity
   \be 
\tilde{y}(t) = y(t) + A
   \ee where $A \ge 1$ is a positive constant to be determined later. Then $\tilde{y}$ obeys 
   \be 
\tilde{y}'(t) \le C_1 \tilde{y}(t)^k + C_2 z(t).
   \ee Dividing both sides by $\tilde{y}^k$ and using the fact that $\tilde{y} \ge A$, we obtain 
   \be 
\frac{d}{dt} \frac{\tilde y^{1-k}}{1-k} \le C_1 + C_2 \frac{z}{\tilde{y}^k} 
\le C_1 + C_2 A^{-k} z.
   \ee Integrating in time from $0$ to $t$ gives rise to 
   \be 
-\frac{1}{(k-1)\tilde{y}(t)^{k-1}} + \frac{1}{(k-1)\tilde{y}(0)^{k-1}} \le C_1t + C_2 A^{-k} (C_3 + C_4t) 
\le (C_1 + C_2C_4)t + C_2C_3A^{-k}.
   \ee Consequently, it follows that 
   \be 
\frac{1}{(k-1)\tilde{y}(t)^{k-1}}
\ge \frac{1}{(k-1)(c_0 + A)^{k-1}} - (C_1 + C_2C_4)t - C_2C_3A^{-k}.
   \ee We choose $A = A(k, c_0,  C_2, C_3)$ sufficiently large such that 
   \be 
\frac{C_2C_3}{A^k} \le \frac{1}{4(k-1)(c_0 + A)^{k-1}} 
   \ee  and a time $\tilde{\mathcal T} = \tilde{\mathcal T} (k, c_0, C_1, C_2, C_3, C_4)$ sufficiently small such that 
   \be 
(C_1 + C_2 C_4) \tilde{\mathcal{T}} \le \frac{1}{4(k-1)(c_0 + A)^{k-1}}.
   \ee Therefore, it holds that 
   \be 
\tilde{y}(t) \le 2^{\frac{1}{k-1}}(c_0 + A)
   \ee on $[0, \tilde{\mathcal{T}}]$. This implies \eqref{eqn:lemma-gronwall1} with $K= 2^{\frac{1}{k-1}}c_0 +  (2^{\frac{1}{k-1}}-1) A >0.$
\end{proof}

\beg{lem}\label{lemma:iteration}
Suppose $(a_n(t))_{n\in\mathbb N}$ is a sequence of nonnegative smooth functions such that $a_n(0)=a$ is independent of $n$ and $a_0(t) = a$ for $t>0$ with the constant $a>0$. Assume that $a_{n+1}(t)$ obeys the differential inequality
\be 
\frac{d}{dt} a_{n+1} \le C_1a_n^m + C_2(a_n^k +1) a_{n+1} + C_3
\ee for any $n \in \N$, where  $C_1, C_2, C_3$ are some nonnegative constants (independent of $n$) and $m, k \in \N$. Then there exists a time $T_0= T_0(C_1,C_2,C_3,m,k,a)>0$ such that
\be 
\sup_{n\in\mathbb N}\sup\limits_{0\le t \le T_0}a_n(t) \le e^{4a}.
\ee

\end{lem}

\begin{proof}
    An application of Gr\"onwall's inequality gives the instantaneous bound 
    \be 
a_{n+1}(t) \le \left(a+ C_1 \int_{0}^{t}a_n^m ds + C_3 t \right) \exp \left\{C_2 \int_{0}^{t} (a_n^k +1) ds \right\}
    \ee for any $t \ge 0$. For $n \in \N$ and $T > 0$, let 
    \be 
A_n(T) = \sup \left\{a_n(t) : 0 \le t \le T \right\}.
    \ee It holds that 
    \be 
A_{n+1}(T) \le \exp \left\{a + C_1T A_n^m(T) + C_2TA_n^k(T) +(C_2 + C_3)T \right\}
    \ee for any $n \in \N$ and $T > 0$. Letting 
    \be 
T_0 = \min \left\{1,\frac{a}{1+C_2 + C_3}, \frac{a}{1+C_1 e^{4am}}, \frac{a}{1+C_2e^{4ak}} \right\},
    \ee we obtain 
    \be 
A_{n+1}(T_0) \le e^{4a}
    \ee via a classical induction argument. The latter gives the desired bound. \
\end{proof}

\section{Local existence of solutions to the mollified system}
\label{sec:appendix-b}

In this appendix, we prove the local existence of solutions to the $\eta$-regularized system for each $\eta>0$. We drop the superscript $\eta$ in order to make the notation more readable. As for the initial data, we keep the superscript $\eta$ to distinguish between the mollified and unmollified initial data.

\subsection{An iteration scheme}

For each $n\geq 0$, we consider the iterative system
\noeqref{eqn:npb-full-mo-iteration-1}
\noeqref{eqn:npb-full-mo-iteration-2}
\noeqref{eqn:npb-full-mo-iteration-3}
\noeqref{eqn:npb-full-mo-iteration-4}
\noeqref{eqn:npb-full-mo-iteration-5}
\begin{subequations}\label{sys:npb-full-mo-iteration}
    \begin{align}
        &\partial_t c_i^{(n+1)} + \mathcal J_\eta   u^{(n+1)}\cdot \nabla c_i^{(n+1)}  \nonumber
        \\
        &\hspace{1cm}= D_i\nabla\cdot\left(\nabla c_i^{(n+1)} + \frac{ez_i}{k_B T^{(n+1)}} c_i^{(n+1)} \nabla \j \phi^{(n)} + \chi_i^{(n)}c_i^{(n+1)}\log c_i^{(n)}\j\nabla \log  T^{(n+1)}\right),\label{eqn:npb-full-mo-iteration-1}
        \\
        &-\varepsilon\Delta \phi^{(n+1)} = \rho^{(n+1)} = \sum\limits_{i=1}^N e  z_i c_i^{(n+1)}, \label{eqn:npb-full-mo-iteration-2}
        \\
        &\partial_t u^{(n+1)} + \mathcal J_\eta u^{(n)}\cdot \nabla u^{(n+1)} - \nu\Delta u^{(n+1)} + \nabla p^{(n+1)}= g\alpha_T(T^{(n+1)}-T_r) \vec{k} - \mathcal J_\eta   \left( \rho^{(n)} \nabla \j \phi^{(n)} \right), \label{eqn:npb-full-mo-iteration-3}
        \\
        &\nabla\cdot u^{(n+1)} = 0,\label{eqn:npb-full-mo-iteration-4} 
        \\
        &\partial_t T^{(n+1)} + \mathcal J_\eta   u^{(n)} \cdot \nabla T^{(n+1)} - \kappa\Delta T^{(n+1)} = 0,\label{eqn:npb-full-mo-iteration-5}
    \end{align}
\end{subequations}
with initial conditions 
\[
(c_i^{(n+1)}(0),u^{(n+1)}(0),T^{(n+1)}(0))=(\j c_{i0},\j u_0,\j T_0):=(c^\eta_{i0},u^\eta_0,T^\eta_0).
\]
Here the cutoff $\chi_i^{(n)}:=\chi^\eta(c_i^{(n)})$ satisfies $\chi_i^{(n)} = 1$ when $c_i^{(n)}(x,t) \geq 2\eta$ and $\chi_i^{(n)} = 0$ when $c_i^{(n)}(x,t) \leq \eta$. Note that the term $\chi^{(n)} \log c_i^{(n)}$ is well-defined even when $c_i^{(n)} \le 0$ as the cutoff function nullifies the latter when $c_i^{(n)} \le \eta$. The constant $T_r= \overline{T_0^\eta}$ is fixed and independent of the spatial-temporal variables and the index $n$. The compatibility condition of the initial ionic concentrations \[
\sum_{i=1}^N\int_{\mathbb T^2}  z_i c_{i0}^\eta dx = \sum_{i=1}^N  \int_{\mathbb T^2}  z_i c_{i0} dx =0
\]
holds due to the properties of mollifiers.

In addition, we set
\[
 c_i^{(0)} = c^\eta_{i0}, \quad u^{(0)}= u^\eta_0, \quad T^{(0)} = T_0^\eta, \quad -\varepsilon\Delta \phi^{(0)} = \rho^{(0)} = \sum\limits_{i=1}^N e z_i c_i^{(0)},
\]
which are all constants in time. 

For $n=0$, as $u^{(0)}$ is smooth and \eqref{eqn:npb-full-mo-iteration-5} is a linear drift-diffusion equation with periodic boundary conditions and smooth initial conditions, it has a unique smooth solution $T^{(1)}$ for any $t\geq 0$. Moreover, as $T_0^\eta\geq T^*$, we have $T^{(1)}\geq T^*.$
Next, since equation \eqref{eqn:npb-full-mo-iteration-3} is linear in $u^{(1)}$ and has smooth forcing terms depending on $T^{(1)}$ and $c_i^{(0)}$, and $u^{(1)}$ is transported by a smooth velocity field $\mathcal J_\eta u^{(0)}$, we infer the existence of a unique smooth solution $u^{(1)}$ for any $t\geq 0$. As $T^{(1)}\geq T^*$, the terms $\frac{1}{T^{(1)}}$ and $\j\nabla \log T^{(1)}$ are smooth. 
The logarithmic term  $\chi^{(0)} \log c_i^{(0)}$ is smooth since it vanishes when $c_i^{(0)} \le \eta$.  %(indeed, as $c_i^{(0)}=c_{i0}^\eta\geq a>0$, we automatically have $\log c_i^{(0)}$ is smooth). 
Since equation \eqref{eqn:npb-full-mo-iteration-1} is linear in $c_i^{(1)}$ with all other functions in the equation being smooth, one can obtain a unique smooth solution $c_i^{(1)}$ for each $i=1,2,\dots,N$ and for all $t\geq 0$.

Next we show that $c_i^{(1)}\geq 0$ for all times. Indeed, we compute
\begin{align*}
    \int_0^{\mathcal T} \|\frac{1}{T^{(1)}}\nabla\j\phi^{(0)}\|^2_{L^\infty} + (1+ \|c_i^{(0)}\|_{L^\infty}^2) \|\j\nabla \log T^{(1)}\|_{L^\infty}^2 dt <\infty
\end{align*}
for any $\mathcal T>0$ since $T^{(1)}$, $\phi^{(0)}$, and $c_i^{(0)}$ are all smooth functions and $T^{(1)} \geq T^*$.
By virtue of Proposition~\ref{nonneg}, we can conclude that $c_i^{(1)}\geq 0$ on $[0,\mathcal T]$.

Having constructed global smooth solutions $c_i^{(1)}\geq 0$, $u^{(1)}$, $T^{(1)}\geq T^*$, $\rho^{(1)}$, and $\phi^{(1)}$ for the first iteration corresponding to $n=0$,  we can proceed in the same manner and obtain another family of solutions $c_i^{(2)}\geq 0$, $u^{(2)}$, $T^{(2)}\geq T^*$, $\rho^{(2)}$, and $\phi^{(2)}$ corresponding to the next iteration $n=1$. Repeating the same argument, one can obtain a sequence of smooth solutions $c_i^{(n)}$, $u^{(n)}$, $T^{(n)}$, $\rho^{(n)}$, and $\phi^{(n)}$ for each $n\in \NN$. In addition, it holds that $T^{(n)}\geq T^*$ and $c_i^{(n)} \geq 0$ for all times $t \ge 0$ and all $n \in \NN$. As $\overline{T^{(n)}} = \overline{T_0^\eta}$ and $\overline{c_i^{(n)}} = \overline{c_{i0}^\eta}$ are invariant in time, one has $\overline{u^{(n)}} = \overline{u_0^\eta} = 0$.

\subsection{Uniform energy estimates}\label{section:appendix-2}
Having constructed a sequence of solutions for $n\in \mathbb N$, we now establish the uniform-in-$n$ energy estimates that are needed to prove the local well-posedness and regularity of the $\eta$-mollified system. These estimates will be performed in several steps.

\smallskip

\noindent {\bf Step 1. Uniform bounds for $\|T^{(n+1)}\|_{L^2}$.} By integrating by parts and using the divergence-free condition of $\mathcal J_\eta  u^{(n)}$, we infer that the $L^2$ norm of the iterative temperature $T^{(n+1)}$ evolves according to 
$$
\frac{1}{2} \frac{d}{d t}\|T^{(n+1)}\|_{L^2}^2+\kappa\|\nabla T^{(n+1)}\|_{L^2}^2=0,
$$
which implies that for any $t \ge 0$,
\begin{align}\label{est:iteration-TL2}
    \|T^{(n+1)}(t)\|_{L^2}^2 + 2\kappa \int_0^t \|\nabla T^{(n+1)}(s)\|_{L^2}^2 ds = \|T^\eta_0\|_{L^2}^2.
\end{align}

\noindent {\bf Step 2. Uniform bounds for $\|c_i^{(n+1)}\|_{L^2}$.}
The $L^2$ norm of each ionic concentration $c_i^{(n+1)}$ satisfies
\begin{align*}
    \frac{1}{2} \frac{d}{d t}\|c_i^{(n+1)}\|_{L^2}^2 + D_i \|\nabla c_i^{(n+1)}\|_{L^2}^2= &-D_i\frac{ez_i}{k_B} \int_{\mathbb{T}^2} \frac{1}{T^{(n+1)}} c_i^{(n+1)} \nabla c_i^{(n+1)} \nabla \j \phi^{(n)}  dx
    \\
    &-D_i \int_{\mathbb{T}^2} c_i^{(n+1)} \nabla c_i^{(n+1)} \chi_i^{(n)} \log c_i^{(n)} \j \nabla \log T^{(n+1)} dx ,
\end{align*}
due to the divergence-free condition obeyed by $\mathcal J_\eta   u^{(n+1)}$. 
Applications of the H\"older, Ladyzhenskaya, and Young inequalities give rise to the energy estimate
\begin{align*}
   &\frac{1}{2} \frac{d}{d t}\|c_i^{(n+1)}\|_{L^2}^2 + \frac{1}{2}D_i \|\nabla c_i^{(n+1)}\|_{L^2}^2 
   \\
   \leq &C \left\|\frac{1}{T^{(n+1)}}\right\|_{L^\infty}^2 \|\nabla \j \phi^{(n)} \|_{L^\infty}^2  \|c_i^{(n+1)}\|_{L^2}^2 + C\|\j \nabla\log T^{(n+1)}\|_{L^\infty}^2 \|\chi_i^{(n)} \log c_i^{(n)}\|_{L^4}^2 \|c_i^{(n+1)}\|_{L^2}^2
   \\
   &+ C\|\j \nabla\log T^{(n+1)}\|_{L^\infty}^4 \|\chi_i^{(n)} \log c_i^{(n)}\|_{L^4}^4 \|c_i^{(n+1)}\|_{L^2}^2.
\end{align*}
Using the smoothing  properties of the mollifiers and employing a standard duality argument, one has
\begin{align}
    &\|\nabla \j \phi^{(n)} \|_{L^\infty} \leq C\|\j  \phi^{(n)}\|_{H^3} \leq C_\eta \| \rho^{(n)}\|_{H^{-2}}\leq C_\eta \| \rho^{(n)}\|_{L^1} \leq C_\eta \sum\limits_{i=1}^N \|c_i^{(n)}\|_{L^1} = C_\eta \sum\limits_{i=1}^N \overline{c_{i0}^\eta}, \label{est:iteration-duality}
\end{align}
and
\begin{align*}
    \|\j \nabla \log T^{(n+1)}\|_{L^\infty} \leq C \|\j \log T^{(n+1)}\|_{H^3} \leq C_\eta\|\log T^{(n+1)}\|_{L^2} \leq C_\eta (1+ \|T^{(n+1)}\|_{L^2}),
\end{align*}
 where we have used $c_i^{(n+1)}\geq 0$, $T^{(n+1)}\geq T^*$, and the fact that the spatial means of $c_i^{(n+1)}$ are invariant in time. Moreover, thanks to the cutoff function, the logarithmic term does not blow up and can be bounded uniformly in $n$ as follows,
 \begin{align*}
    \|\chi_i^{(n)} \log c_i^{(n)}\|_{L^4} \leq C_\eta (1+ \|c_i^{(n)}\|_{L^1}) = C_\eta (1+\overline{c_{i0}^\eta}).
\end{align*}
 Consequently, we deduce that
\begin{align*}
   \frac{d}{d t}\|c_i^{(n+1)}\|_{L^2}^2 + D_i \|\nabla c_i^{(n+1)}\|_{L^2}^2 \leq  C_\eta(1+\sum\limits_{j=1}^N \overline{c_{j0}^\eta}^2  + \|T^{(n+1)}\|_{L^2}^4 \overline{c_{i0}^\eta}^4) \|c_i^{(n+1)}\|_{L^2}^2.
\end{align*}
Applying Gr\"onwall's inequality and using \eqref{est:iteration-TL2}, we have
\begin{align}\label{est:iteration-ciL2}
    \|c_i^{(n+1)}(t)\|_{L^2}^2 + \int_0^t D_i\|\nabla c_i^{(n+1)}(s)\|_{L^2}^2 ds \leq \|c_{i0}^\eta\|_{L^2}^2 e^{\widetilde{C}t},
\end{align} 
for any $t \ge 0$, where $\widetilde{C}$ only depends on the initial conditions and $\eta$ but is independent of $n$.

\smallskip

\noindent {\bf Step 3. Uniform bounds for $\|u^{(n+1)}\|_{L^2}$.} 
The $L^2$ norm of the iterative velocity $u^{(n+1)}$ satisfies
\begin{align*}
    \frac{1}{2} \frac{d}{d t}\|u^{(n+1)}\|_{L^2}^2+\nu\|\nabla u^{(n+1)}\|_{L^2}^2=&-\int_{\mathbb{T}^2} \mathcal J_\eta  \left(\rho^{(n)} \nabla\j \phi^{(n)} \right)\cdot u^{(n+1)} dx 
    \\
    &+ \int_{\mathbb{T}^2} g\alpha_T (T^{(n+1)}-T_r) u^{(n+1)}_2 dx.
\end{align*}
Using the mean-free property obeyed by $u^{(n+1)}$, and applying the  H\"older, Young, Poincar\'e and Sobolev inequalities, we have
\begin{align*}
    &\frac{1}{2} \frac{d}{d t}\|u^{(n+1)}\|_{L^2}^2+\nu\|\nabla u^{(n+1)}\|_{L^2}^2
    \\
    \leq &C \Big(\sum\limits_{i=1}^N \|c_i^{(n)}\|_{L^2} \|\nabla \j \phi^{(n)} \|_{L^\infty} + \|T^{(n+1)}\|_{L^2} \Big) \|u^{(n+1)}\|_{L^2}
    \\
    \leq &C_\eta(\sum\limits_{i=1}^N \|c_i^{(n)}\|_{L^2}^2 + \|T^{(n+1)}\|_{L^2}^2) + \frac\nu2 \|\na u^{(n+1)}\|_{L^2}^2,
\end{align*}
where we exploited the uniform bound  \eqref{est:iteration-duality} obeyed by the mollified iterative potential. As $c_i^{(n)}$ obeys the same estimate derived for $c_i^{(n+1)}$ in \eqref{est:iteration-ciL2}, and thanks to \eqref{est:iteration-TL2}, the latter differential inequality boils down to
\begin{align*}
    \frac{d}{d t}\|u^{(n+1)}\|_{L^2}^2+\nu\|\nabla u^{(n+1)}\|_{L^2}^2\leq C_\eta(\sum_{i=1}^N\|c_{i0}^\eta\|_{L^2}^2 e^{\widetilde{C}t} + \|T_0^\eta\|_{L^2}^2).
\end{align*}
An application of Gr\"onwall's inequality yields 
\begin{align}\label{est:iteration-uL2}
    \|u^{(n+1)}(t)\|_{L^2}^2 + \nu \int_0^t \|\nabla u^{(n+1)}(s)\|_{L^2}^2 ds \leq \|u_0^\eta\|_{L^2}^2 + C_\eta(\sum\limits_{i=1}^{N}\|c_{i0}^\eta\|_{L^2}^2  e^{\widetilde{C}t} + \|T_0^\eta\|_{L^2}^2 t).
\end{align} 

From \eqref{est:iteration-TL2}, \eqref{est:iteration-ciL2}, and \eqref{est:iteration-uL2}, one can infer that for any time $\mathcal T>0$, $\{(c_i^{(n+1)}, u^{(n+1)}, T^{(n+1)})\}$ are uniformly bounded in $L^\infty(0,\mathcal T;L^2) \cap L^2(0,\mathcal T; H^1)$. Observe that at the $L^2$ level, these bounds are global in time.

\smallskip

\noindent {\bf Step 4. Higher order estimates.} Next we will perform $H^m$ estimates for $\{(c_i^{(n+1)}, u^{(n+1)}, T^{(n+1)})\}$ with $m\geq 2$. Consider a multi-index $\alpha\in \mathbb N^2$ such that $|\alpha|\leq m$. %Thanks to the Poincar\'e inequality and since we have already obtained bounds in $L^2$, we can directly estimate the homogeneous $H^m$ norms.

The $H^m$ evolution of $T^{(n+1)}$ reads
\begin{align*}
    \frac12 \frac{d}{dt} \| T^{(n+1)}\|_{H^m}^2 + \kappa \|\nabla T^{(n+1)}\|_{H^m}^2 = \sum\limits_{0\leq |\alpha|\leq m}\int_{\TT^2} D^\alpha (\j u^{(n)}  T^{(n+1)})\cdot D^\alpha \nabla T^{(n+1)} dx.
\end{align*}
As $\mathcal J_\eta  u^{(n)}$ is divergence-free, we have $\langle  \mathcal J_\eta  u^{(n)} \cdot \nabla D^\alpha T^{(n+1)}, D^\alpha T^{(n+1)}\rangle = 0$.  We use the H\"older and the Sobolev inequalities to  bound 
\begin{equation}\label{est-mollifed-reg-1}
    \begin{split}
        &\left|\sum\limits_{0\leq |\alpha|\leq m} \langle D^\alpha (\mathcal J_\eta  u^{(n)} \cdot \nabla T^{(n+1)}), D^\alpha T^{(n+1)}\rangle\right| 
        \\
        = &\left| \sum\limits_{0\leq |\alpha|\leq m} \sum\limits_{0<\beta\leq \alpha} \langle D^\beta \mathcal J_\eta  u^{(n)} \cdot D^{\alpha-\beta} \nabla T^{(n+1)}, D^\alpha T^{(n+1)}\rangle\right|
   \\
   \leq & C \sum\limits_{0\leq |\alpha|\leq m} \sum\limits_{0<\beta\leq \alpha} \|D^\beta \mathcal J_\eta  u^{(n)}\|_{H^2} \|D^{\alpha-\beta} \nabla T^{(n+1)}\|_{L^2} \|D^\alpha T^{(n+1)}\|_{L^2}
   \\
   \leq & C_{m,\eta} \|u^{(n)}\|_{L^2} \|T^{(n+1)}\|_{H^m}^2.
    \end{split}
\end{equation}
Thus it follows that
\begin{align*}
    \frac{d}{dt}\|T^{(n+1)}\|_{H^m}^2 + 2\kappa \|\nabla T^{(n+1)}\|_{H^m}^2 \leq  C_{m,\eta} \|u^{(n)}\|_{L^2} \|T^{(n+1)}\|_{H^m}^2.
\end{align*}
Application of Gr\"onwall's inequality and the regularity criterion $u^{(n)}\in L^\infty(0,\mathcal T; L^2)$ (that holds for any $\mathcal T>0$) yields
\begin{equation}\label{est-mollifed-reg-2}
    \sup_{0\leq t\leq\mathcal T}\|T^{(n+1)}(t)\|_{H^m}^2 + 2\kappa \int_0^{\mathcal T} \|\nabla T^{(n+1)}(s)\|_{H^m}^2 ds \leq C,
\end{equation}
with $C$ independent of $n$. Note that here $\mathcal T>0$ is arbitrary. %Here we have used the fact that $\|u_0^\eta\|_{L^2} \leq \|u_0\|_{L^2}$ and $\|T_0^\eta\|_{H^m} \leq C_{m,\eta}\|T_0\|_{L^2}$.
%Typically, as the initial condition $T_0^\eta \in C^\infty(\mathbb T^2)$, one can perform the estimate above for any $m\in \mathbb N$ and eventually conclude that
% \begin{equation}\label{mollfied-T-reg}
%     T \in C([0,\mathcal T]; C^\infty(\mathbb T^2)).
% \end{equation}

Next, we consider the $H^m$ evolution of $u^{(n+1)}$, which is given by 
\begin{equation} \label{est-mollifed-reg-4}
    \begin{split}
        &\frac{d}{dt}\|u^{(n+1)}\|_{H^m}^2 + 2\nu \|\nabla u^{(n+1)}\|_{H^m}^2 
     \\
     = &- \sum\limits_{0\leq |\alpha|\leq m} \langle D^\alpha (\mathcal J_\eta  u^{(n)} \cdot \nabla u^{(n+1)}), D^\alpha u^{(n+1)}\rangle 
     + \sum\limits_{0\leq |\alpha|\leq m}  g\alpha_T \langle D^\alpha( T^{(n+1)}-T_r), D^\alpha u^{(n+1)}_2 \rangle 
     \\
     &- \sum\limits_{0\leq |\alpha|\leq m} \langle D^\alpha\mathcal J_\eta \left(\rho^{(n)}\nabla\j \phi^{(n)}\right), D^\alpha u^{(n+1)} \rangle. 
    \end{split}
\end{equation}
The first nonlinear term on the right-hand side can be estimated similarly as in \eqref{est-mollifed-reg-1}, and one obtains
\begin{align*}
    \left|\sum\limits_{0\leq |\alpha|\leq m} \langle D^\alpha (\mathcal J_\eta  u^{(n)} \cdot \nabla u^{(n+1)}), D^\alpha u^{(n+1)}\rangle\right| \leq C_{m,\eta} \|u^{(n)}\|_{L^2} \|u^{(n+1)}\|_{H^m}^2.
\end{align*}
By H\"older's and Young's inequalities, we have
\begin{align*}
    \left| \sum\limits_{0\leq |\alpha|\leq m}  g\alpha_T \langle D^\alpha( T^{(n+1)}-T_r), D^\alpha u^{(n+1)}_2 \rangle \right| 
    \leq  C  \|T^{(n+1)}\|_{H^m}^2 + C\|u^{(n+1)}\|_{H^m}^2.
\end{align*}
Using in addition the bound \eqref{est:iteration-duality}, we estimate
\begin{align*}
    &\left|\sum\limits_{0\leq |\alpha|\leq m} \langle D^\alpha\mathcal J_\eta \left(\rho^{(n)}\nabla\j \phi^{(n)}\right), D^\alpha u^{(n+1)} \rangle  \right| 
    \\
    \leq &C_{m,\eta}\|\rho^{(n)}\|_{L^2}\|\nabla\j \phi^{(n)}\|_{L^\infty} \|u^{(n+1)}\|_{H^m}
    \leq C_{m,\eta} \left(\sum\limits_{i=1}^N \|c_i^{(n)}\|_{L^2}^2  + \|u^{(n+1)}\|_{H^m}^2\right).
\end{align*}
Combining the estimates above for the right-hand side of \eqref{est-mollifed-reg-4}, one has
\begin{align*}
     &\frac{d}{dt}\|u^{(n+1)}\|_{H^m}^2 + 2\nu \|\nabla u^{(n+1)}\|_{H^m}^2 
     \\
     \leq &C_{m,\eta} (\sum\limits_{i=1}^N \|c_i^{(n)}\|_{L^2}^2 + \|T^{(n+1)}\|_{H^m}^2) + C_{m,\eta}(1+\|u^{(n)}\|_{L^2}) \|u^{(n+1)}\|_{H^m}^2.
\end{align*}
By Gr\"onwall's inequality and thanks to \eqref{est:iteration-ciL2} and \eqref{est-mollifed-reg-2}, we deduce that
\begin{align*}
    \sup_{0\leq t\leq\mathcal T}\|u^{(n+1)}(t)\|_{H^m}^2 + 2\nu \int_0^{\mathcal T} \|\nabla u^{(n+1)}(s)\|_{H^m}^2 ds \leq C,
\end{align*} 
with $C$ independent of $n$. Note that here $\mathcal T>0$ is arbitrary.
% As $u_0^\eta \in C^\infty(\mathbb T^2)$, the above is true for any $m\in \mathbb N$, and therefore,
% \begin{equation}\label{mollfied-u-reg}
%     u \in C([0,\mathcal T]; C^\infty(\mathbb T^2)).
% \end{equation}

Now we move to the $H^m$ estimate of $c_i^{(n+1)}$. We have
\begin{align*}
    &\frac12 \frac{d}{dt} \|c_i^{(n+1)}\|_{H^m}^2 + D_i \|\nabla c_i^{(n+1)}\|_{H^m}^2 
    \\
    = &-\sum\limits_{0\leq |\alpha|\leq m}\int_{\TT^2} D^\alpha (\j u^{(n+1)} \cdot\nabla c_i^{(n+1)})\cdot D^\alpha c_i^{(n+1)} dx 
    \\
    &- \frac{D_i ez_i}{k_B} \sum\limits_{0\leq |\alpha|\leq m} \int_{\TT^2}D^\alpha( \frac{1}{T^{(n+1)}} c_i^{(n+1)} \nabla \j\phi^{(n)}) \cdot D^\alpha\nabla c_i^{(n+1)} dx
    \\
    &-D_i\sum\limits_{0\leq |\alpha|\leq m} \int_{\TT^2} D^\alpha(\chi_i^{(n)}c_i^{(n+1)}\log c_i^{(n)}\j\nabla \log  T^{(n+1)})  \cdot D^\alpha\nabla c_i^{(n+1)}dx
\end{align*}
The first nonlinear term on the right-hand side can be estimated similarly as in \eqref{est-mollifed-reg-1}, and one can get
\begin{align*}
    \left|\sum\limits_{0\leq |\alpha|\leq m} \langle D^\alpha (\mathcal J_\eta  u^{(n+1)} \cdot \nabla c^{(n+1)}_i), D^\alpha c^{(n+1)}_i\rangle\right| \leq C_{m,\eta} \|u^{(n+1)}\|_{L^2} \|c_i^{(n+1)}\|_{H^m}^2 \leq C_{m,\eta} \|c_i^{(n+1)}\|_{H^m}^2 .
\end{align*}
For any $s\in\mathbb N$ and $|\alpha|=s$, the derivative $D^\alpha (\frac1f)$ consists of terms whose denominators are $f^{s'}$ with $s'\leq s+1$, and whose numerators are the product of derivatives of $f$ with order at most $s$. Thanks to the fact that $T^{(n+1)}\geq T^*$ and \eqref{est-mollifed-reg-2}, one can conclude that
\[
  \left\|D^\alpha (\frac1{T^{(n+1)}})\right\|_{L^2} \leq C \left\| \frac1{T^{(n+1)}}\right\|_{H^{m}} \leq C_{m,\eta,\mathcal T, \|u_0\|_{L^2}, \|T_0\|_{L^2}, T^*}.
\]
Using an estimate similar to \eqref{est:iteration-duality} with $3$ replaced by $m+1$, we have, for any $|\alpha|\leq m$, that
\begin{align*}
    &\|D^\alpha( \frac{1}{T^{(n+1)}} c_i^{(n+1)} \nabla \j\phi^{(n)})\|_{L^2} 
    \\
    \leq &C \left\|\frac{1}{T^{(n+1)}}\right\|_{H^m} \|c_i^{(n+1)}\|_{H^m} \| \j\phi^{(n)}\|_{H^{m+1}} \leq C \|c_i^{(n+1)}\|_{H^m},
\end{align*}
where the constant $C$ depends on $m,\eta,\mathcal T, T^*$, and the initial conditions. For the last nonlinear term on the right-hand side, we apply Proposition~\ref{prop:est-logterm} to get
\begin{align*}
    \|D^\alpha(\chi_i^{(n)}c_i^{(n+1)}\log c_i^{(n)}\j\nabla \log  T^{(n+1)})\|_{L^2} \leq &C_{\eta,m} \|c_i^{(n+1)}\|_{H^m}(1+\|c_i^{(n)}\|^m_{H^m})  \|\j \nabla \log T^{(n+1)}\|_{H^m}
    \\
    \leq &C_{\eta,m} \|c_i^{(n+1)}\|_{H^m}(1+\|c_i^{(n)}\|^m_{H^m}) 
\end{align*}
thanks to \eqref{est-mollifed-reg-2}. Therefore, by applying H\"older's and Young's inequalities we have
\begin{align*}
    \frac{d}{dt} \|c_i^{(n+1)}\|_{H^m}^2 + D_i \|\nabla c_i^{(n+1)}\|_{H^m}^2 \leq C(1+\|c_i^{(n)}\|^{2m}_{H^m}) \|c_i^{(n+1)}\|_{H^m}^2,
\end{align*}
where the constant $C$ depends on $m,\eta,\mathcal T, T^*$, and the initial conditions, but not on $n$. Taking the summation over $i$, we have
\begin{align}\label{iter-old}
    \frac{d}{dt} \sum_{i=1}^N\|c_i^{(n+1)}\|_{H^m}^2 + \sum_{i=1}^N D_i \|\nabla c_i^{(n+1)}\|_{H^m}^2\leq C(1+\sum_{i=1}^N\|c_i^{(n)}\|^{2m}_{H^m}) \sum_{i=1}^N\|c_i^{(n+1)}\|_{H^m}^2.
\end{align}
For each $n\in\mathbb N$, denote by 
\[
a_n(t) := \sum_{i=1}^N\|c_i^{(n)}(t)\|_{H^m}^2,\quad a :=a_n(0)= \sum_{i=1}^N\|c_i^{(n)}(0)\|_{H^m}^2 = \sum_{i=1}^N\|c_{i0}^\eta\|_{H^m}^2.
\]
Recall that for step $n=0$ we have $a_0(t) = a$. Then \eqref{iter-old} implies that
\begin{align*}
    \frac d{dt} a_{n+1} \leq C(1+ a_n^{m}) a_{n+1}.
\end{align*}
If $a>0$, by applying Lemma~\ref{lemma:iteration} we conclude there exists a time $\mathcal T_{\eta,m}>0$ such that 
\begin{align}\label{appendix:uni-bdd-ci}
    \sup_{n\in\mathbb N}\sup_{0\leq t \leq T_{\eta,m}} a_n(t) \leq C_0
\end{align}
for some constant $C_0>0$ that is independent of $n$. If $a=0$, then clearly $a_n(t)=0$ for all $n\in\mathbb N$ and any $t\geq 0$. Thus \eqref{appendix:uni-bdd-ci} still holds. Moreover, integrating \eqref{iter-old} in time from $0$ to $\mathcal T_{\eta,m}$ and using \eqref{appendix:uni-bdd-ci}, we obtain that 
\[
\sup_{n\in\mathbb N}\int_0^{\mathcal T_{\eta,m}} \sum_{i=1}^N D_i \|\nabla c_i^{(n+1)}\|_{H^m}^2 \leq \widetilde{C_0}
\]
for some constant $\widetilde{C_0}>0$ that is independent of $n$.

In conclusion, for each fixed $\eta>0$ and $m\in\mathbb N$, the family of iterative solutions $\{u^{(n)}, T^{(n)}\}$ is uniformly bounded in $L^\infty (0,\mathcal T; H^m)\cap L^2(0,\mathcal T; H^{m+1})$ for any $\mathcal T>0$, and 
there exists a time $\mathcal T_{\eta,m}>0$ that is independent of $n$ such that $\{c_i^{(n)}\}$ is uniformly bounded in $L^\infty (0,\mathcal T_{\eta,m}; H^m)\cap L^2(0,\mathcal T_{\eta,m}; H^{m+1})$. Without loss of generality, we assume $\mathcal T_{\eta,m}$ is decreasing as $m$ increases.

\medskip

\noindent {\bf Step 5. Estimates of time derivatives.} In order to prove the strong convergence through the Aubin-Lions compactness theorem, we will
derive uniform-in-$n$ bounds for the time derivatives. It will be enough to establish the bounds in $L^2$. We start with $\partial_t T^{(n+1)}$, which is controlled by
\begin{align*}
    \|\partial_t T^{(n+1)}\|_{L^2} \leq \|\j u^{(n)} \cdot \nabla T^{(n+1)}\|_{L^2} + \kappa \|\Delta T^{(n+1)}\|_{L^2}\leq C(\|u^{(n)}\|_{L^2} \|T^{(n+1)}\|_{H^2} + \|T^{(n+1)}\|_{H^2}).
\end{align*}
This implies that $\{\partial_t T^{(n+1)}\}$ are uniformly bounded in $L^\infty(0,\mathcal T; L^2)$ for any $\mathcal T>0$ due to the uniform bounds of $\{u^{(n)},T^{(n)}\}$ derived in Step 4.

For $\partial_t u^{(n+1)}$, we consider a test function $\psi\in H$ and estimate
\begin{align*}
    &|\langle \partial_t u^{(n+1)}, \psi\rangle| 
    \\
    \leq &|\langle \j u^{(n)}\cdot\nabla u^{(n+1)}, \psi\rangle| + \nu |\langle\Delta u^{(n+1)}, \psi\rangle| + g\alpha_T |\langle T^{(n+1)}-T_r, \psi_2\rangle| + |\j(\rho^{(n)}\nabla\j\phi^{(n)}),\psi\rangle|
    \\
    \leq & C\left(\|u^{(n)}\|_{L^2}\|u^{(n+1)}\|_{H^1} + \|u^{(n+1)}\|_{H^2} + \|T^{(n+1)}\|_{L^2} + \sum_{i=1}^N \|c_i^{(n)}\|_{L^2}\sum_{i=1}^N \overline{c_{i0}^\eta}\right)\|\psi\|_{L^2}.
\end{align*}
Due to the uniform bounds of $\{u^{(n)},T^{(n)}\}$ and \eqref{est:iteration-ciL2}, we have $\{\partial_t u^{(n+1)}\}$ uniformly bounded in $L^\infty(0,\mathcal T; L^2)$ for any $\mathcal T>0$.

Finally, for $\partial_t c_i^{(n+1)}$, we have 
\begin{align*}
    \|\partial_t c_i^{(n+1)}\|_{L^2} \leq &\|\j u^{(n+1)}\cdot\nabla c_i^{(n+1)}\|_{L^2}  
    \\
    &+ D_i \left\|\nabla c_i^{(n+1)}+\frac{ez_i}{k_B T^{(n+1)}} c_i^{(n+1)} \nabla \j \phi^{(n)} + \chi_i^{(n)}c_i^{(n+1)}\log c_i^{(n)}\j\nabla \log  T^{(n+1)}\right\|_{H^1}
    \\
    \leq &C\Big(\|u^{(n+1)}\|_{H^2} \|\na c_i^{(n+1)}\|_{L^2} + \|c_i^{(n+1)}\|_{H^2} + \|\frac1{T^{(n+1)}}\|_{H^2}\|c_i^{(n+1)}\|_{H^2} \sum_{i=1}^N \overline{c_{i0}^\eta} 
    \\
    &\hspace{1cm}+(|\log\eta|+\|c_i^{(n)}\|_{H^2})\|c_i^{(n+1)}\|_{H^2}(|\log T^*|+\|T^{(n+1)}\|_{H^2})\Big).
\end{align*} 
Thanks to the uniform bounds in Step 4, we have $\{\partial_t c_i^{(n+1)}\}$ uniformly bounded in $L^\infty(0,\mathcal T_{\eta,2}; L^2)$ for some $\mathcal T_{\eta,2}>0$. Since we only have local in time uniform bounds for higher order norms of $c_i^{(n)}$, we end up with local in time uniform bounds for $\|\partial_t c_i^{(n+1)}\|_{L^2}$.

\subsection{Local existence of solutions} 
The sequence of iterative solutions $\left\{c_i^{(n)}, u^{(n)}, T^{(n)} \right\}$ is a contraction in $C(0, T_{\eta}; L^2)$  on a short time interval $[0, T_{\eta}]$, a fact that follows from the uniform-in-$n$ boundedness of the solutions and classical energy arguments. Consequently, all converging subsequences in certain functional spaces will converge to the same limit on $[0, T_{\eta}]$.

In view of the uniform-in-$n$ bounds of $\{c_i^{(n)},u^{(n)}, T^{(n)}\}$ and their derivatives, we can apply the Banach-Alaoglu and Aubin-Lions theorems and deduce that for each fixed $\eta>0$ and $m\geq 2$, there exist a time $T_{\eta,m}>0$, a limit $\{c^\eta_i,u^\eta,T^\eta\}$, and a subsequence of $\{c_i^{(n)},u^{(n)}, T^{(n)}\}$, still indexed by $n$, such that 
\begin{equation}\label{appendix:convergence}
    \begin{split}
        &(c_i^{(n)}, u^{(n)}, T^{(n)}) \to (c^\eta_i, u^\eta, T^\eta) \;\; \text{in} \;\; C([0,\mathcal T_{\eta,m}]; H^{m-1}),
    \\
    &(c_i^{(n)}, u^{(n)}, T^{(n)}) \overset{\ast}{\rightharpoonup} (c^\eta_i, u^\eta, T^\eta) \;\; \text{in} \;\;  L^\infty(0,\mathcal T_{\eta,m};H^m),
    \\
    &(c_i^{(n)}, u^{(n)}, T^{(n)}) \;{\rightharpoonup} \;(c^\eta_i, u^\eta, T^\eta) \;\; \text{in} \;\;  L^2(0,\mathcal T_{\eta,m};H^{m+1}).
    \end{split}
\end{equation}
As $c_i^{(n)}(x,t)\geq 0$ and $T^{(n)}(x,t)\geq T^*$ for all $(x,t)\in \TT^2\times [0,\mathcal T_{\eta,m}]$, it follows that $c_i(x,t)\geq 0$ and $T(x,t)\geq T^*$ for a.e. $(x,t)\in\mathbb T^2\times [0,\mathcal T_{\eta,m}]$. Since $u^{(n)}$ has zero mean, it follows that $u^\eta$ has zero mean.

In addition, since the solutions $\{\phi^{(n)}\}$ to \eqref{eqn:npb-full-mo-iteration-2} are uniformly bounded in $L^\infty(0,\mathcal T_{\eta,m};H^{m+2})$, we infer the existence of a limit $\phi$ such that $\phi^{(n)}\overset{\ast}{\rightharpoonup} \phi^\eta$ in $L^\infty(0,\mathcal T_{\eta,m};H^{m+2})$. Define $\rho^\eta = \sum_{i=1}^N e z_i c^\eta_i$ and $\widetilde{\rho^\eta}=-\varepsilon\Delta \phi^\eta$. Then by uniqueness of limits we have $\rho^\eta=\widetilde{\rho^\eta}$, and therefore $ -\varepsilon\Delta \phi^\eta= \rho^\eta =\sum_{i=1}^N ez_i c_i^\eta$.

As $\{c_i^{(n+1)},u^{(n+1)},T^{(n+1)}\}$ is a sequence of solutions to the iterative system \eqref{sys:npb-full-mo-iteration},
for any time $\mathcal T\in(0,\mathcal T_{\eta,m}]$ and any test functions $\psi,\varphi \in C^\infty([0,\mathcal T]\times \TT^2)$ such that $\psi(\mathcal T) = \varphi(\mathcal T)=0$ and $\nabla\cdot\varphi=0$, we take the space-time inner product of equations of \eqref{eqn:npb-full-mo-iteration-1}, \eqref{eqn:npb-full-mo-iteration-3}, and \eqref{eqn:npb-full-mo-iteration-5} with $\psi$, $\varphi$, and $\psi$ respectively over $[0,\mathcal T]\times \mathbb T^2$, and we obtain, via integration by parts, the following weak formulations: 
\begin{subequations}
    \begin{align}
        &\int_0^\mathcal T\langle\partial_t c_i^{(n+1)}, \psi\rangle - \langle\mathcal J_\eta   u^{(n+1)}\cdot \nabla \psi, c_i^{(n+1)} \rangle  \nonumber
        \\
        & +D_i\left\langle\nabla c_i^{(n+1)} + \frac{ez_i}{k_B T^{(n+1)}} c_i^{(n+1)} \nabla \j \phi^{(n)} + \chi_i^{(n)}c_i^{(n+1)}\log c_i^{(n)}\j\nabla \log  T^{(n+1)},\nabla\psi \right\rangle dt =0, \label{appendix:1}
        \\
        &\int_0^\mathcal T\langle\partial_t u^{(n+1)},\varphi\rangle -\langle \mathcal J_\eta u^{(n)}\cdot \nabla \varphi,  u^{(n+1)}\rangle + \langle \nu\na u^{(n+1)} ,\na\varphi\rangle \nonumber
        \\
        &\hspace{1cm}-\langle g\alpha_T(T^{(n+1)}-T_r) ,\varphi_2\rangle + \left\langle \mathcal J_\eta   \left( \rho^{(n)} \nabla \j \phi^{(n)} \right) , \varphi\right\rangle dt =0, \label{appendix:2}
        \\
        &\int_0^\mathcal T\langle\partial_t T^{(n+1)},\psi\rangle - \langle \mathcal J_\eta   u^{(n)} \cdot \nabla \psi, T^{(n+1)}\rangle + \kappa \langle\na T^{(n+1)} ,\na\psi\rangle dt = 0. \label{appendix:3}
    \end{align}
\end{subequations}

Now let's take $m$ large enough, for example $m=5$, so that $\{c_i^{(n)}, u^{(n)}, T^{(n)}\}$ converge strongly in $C([0,\mathcal T_{\eta,5}]; H^{4})$ to $\{c^\eta_i,u^\eta,T^\eta\}$. With such a convergence in a strong space $H^4$ as well as the strong regularities of $\{c_i^{(n)}, u^{(n)}, T^{(n)}\}$ and $\{c^\eta_i,u^\eta,T^\eta\}$, the convergence of both linear terms and nonlinear terms to the corresponding limits follows easily. Here we only discuss the most challenging term, which involves a cutoff function and a logarithm. To be more specific, we will show that
\begin{align}\label{appendix:desired-conv}
    \int_0^\mathcal T \langle \chi_i^{(n)}c_i^{(n+1)}\log c_i^{(n)}\j\nabla \log  T^{(n+1)},\nabla\psi \rangle dt \to \int_0^\mathcal T \langle \chi^\eta(c_i^\eta)c^\eta_i\log c^\eta_i\j\nabla \log  T^\eta,\nabla\psi \rangle dt.
\end{align}
The difference of the integrands can be bounded by
\begin{align*}
    &|\langle \chi_i^{(n)}c_i^{(n+1)}\log c_i^{(n)}\j\nabla \log  T^{(n+1)},\nabla\psi \rangle  - \langle \chi^\eta(c_i^\eta)c^\eta_i\log c^\eta_i\j\nabla \log  T^\eta,\nabla\psi \rangle|
    \\
    \leq &\left|\left\langle \left(\chi^\eta(c_i^{(n)})\log c_i^{(n)}-\chi^\eta(c^\eta_i)\log c^\eta_i\right) c_i^{(n+1)}\j\nabla \log  T^{(n+1)}, \nabla\psi\right\rangle\right|
    \\
    &+\left|\left\langle  \chi^\eta(c^\eta_i)\log c^\eta_i (c_i^{(n+1)} - c^\eta_i) \j\nabla \log  T^{(n+1)},\nabla\psi  \right\rangle\right|
    \\
    &+ \left|\left\langle  \chi^\eta(c^\eta_i)\log c^\eta_i   (\j\nabla \log  T^{(n+1)} - \j\nabla \log T^\eta) c^\eta_i,\nabla\psi  \right\rangle\right| := A_1 + A_2 + A_3,
\end{align*}
where we recall that $\chi_i^{(n)} = \chi^\eta(c_i^{(n)})$. We first consider $A_1$. From the strong convergence in \eqref{appendix:convergence}, we know that there exists $N>0$ such that for all $n\geq N$ we have 
\[
\|c_i^{(n)} - c^\eta_i\|_{C([0,\mathcal T_{\eta,5}];L^\infty)} \leq \|c_i^{(n)} - c^\eta_i\|_{C([0,\mathcal T_{\eta,5}];H^2)} < \frac{\eta}{2}.
\]
For $(x,t)\in\mathbb T^2 \times [0,\mathcal T_{\eta,5}]$, if $c^\eta_i(x,t) \leq \eta$ and $c_i^{(n)}(x,t)\leq \eta$, then $ \chi^\eta(c_i^{(n)})\log c_i^{(n)}-\chi^\eta(c^\eta_i)\log c^\eta_i = 0$. If $c^\eta_i(x,t) \leq \eta$ and $c_i^{(n)}(x,t)\in [\eta,\frac32 \eta]$, then 
\begin{align*}
    &|\chi^\eta(c_i^{(n)}(x,t))\log c_i^{(n)}(x,t)-\chi^\eta(c^\eta_i(x,t))\log c^\eta_i(x,t)|
    \\
    = &|\chi^\eta(c_i^{(n)}(x,t))\log c_i^{(n)}(x,t)| \leq C_\eta |\chi^\eta(c_i^{(n)}(x,t))|.
\end{align*}
Since $\chi^\eta$ is smooth, $\lim_{n\to\infty} c_i^{(n)}(x,t) = c^\eta_i(x,t)$, and  $c^\eta_i(x,t)\leq \eta$, we conclude that the above goes to zero as $n\to \infty$. If $c^\eta_i(x,t)\geq \eta$, then we know $c_i^{(n)}(x,t) \geq \frac\eta2$ for $n\geq N$. In this case, we have
\begin{align*}
     &|\chi^\eta(c_i^{(n)}(x,t))\log c_i^{(n)}(x,t)-\chi^\eta(c^\eta_i(x,t))\log c^\eta_i(x,t)| 
     \\
     \leq &|\chi^\eta(c_i^{(n)}(x,t)) - \chi^\eta(c^\eta_i(x,t))| |\log c_i^{(n)}(x,t)| + |\chi^\eta(c^\eta_i(x,t))| |\log c_i^{(n)}(x,t) - \log c^\eta_i(x,t)|
     \\
     \leq &C_{\eta} |\chi^\eta(c_i^{(n)}(x,t)) - \chi^\eta(c^\eta_i(x,t))| (1+\|c_i^{(n)}\|_{L^\infty(0,\mathcal T_{\eta,5};L^\infty)}) + |\log c_i^{(n)}(x,t) - \log c^\eta_i(x,t)|,
\end{align*} for any $n \ge N$. Letting $n \rightarrow \infty$, the latter converges to 0 due to the uniform bound \eqref{appendix:uni-bdd-ci} and the smoothness of the cutoff function $\chi^\eta$. Combining all these cases, one obtains 
\begin{align*}
    \lim_{n\to \infty}\|\chi^\eta(c_i^{(n)})\log c_i^{(n)}-\chi^\eta(c^\eta_i)\log c^\eta_i\|_{L^\infty(0,\mathcal T_{\eta,5};L^\infty)} = 0.
\end{align*}
Consequently, we have
\begin{align*}
    \int_0^\mathcal T A_1 dt \leq &C\|\chi^\eta(c_i^{(n)})\log c_i^{(n)}-\chi^\eta(c^\eta_i)\log c^\eta_i\|_{L^\infty(0,\mathcal T_{\eta,5};L^\infty)} \|c_i^{(n+1)}\|_{L^\infty(0,\mathcal T_{\eta,5}; L^2)} 
    \\
    &\hspace{0.5cm}\times \|\j \nabla \log T^{(n+1)}\|_{L^\infty(0,\mathcal T_{\eta,5}; L^2)} \|\nabla\psi\|_{L^\infty} \to 0
\end{align*} due to the uniform-in-$n$ boundedness of the iterative concentrations and temperature. 
For $A_2$ and $A_3$, it holds that 
\[
\|\chi^\eta(c^\eta_i)\log c^\eta_i\|_{L^\infty(0,\mathcal T_{\eta,5};L^\infty)} \leq C_\eta \|c^\eta_i\|_{L^\infty(0,\mathcal T_{\eta,5};L^\infty)}.
\]
By making use of the uniform  bounds of $c_i^{(n)}$ and $T^{(n)}$, and the regularity of $c^\eta_i$ and $T^\eta$, we deduce that $\int_0^\mathcal T A_2 dt \to 0$ and $\int_0^\mathcal T A_3 dt \to 0$. Therefore we obtain the desired convergence \eqref{appendix:desired-conv}.

In conclusion, one obtains \eqref{appendix:1}--\eqref{appendix:3} with the sequence of functions replaced by their limits. Therefore, we obtain that for each $m\geq 2$, there exists a time $\mathcal T_{\eta,m}>0$ and a local solution $(c^\eta_i, u^\eta, T^\eta)$ to the mollified system \eqref{N-NPNS-mo-system} with regularity 
\[
(c^\eta_i, u^\eta, T^\eta) \in C([0,\mathcal T_{\eta,m}]; H^{m-1}) \cap L^\infty(0,\mathcal T_{\eta,m}; H^m)\cap L^2(0,\mathcal T_{\eta,m}; H^{m+1}).
\]
We note that the existence time depends on the regularity of solutions. However, in the next appendix we will prove by induction that for any $m$, the existence time is indeed infinity.

\section{Global existence of solutions to the mollified system}
\label{sec:appendix-c}
In this appendix, we fix $\eta>0$ and extend the local solutions constructed in Appendix B. 
For every finite $\mathcal T>0$, we derive $\eta$-dependent bounds on $[0,\mathcal T]$. 
These bounds prevent finite-time blow-up and therefore imply global existence for the $\eta$-regularized system. In the following, we will show that for any $m\geq0$, any solution $(c_i^{\eta}, u^{\eta}, T^{\eta})$ obtained from Appendix~\ref{sec:appendix-b} satisfies
\[
(c_i^{\eta}, u^{\eta}, T^{\eta})\in L^\infty(0,\mathcal{T}; H^m)\cap L^2(0, \mathcal{T}; H^{m+1})
\]
for any $\mathcal T>0$. 

% Let $\mathcal{T}_{\max,m} >0$ be the maximal time of existence of solutions $(c_i^{\eta}, u^{\eta}, T^{\eta})$ with regularity 
% \[
% C([0,\mathcal{T}_{\max,m} ); H^{m-1}) \cap L^\infty(0,\mathcal{T}_{\max,m}; H^m)\cap L^2(0, \mathcal{T}_{\max,m}; H^{m+1})
% \]
% for $m\geq 2$.
Firstly, we note that solutions $c_i^\eta, u^\eta$, and $T^\eta$ are bounded in $L^\infty(0,\mathcal T; L^2)\cap L^2(0,\mathcal T; H^1)$ for arbitrary $\mathcal T>0$, a fact whose proof follows along the lines of Steps 1--3 of Section~\ref{section:appendix-2}. We also recall that $T^\eta\geq T^*>0$ and $c_i^\eta \geq 0$.
Next, we derive $H^m$ bounds for any integer $m \ge 1$ using an induction argument. 
  
    We start with the base case $m = 1$. In fact, the $H^1$ evolution of $T^\eta$ obeys
    \begin{align*}
    \frac12 \frac{d}{dt} \| T^{\eta}\|_{H^1}^2 + \kappa \|\nabla T^{\eta}\|_{H^1}^2 = &\sum\limits_{0\leq |\alpha|\leq 1}\int_{\TT^2} D^\alpha (\j u^{\eta}  T^{\eta})\cdot D^\alpha \nabla T^{\eta} dx
    \leq C_\eta \|u^\eta\|_{L^2} \|T^{\eta}\|_{H^1}\|\nabla T^\eta\|_{H^1}
    \end{align*} due to standard cancellation laws, integration by parts, and mollifier bounds.  By Young's inequality, it follows that 
    \begin{align*}
        \frac{d}{dt} \| T^{\eta}\|_{H^1}^2 + \kappa \|\nabla T^{\eta}\|_{H^1}^2  \leq C_\eta \|u^\eta\|^2_{L^2} \|T^{\eta}\|_{H^1}^2.
    \end{align*}
    Thanks to the Gr\"onwall inequality and by virtue of the bound of $u^\eta$ in $L^2$, we conclude that
    \begin{align*}
        \sup_{t\in[0,\mathcal T]}\| T^{\eta}(t)\|_{H^1}^2 + \kappa \int_0^\mathcal T \|\nabla T^{\eta}(s)\|_{H^1}^2 ds <\infty.
    \end{align*} 
    %\todo[inline]{we can do the $L^\infty$ and $L^2$ bounds here for every $m$, and then at the end mention that we can get continuous in time for every $m$.}
    Thus, the family of regularized temperatures $T^\eta$ is bounded in $L^\infty(0,\mathcal T; H^1)\cap L^2(0,\mathcal T; H^2)$. 

    As for the $H^1$ evolution of the velocities $u^{\eta}$, we have  
    \begin{equation*}
    \begin{split}
        &\frac{d}{dt}\|u^{\eta}\|_{H^1}^2 + 2\nu \|\nabla u^{\eta}\|_{H^1}^2 
     \\
     = &- \sum\limits_{0\leq |\alpha|\leq 1} \langle D^\alpha (\mathcal J_\eta  u^{\eta} \cdot \nabla u^{\eta}), D^\alpha u^{\eta}\rangle 
     + \sum\limits_{0\leq |\alpha|\leq 1}  g\alpha_T \langle D^\alpha( T^{\eta}-T_r), D^\alpha u^{\eta}_2 \rangle 
     \\
     &\quad\quad- \sum\limits_{0\leq |\alpha|\leq 1} \langle D^\alpha\mathcal J_\eta \left(\rho^{\eta}\nabla\j \phi^{\eta}\right), D^\alpha u^{\eta} \rangle
     \\
    \leq & C_\eta (\|\j u^\eta\cdot\nabla u^\eta\|_{L^2} + \|T^\eta- T_r\|_{L^2} + \|\rho^\eta \j\phi^\eta\|_{L^2} ) \|\nabla u^{\eta}\|_{H^1}
    \\
    \leq & C_\eta(1+\|u^\eta\|_{L^2} \|u^\eta\|_{H^1} + \|T^\eta\|_{L^2} + \sum_{i=1}^N \|c^\eta_i\|_{L^2}^2 ) \|\nabla u^{\eta}\|_{H^1},
    \end{split}
\end{equation*}
    where we have used Poincar\'e's inequality for $u^\eta$. Thanks to Young's inequality, we deduce that 
    \begin{align*}
        \frac{d}{dt}\|u^{\eta}\|_{H^1}^2 + \nu \|\nabla u^{\eta}\|_{H^1}^2 \leq C_\eta(1+\|u^\eta\|^2_{L^2} \|u^\eta\|^2_{H^1} + \|T^\eta\|^2_{L^2} + \sum_{i=1}^N \|c^\eta_i\|_{L^2}^4 ).
    \end{align*}
    By making use of Gr\"onwall's inequality and the global bounds derived for $\|u^\eta\|_{L^2}$, $\|T^\eta\|_{L^2}$, and $\|c^\eta_i\|_{L^2}$, we conclude that 
    $u^\eta\in L^\infty(0,\mathcal T; H^1)\cap L^2(0,\mathcal T; H^2)$.

    Now the $H^1$ evolution of the regularized concentrations $c_i^{\eta}$ satisfies
    \begin{align*}
    &\frac12 \frac{d}{dt} \|c_i^{\eta}\|_{H^1}^2 + D_i \|\nabla c_i^{\eta}\|_{H^1}^2 
    =-\sum\limits_{0\leq |\alpha|\leq 1}\int_{\TT^2} D^\alpha (\j u^{\eta} \cdot\nabla c_i^{\eta})\cdot D^\alpha c_i^{\eta} dx 
    \\
    &\quad\quad- \frac{D_i ez_i}{k_B} \sum\limits_{0\leq |\alpha|\leq 1} \int_{\TT^2}D^\alpha( \frac{1}{T^{\eta}} c_i^{\eta} \nabla \j\phi^{\eta}) \cdot D^\alpha\nabla c_i^{\eta} dx
    \\
    &\quad\quad\quad\quad-D_i\sum\limits_{0\leq |\alpha|\leq 1} \int_{\TT^2} D^\alpha(\chi_i^{\eta}(c_i^\eta)c_i^{\eta}\log c_i^{\eta}\j\nabla \log  T^{\eta})  \cdot D^\alpha\nabla c_i^{\eta}dx
    \\
    \leq &C(\|\j u^{\eta} \cdot\nabla c_i^{\eta}\|_{L^2} + \|\frac{1}{T^{\eta}} c_i^{\eta} \nabla \j\phi^{\eta}\|_{H^1} + \|\chi_i^{\eta}(c_i^\eta)c_i^{\eta}\log c_i^{\eta}\j\nabla \log  T^{\eta}\|_{H^1} )\|\nabla c_i^{\eta}\|_{H^1}
    \\
    \leq &C_\eta( \|u^\eta\|_{L^2} \|c_i^\eta\|_{H^1} + \|T^\eta\|_{H^1} \|c_i^\eta\|_{L^\infty}  + \|c_i^\eta\|_{H^1}+ \|\chi_i^{\eta}(c_i^\eta)c_i^{\eta}\log c_i^{\eta}\|_{H^1} \|T^\eta\|_{H^1} )\|\nabla c_i^{\eta}\|_{H^1}
    \\
    \leq &C_\eta( \|u^\eta\|_{L^2} \|c_i^\eta\|_{H^1} + \|T^\eta\|_{H^1} \|c_i^\eta\|_{L^\infty}  + \|c_i^\eta\|_{H^1} 
    + (1+\|c_i^\eta\|_{L^\infty}) \|c_i^\eta\|_{H^1} \|T^\eta\|_{H^1} )\|\nabla c_i^{\eta}\|_{H^1},
\end{align*}
where we used the estimate \eqref{est:iteration-duality} for the potential term $\nabla\j\phi^\eta$ and applied Proposition~\ref{prop:log} to deal with the logarithmic term. By Agmon's inequality, we have $\|c_i^\eta\|_{L^\infty} \leq C \|c_i^\eta\|_{L^2}^{\frac12} \|c_i^\eta\|_{H^2}^{\frac12}$, and thus we obtain 
\begin{align*}
    &\frac12 \frac{d}{dt} \|c_i^{\eta}\|_{H^1}^2 + D_i \|\nabla c_i^{\eta}\|_{H^1}^2 
    \\
    \leq &C_\eta( \|u^\eta\|_{L^2} \|c_i^\eta\|_{H^1} + \|T^\eta\|_{H^1} \|c_i^\eta\|_{L^2}^{\frac12} \|c_i^\eta\|_{H^2}^{\frac12}  + \|c_i^\eta\|_{H^1} + (1+\|c_i^\eta\|_{L^2}^{\frac12} \|c_i^\eta\|_{H^2}^{\frac12} )\|c_i^\eta\|_{H^1} \|T^\eta\|_{H^1} )\|\nabla c_i^{\eta}\|_{H^1}.
\end{align*}
Applying Young's inequality  yields
\begin{align*}
    &\frac{d}{dt} \|c_i^{\eta}\|_{H^1}^2 +  D_i \|\nabla c_i^{\eta}\|_{H^1}^2 
    \\
    \leq &C_\eta\Big(1+\|u^\eta\|^2_{L^2} + \|T^\eta\|^2_{H^1} +   \|T^\eta\|^4_{H^1}    \|c_i^\eta\|^2_{L^2}+  \|c_i^\eta\|_{L^2}^2(1+\|T^\eta\|^4_{H^1}  \|c_i^\eta\|^2_{H^1})\Big) \|c_i^\eta\|^2_{H^1}.
\end{align*}
Since $T^\eta,u^\eta\in L^\infty(0,\mathcal T; H^1)\cap L^2(0,\mathcal T; H^2)$ and $c_i^\eta \in L^\infty(0,\mathcal T; L^2)\cap L^2(0,\mathcal T; H^1)$, we conclude that $c_i^\eta\in L^\infty(0,\mathcal T; H^1)\cap L^2(0,\mathcal T; H^2)$ by Gr\"onwall's inequality. This finishes the proof of the case $m=1$.

Next, for any $m\geq 2$, we assume that $(T^\eta,u^\eta,c_i^\eta)\in L^\infty(0,\mathcal T; H^{m-1})\cap L^2(0,\mathcal T; H^m)$. Our goal is to show that  $(T^\eta,u^\eta,c_i^\eta)\in L^\infty(0,\mathcal T; H^{m})\cap L^2(0,\mathcal T; H^{m+1})$.
The bounds $(T^{\eta},u^{\eta})$ in $L^\infty(0,\mathcal T; H^{m})\cap L^2(0,\mathcal T; H^{m+1})$  are obtained by following verbatim the proof of Step 4 in Appendix~\ref{section:appendix-2} with the superscripts $(n)$ and $(n+1)$ replaced by $\eta$. The induction hypothesis is not needed there. 
As for the $H^m$ evolution of  $c_i^{\eta}$, it holds that 
\begin{align*}
    &\frac12 \frac{d}{dt} \|c_i^{\eta}\|_{H^m}^2 + D_i \|\nabla c_i^{\eta}\|_{H^m}^2 
    = -\sum\limits_{0\leq |\alpha|\leq m}\int_{\TT^2} D^\alpha (\j u^{\eta} \cdot\nabla c_i^{\eta})\cdot D^\alpha c_i^{\eta} dx 
    \\
    &\quad\quad- \frac{D_i ez_i}{k_B} \sum\limits_{0\leq |\alpha|\leq m} \int_{\TT^2}D^\alpha( \frac{1}{T^{\eta}} c_i^{\eta} \nabla \j\phi^{\eta}) \cdot D^\alpha\nabla c_i^{\eta} dx
    \\
    &\quad\quad\quad\quad-D_i\sum\limits_{0\leq |\alpha|\leq m} \int_{\TT^2} D^\alpha(\chi_i^{\eta}(c_i^\eta)c_i^{\eta}\log c_i^{\eta}\j\nabla \log  T^{\eta})  \cdot D^\alpha\nabla c_i^{\eta}dx.
\end{align*}
The estimates of the first two terms on the right-hand side follow verbatim the analogous ones derived in Step 4 of Appendix~\ref{section:appendix-2}, yielding  
\begin{align*}
    &\frac12 \frac{d}{dt} \|c_i^{\eta}\|_{H^m}^2 + D_i \|\nabla c_i^{\eta}\|_{H^m}^2 
    \leq C_\eta (\|c_i^\eta\|_{H^m}^2 + \|c_i^\eta\|_{H^m} \|\nabla c_i^\eta\|_{H^m})
    \\
    &\quad\quad\quad\quad-D_i\sum\limits_{0\leq |\alpha|\leq m} \int_{\TT^2} D^\alpha(\chi_i^{\eta}(c_i^\eta)c_i^{\eta}\log c_i^{\eta}\j\nabla \log  T^{\eta})  \cdot D^\alpha\nabla c_i^{\eta}dx.
\end{align*}
As for the logarithmic term, we use the fact that $H^m$ is a Banach algebra and we apply Proposition~\ref{prop:log} to get
\begin{align*}
    \left\|\chi_i^{\eta}(c_i^{\eta})c_i^{\eta}\log c_i^{\eta}\j\nabla \log  T^{\eta} \right\|_{H^m} \leq &\left\|\chi_i^{\eta}(c_i^{\eta})c_i^{\eta}\log c_i^{\eta} \right\|_{H^m} \| \j\nabla \log  T^{\eta}\|_{H^m} 
    \\
    \leq &C_\eta(\|c_i^\eta\|_{H^m} + \|c_i^\eta\|_{H^m}^2 + \|c_i^\eta\|_{H^{m-1}}^{m-1} \|c_i^\eta\|_{H^m}^2) \|T^\eta\|_{H^1}.
\end{align*} 
%where we have used the continuous embedding of $H^2$ in $L^\infty$. 
Therefore, by Young's inequality, we have
\begin{align*}
    \frac{d}{dt} \|c_i^{\eta}\|_{H^m}^2 + D_i \|\nabla c_i^{\eta}\|_{H^m}^2 \leq C(1+\|T^\eta\|_{H^1}^2)(1 +\|c_i^\eta\|_{H^{m}}^2 + \|c_i^\eta\|_{H^{m-1}}^{2m-2}\|c_i^\eta\|_{H^{m}}^2 ) \|c_i^\eta\|_{H^m}^2.
\end{align*}
By virtue of the induction hypothesis, we have $c_i^\eta\in L^\infty(0,\mathcal T; H^{m-1})\cap L^2(0,\mathcal T; H^{m})$.  Hence, by Gr\"onwall's inequality, we achieve the desired global bounds for $c_i^\eta$. 

It remains to justify the time continuity. From the equations and the bounds already obtained, one has
$
(\partial_t T^\eta,\ \partial_t u^\eta,\ \partial_t c_i^\eta)
\in L^2(0,\mathcal T;H^{m-1})
$
whenever
$
(T^\eta,u^\eta,c_i^\eta)\in L^\infty(0,\mathcal T;H^m)\cap L^2(0,\mathcal T;H^{m+1}).
$
Indeed, the diffusion terms belong to $L^2(0,\mathcal T;H^{m-1})$, while the nonlinear terms are controlled by the mollifier estimates, the lower bound $T^\eta\ge T^*$, and the product estimates used above. Hence, by the Lions-Magenes lemma,
\[
(c_i^\eta,u^\eta,T^\eta)\in C([0,\mathcal T];H^m).
\] 
Since these estimates remain finite on every finite time interval, the local solution constructed in Appendix B can be continued globally in time. This proves Proposition~\ref{ap}.

\vspace{0.5cm}

{\bf{Acknowledgments.}} E.A. was partially supported by the University Research Board (URB) of the American University of Beirut under Grant No. 104752. Q.L. was partially supported by the Simons Foundation (SFI-MPS-TSM-00013384).

\vspace{0.5cm}

{\bf{Data Availability Statement.}} The research does not have any associated data.

\vspace{0.5cm}

{\bf{Conflict of Interest.}} The authors declare that they have no conflict of interest.

\bibliographystyle{plain}
\bibliography{Reference}

\end{document}